\documentclass[a4paper,11pt]{article}

\newcommand{\N}{\mathbb{N}}
\newcommand{\Z}{\mathbb{Z}}
\newcommand{\R}{\mathbb{R}}
\newcommand{\C}{\mathbb{C}}
\newcommand{\Q}{\mathbb{Q}}

\newcommand{\bP}{\mathbb{P}}

\newcommand{\Sum}{\displaystyle\sum}

\newcommand\norm[1]{\left\lVert#1\right\rVert}

\newcommand{\information}{{
  \bigskip
  \footnotesize
    
    \textbf{Jamerson Bezerra}:
    \textsc{Faculty of Mathematics and Computer Science, Nicolaus Co\-pernicus University, ul. Chopina 12/18, 87-100 Toruń, Poland.} \par\nopagebreak
    \textit{E-mail:} \texttt{jdbezerra@mat.umk.pl}

    \textbf{Mauricio Poletti}:
    \textsc{
        Departamento de Matemática, Universidade do Ceará, Campus PICI, Bloco 914, 60455-760, Fortaleza-CE, Brasil.
    } \par\nopagebreak
    \textit{E-mail:} \texttt{mpoletti@mat.ufc.br}

}}

\def\supp{\operatorname{supp}}

\def\dim{\operatorname{dim}}
\def\id{\operatorname{Id}}

\def\max{\operatorname{max}}

\def\dist{\operatorname{dist}}

\def\interior{\operatorname{Int}}

\def\SL{SL}

\usepackage[margin=1.2in]{geometry}

\usepackage{amsmath,amsthm,amssymb,amsfonts}
\usepackage{blindtext}
\usepackage{epsfig}
\usepackage{multicol}
\usepackage{xcolor}
\usepackage{hyperref}
\hypersetup{hidelinks}
\usepackage{graphicx}
\usepackage{color}
\usepackage{psfrag}
\usepackage{multirow}

\usepackage{enumitem}
\usepackage{mathtools}
\usepackage[all]{xy}
\usepackage{xfrac}

\usepackage{caption}
\usepackage{subcaption}

\newtheorem{theorem}{Theorem}

\newtheorem{proposition}[theorem]{Proposition}
\newtheorem{example}{Example}

\newtheorem{problem}{Problem}

\newtheorem{remark}[]{Remark}

\numberwithin{theorem}{section}
\numberwithin{example}{section}
\numberwithin{remark}{section}

\begin{document}

\title{An invitation to $\SL_2(\R)$ cocycles over random dynamics}
\author{Jamerson Bezerra and Mauricio Poletti}

\maketitle

\begin{abstract}
    The purpose of these notes is to discuss the advances in the theory of Lyapunov exponents of linear $\SL_2(\R)$ cocycles over hyperbolic maps. The main focus is around results regarding the positivity of the Lyapunov exponent and the regularity of this function with respect to the underlying data.
\end{abstract}

\section{Continuous Random cocycles}
\label{sec:continuousRandomCocycles}

We start by introducing the context that will be the used throughout these notes and later, we move on to specific models.

\subsection{Continuous random cocycles and Lyapunov exponent}

Let $\Sigma = \{1,\ldots, \kappa\}^{\Z}$ be the space of infinite bilateral sequences in the symbols $\{1,\ldots, \kappa\}$ and let $\sigma:\Sigma\to\Sigma$ be the shift map given by $\sigma((x_n)_n) = (x_{n+1})_n$. Joint with the invariant measure $\mu = p^{\Z}$, $p = (p_1,\ldots, p_{\kappa})$ being a probability vector, $p_i>0$, the ergodic system $(\Sigma, \sigma, \mu)$ will be called \emph{base dynamics}.

Each continuous matrix valued map $A:\Sigma\to \SL_2(\R)$ uniquely determines a skew-product $F_A:\Sigma\times \R^2\to \Sigma\times \R^2$ given by $F_A(x, v) = (\sigma(x),\, A(x)\, v)$. $F_A$ is usually referred to as \emph{Linear cocycle} associated with the base dynamics $(\Sigma, \sigma, \mu)$ and the fiber action $A$.

Fixed the base dynamics, we use the term linear cocycle to refer, not only to the map $F_A$ but also to the fiber action $A$. It is also worth noticing that each map $A$ induces in a standard way a skew-product action on $\Sigma\times\bP^1$ that we still denote by $F_A$.

We also use the following convenient notation (motivated by the chain rule for derivatives):
\begin{align*}
    A^n(x) = \left\{
        \begin{array}{ll}
            A(f^{n-1}(x))\cdots A(f(x))A(x) & \text{ if } n>0 \\
            I & \text{ if } n=0 \\
            A(f^n(x))^{-1}\cdots A(f^{-2}(x))^{-1}A(f^{-1}(x))^{-1} & \text{ if } n<0. 
        \end{array}
    \right.
\end{align*}

The Lyapunov exponent associated with the map $A$ may be defined as the exponential growth rate of the quantities $\norm{A^n(x)}$. More specifically, by Kingman's sub-additive ergodic theory, for $\mu$-a.e. $x\in \Sigma$, the limit
\begin{align*}
    L(A) := \lim_{n\to\pm\infty}\, \frac{1}{n}\log\norm{A^n(x)},
\end{align*}
exists and is constant. The number $L(A)$ is called the \emph{Lyapunov exponent} of the linear cocycle $A$. This quantity measures the amount of hyperbolicity produced by the fiber action $A$ along the orbits of $\mu$ typical points in the basis and so it can be used as a measure of the chaoticity of the system expressed by the skew-product $F_A$. For that reason, the following two problems becomes central in the theory.

\begin{enumerate}
    \item \textbf{Positivity problem:} How big is the set of $A$ with positive Lyapunov exponent? 

    \item \textbf{Regularity problem:} How regular is the map $A\mapsto L(A)$?
\end{enumerate}

\textbf{Disclaimer:} Lyapunov exponents for linear cocycles are known in a much more general framework. For instance, we could consider more general basis dynamics or even higher dimensional fiber actions. However, we restrict ourselves to the case of Bernoulli shift and $\SL_2(\R)$ cocycles for a few reasons. First, we focus on making an exposition of the results that relies on the hyperbolic structure of the base dynamics and the model that best represents such a hyperbolic structure is the full shift $(\Sigma, \sigma, \mu)$. Second, the results for higher dimensional fiber actions are more scarce and among most of those the core idea in the study can be demystified when we restrict ourselves to the two dimensional case. These restrictions are the authors choices and do not reflect any particular extra importance of the chosen exposition.

Before we proceed, we list a few basic properties of the Lyapunov exponent:
\begin{proposition}
\label{prop:basicProperties}
    Consider a continuous map $A:\Sigma\to \SL_2(\R)$. Then
    \begin{enumerate}
        \item \label{prop:basicProperties-item1} $L(A)\geq 0$;

        \item\label{prop:basicProperties-item2} $L(A) = \inf_{n\geq 1}\, \int_{\Sigma}\, \frac{1}{n}\log\norm{A^n(x)}\, d\mu(x)$;

        \item\label{prop:basicProperties-item3} The map $A\mapsto L(A)$ is upper semi-continuous. In particular, the maps $A$ with zero exponent are continuity points of the  function $L$;
        
        \item\label{prop:basicPropertiesContCocycles-Oseledets}(Oseledet's theorem, \cite{Os1968})  If $L(A)> 0$, then for $\mu$-a.e. $x\in \Sigma$, there exist $A$-invariant projective directions $\hat e^u(x), \hat e^s(x)\in \bP^1$, such that
        \begin{align*}
            \lim_{n\to \infty}\, \frac{1}{n}\log\norm{A^n(x)\, v} = \left\{
                \begin{array}{cc}
                    L(A), & v\notin \hat e^s(x)  \\
                    -L(A),  & v\in \hat e^s(x).
                \end{array}
            \right.
        \end{align*}
        \begin{align*}
            \lim_{n\to \infty}\, \frac{1}{n}\log\norm{A^{-n}(x)\, v} = \left\{
                \begin{array}{cc}
                    -L(A), & v\notin \hat e^u(x)  \\
                    L(A),  & v\in \hat e^u(x).
                \end{array}
            \right.
        \end{align*}
        Moreover, the invariant sections in $\Sigma\times \bP^1$, $x\mapsto \hat e^s(x), \hat e^u(x)\in \bP^1$ are measurable;

        \item\label{prop:basicProperties-item5} The direction $\hat e^s(x)$ ($\hat e^u(x)$) only depends on the positive (negative) coordinates of $x$.
    \end{enumerate}
\end{proposition}
\begin{remark}[Notation]
\normalfont
    We use the notation $\hat v$ to denote elements of the projective space $\bP^1$. If $\hat v$ and $v$ appears in the same expression, $v$ is (one) unitary vector in $\R^2$ in the direction determined by $\hat v$.
\end{remark}
\begin{proof}[Outline of the proof:]
    Item 1, is obtained using the fact that $A$ takes values in $\SL_2(\R)$. Item 2, is due to the characterization of the value of the Lyapunov exponent given by Kingman's sub-additive ergodic theorem. Item 3, is a direct consequence of Item 2. The item 4 is the Oseledets Theorem and item 5 comes from the proof of Oseledets theorem using that
    \begin{align*}
        \hat e^s(x) = \lim \hat s(A^n(x))
        \quad
        \text{and}
        \quad
        \hat e^u(x)
        = \lim A^n(\sigma^{-n}(x))\, \hat u(A^n(\sigma^{-n}(x))),
    \end{align*}
    where for a matrix $B\in \SL_2(\R)$ with $\norm{B} > 1$, $\hat u(B)$ and $\hat s(B)$, denotes the singular directions associated to $B$ associated respectively with the biggest and smallest singular values.
\end{proof}

\subsection{Uniform hyperbolicity}

Item 3 of the Proposition \ref{prop:basicProperties} is the start point of the study of the soft regularity of the Lyapunov exponent function. For cocycles with positive Lyapunov exponent, however, the regularity of $L$ may have a very nasty behavior. In order to properly discuss this matter it is essential to introduce the class of uniformly hyperbolic cocycles. 

We say that $A$ is \emph{uniformly hyperbolic}, if there exist continuous $A$-invariant sections $\Sigma\ni x\mapsto \hat e^u(x), \hat e^s(x)\in \bP^1$ (defined everywhere in $\Sigma$), and constants $C, \lambda > 0$, such that for every $x\in \Sigma$,
\begin{align*}
    \norm{A^n(x)|_{\hat e^s(x)}} \leq C\, e^{-\lambda n}
    \quad
    \text{and}
    \quad
    \norm{A^{-n}(x)|_{\hat e^u(x)}} \leq C\, e^{-\lambda n}.
\end{align*}
For a matrix $B\in \SL_2(\R)$ and a direction $\hat v\in \bP^1$, $B|_{\hat v}$ denotes the restriction of $B$ to the subspace determined by $\hat v$. Every uniformly hyperbolic cocycle $A$ has positive Lyapunov exponent. Indeed, using the definition we see that $L(A)\geq \lambda > 0$.
\begin{remark}
\normalfont
    The directions $\hat e^u, \hat e^s$, defined above, coincide with the directions given by Oseledets Theorem (see Item \ref{prop:basicPropertiesContCocycles-Oseledets} of Proposition \ref{prop:basicProperties}).
\end{remark}
An alternative characterization of the uniform hyperbolicity for cocycles comes from the \emph{cone field criteria}: for every $x\in \Sigma$ there exists a pair of disjoint closed projective intervals,  $ \hat J^s(x), \hat J^u(x)\subset \bP^1$, such that
\begin{align*}
    A(x)(\hat J^u(x)) \subset \interior(\hat J^u(\sigma(x)))
    \quad
    \text{and}
    \quad
    A^{-1}(x)(\hat J^s(x)) \subset \interior(\hat J^s(\sigma(x))).
\end{align*}
\begin{figure}[h!]
    \centering
    \includegraphics[width=100mm]{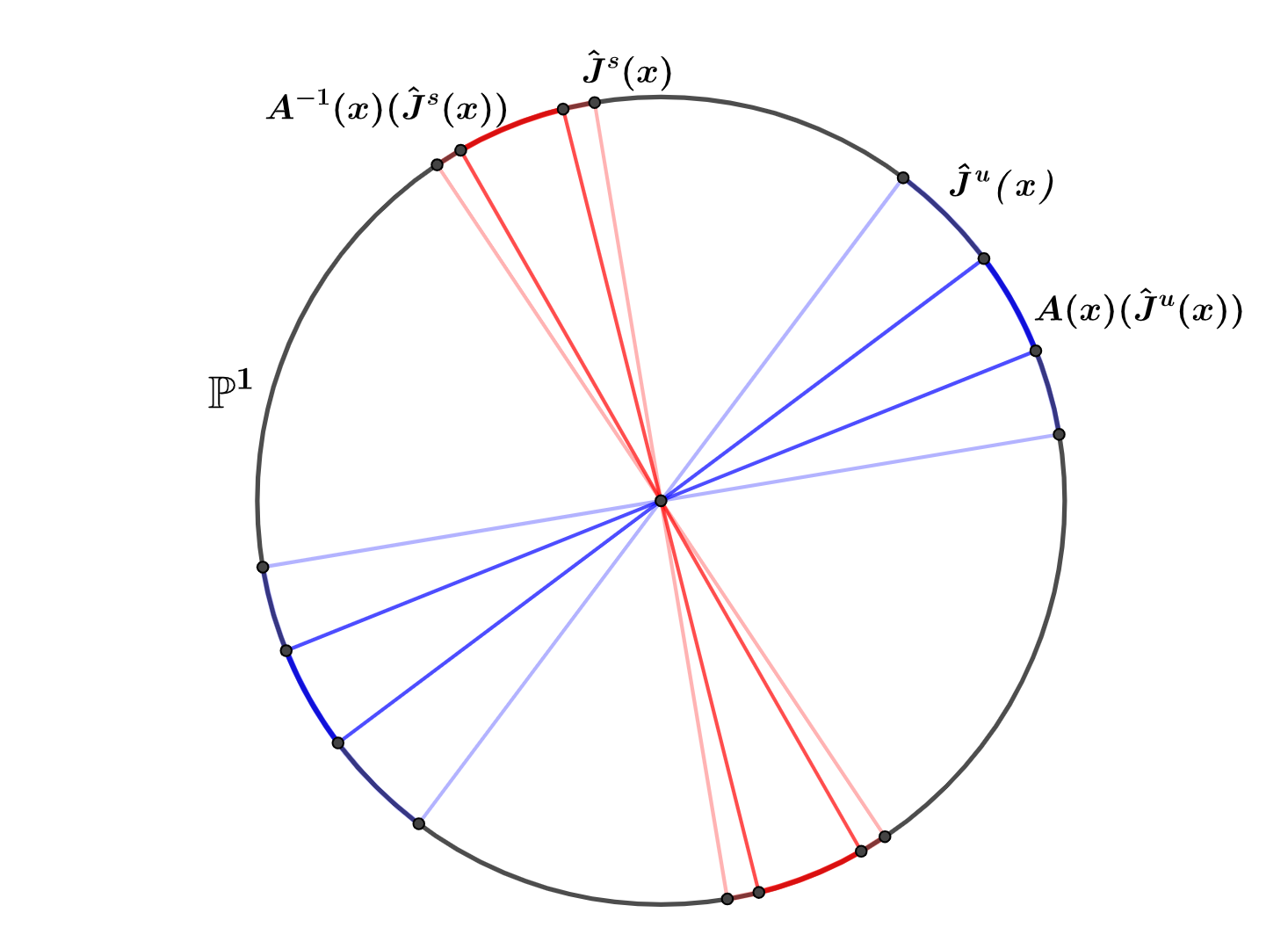}
    \caption{Projective cone field.}
    \label{pic:projectConeField}
\end{figure}
Using the cone field criteria, it is clear that uniform hyperbolicity is a stable property among the maps $A\in C^0(\Sigma,\, \SL_2(\R))$.
\begin{remark}
\label{re:criteriaUH}
\normalfont
    For $\SL_2(\R)$ cocycles a more handy criterion to determine if a cocycle is uniformly hyperbolic is available: it is enough to find constants $C, \lambda >0$ satisfying that for every $x\in \Sigma$ and $n\geq 1$,
    \begin{align*}
        \norm{A^n(x)} \geq C\, e^{\lambda n}.
    \end{align*}
    For a proof see \cite{Vi2014}.
\end{remark}
\begin{example}[Positive cocycles]
\normalfont
    Consider $A:\Sigma\to\SL_2(\R)$ continuous such that for all $x\in \Sigma$ all of the entries of the matrix $A(x)$ are strictly positive. Then, $A$ is uniformly hyperbolic. Indeed, $A$ preserves the constant cone field formed by $I$ ($\bP^1\backslash I$ is preserved by the backwards action), with $I$ being the projectivization of the first quadrant in $\R^2$ i.e., $\{(s,t)\in \R^2\colon\, s,t\geq 0 \}$.
\end{example}

\begin{example}[Triangular cocycles]
\normalfont
    Consider continuous functions $a,b,d:\Sigma\to\R$, with $\zeta := \inf_{\Sigma} |a| > 1$, and define $A:\Sigma\to \SL_2(\R)$ by $ A(x) = \begin{pmatrix}
            a(x) & b(x) \\
            0 & d(x)
        \end{pmatrix}$.
    then $\norm{A^n(x)}\geq \zeta^n$, for every $x\in \Sigma$ and every $n\geq 1$. So, using the equivalent definition mentioned in Remark \ref{re:criteriaUH}, we see that $A$ is uniformly hyperbolic.
\end{example}

\begin{example}[Schr\"odinger cocycles]
\label{ex:schrodingerCocycles}
\normalfont
    Another family of example of uniformly hyperbolic cocycles is provided by the so called \emph{Schr\"odinger cocycles} that we describe now. Let $\varphi:\Sigma\to\R$ be a continuous function. Consider the family of self-adjoint bounded operators $\{H_x:\ell^2(\Z)\to\ell^2(\Z)\}_{x\in\Sigma}$ given by
    \begin{align*}
        (H_x\, u)_n
        := (- \Delta u + \Phi_x\cdot u)_n
        = -u_{n+1} - u_{n-1} + \varphi(\sigma^n(x))\, u_n.
    \end{align*}
    The operator $H_x$, for each $x\in \Sigma$, is called the (discrete) \emph{dynamically defined Schr\"odinger operator} associated with the potential function $\varphi$ at the point $x\in \Sigma$. 
    
    The ergodicity of our base dynamics implies that the spectrum of the operator $H_x$ does not depend on the choice of point $x$ for $\mu$-a.e. $x\in \Sigma$ (see \cite{Da2017}). In a naive attempt to solve the eigenvalue-eigenvector equation, for a vector $u = (u_n)_n$ and $E\in \R$, we end up with a second order linear recursive equation that may be described in a matrix form as follows: for every $n\geq 0$,
    \begin{align}
    \label{eq:300423.1}
        \begin{pmatrix}
            \varphi(\sigma^{n-1}(x)) - E & -1 \\
            1 & 0 
        \end{pmatrix}\cdots
        \begin{pmatrix}
            \varphi(x) - E & -1 \\
            1 & 0
        \end{pmatrix}\,
        \begin{pmatrix}
            u_0\\
            u_{-1}
        \end{pmatrix}
        =\begin{pmatrix}
            u_n \\
            u_{n-1}
        \end{pmatrix}.
    \end{align}
    The \emph{Schr\"odinger cocycle} $A_E:\Sigma\to\SL_2(\R)$ associated to the \emph{energy} $E\in \R$ is defined by
    \begin{align*}
        A_E(x) = \begin{pmatrix}
            \varphi(x) - E & -1 \\
            1 & 0
        \end{pmatrix}.
    \end{align*}
    Notice that the right-hand side of the equation \eqref{eq:300423.1} described the orbit of the vector $\bar{u}=(u_0,\, u_{-1})\in \R^2$ by the cocycle $A_E$. This action produces a formal eigenvector $u = (u_n)_{n}$ associated to the eigenvalue $E\in \R$. However, any growth rate of the sequence $n\mapsto \norm{A^n_E(x)\,\bar{u}}$ is incompatible with the chance of $u\in \ell^2(\Z)$. This is the first indication of the close relationship of study of Schr\"odinger spectrum and the Lyapunov exponent of the Schr\"odinger cocycle $A_E$ as a function of the energy $E$. This close relation is confirmed by the next result.
    \begin{proposition}[R. Johnson, \cite{Jo1986}]
        The real number $E$ is not in the spectrum of the Schr\"odinger operator $\{H_x\}_{x\in \Sigma}$ if and only if $A_E$ is uniformly hyperbolic.
    \end{proposition}    
\end{example}

Using the cone field criteria, it is not hard to see that the set of uniformly hyperbolic cocycles form an open set inside of the set of continuous functions $C^0(\Sigma,\, \SL_2(\R))$ endowed with the sup norm. Back to problems 1 and 2, we see directly from the definition, that if $A$ is uniformly hyperbolic, then $L(A)>0$. Actually, in this set of cocycles we have the best regularity that we can expect. That is the content of the next result.
\begin{theorem}[D. Ruelle \cite{Ru1979}]
    The Lyapunov exponent function $L$ restricted to the set of uniformly hyperbolic cocycles is a real analytic function.
\end{theorem}
Ruelle obtained asymptotic formulas for the derivative of the Lyapunov exponent in terms of the invariant directions of the uniformly hyperbolic cocycle. Therefore, the regularity of the Lyapunov exponent is a consequence of the nice behaviour of the invariant directions with respect to the cocycle due to the contraction properties (cone contractions) provided by the uniform hyperbolicity.

\subsection{Regularity for continuous cocycles}

So far, we have seen that cocycles with zero Lyapunov exponent are continuity points of $L$ and we have also discussed the regularity among the open class of uniformly hyperbolic cocycles. Now, we move forward to analyze what happens in the complement of these two sets. In other words, what can be said about cocycles that have positive Lyapunov exponent but fail to be uniformly hyperbolic? See Example \ref{ex:BockerVianaExample}. This issue was addressed first by R. Ma\~ne and later formalized by J. Bochi.

\begin{theorem}[J. Bochi, R. Ma\~ne \cite{Bo2002}]
\label{thm:BoMa}
    If $A\in C^0(\Sigma,\, \SL_2(\R))$ is not uniformly hyperbolic and $L(A)>0$, then there exists a sequence $(A_n)_n\subset C^0(\Sigma,\, \SL_2(\R))$ converging to $A$ such that $L(A_n)=0$.
\end{theorem}
As a consequence, we have the following dichotomy: for cocycles $A$ with non-vanish Lyapunov exponent, either $A$ is uniformly hyperbolic and $L$ is analytic in a neighborhood of $A$, or $A$ can be approximated by cocycles with zero Lyapunov exponent. Notice that Theorem \ref{thm:BoMa} finishes the Problem 2 for cocycles in $C^0(\Sigma,\SL_2(\R))$.

Theorem \ref{thm:BoMa} relies on the flexibility to design local perturbations provided by the $C^0$-topology. This is essential in the proof and the strategy goes as follows: the fact that the Lyapunov exponent of $A$ is positive guarantees the existence of the measurable Oseledets sections $\hat e^s, \hat e^u$ such as in item 3 of Proposition \ref{prop:basicProperties}. However, the fact that the cocycle is not uniformly hyperbolic implies that up to small $C^0$ local perturbations the distance between $\hat e^s$ and $\hat e^u$ is arbitrarily close to zero in positive measure sets (if it is bounded away from zero, it would be possible to build cone fields).

Once we have that the distance between $\hat e^s(x)$ and $\hat e^u(x)$ is small, we may, again, perform a local $C^0$-perturbation (which can not be achieved in higher regularity) of the cocycle $A$ obtaining a new cocycle $B$ with $B(x)\cdot \hat e^s(x) = \hat e^u(x)$ for this set of points. Changing the Oseledets directions kills the exponential growth rate of the norms $\norm{B^n(x)}$ which is enough to guarantee that the Lyapunov exponent vanishes.
\begin{remark}
\normalfont
    Formally speaking, the procedure described above could be occurring over a coboundary set. In this situation, the Oseledets directions could be swapped keeping the positivity of the Lyapunov exponent. But this technical issue can always be solved in aperiodic system such as our base dynamics. See \cite[Section~9.2.2]{Vi2014} for a precise discussion of this issue.
\end{remark}
The previous strategy, however, does not work if we try to perform the same type of perturbation in higher regularity $C^{\gamma}$ with $\gamma > 0$. Here, the size of the support of a local perturbation must to be related with the amount that we are allowed to perturb. The regularity of the Lyapunov exponent function is still open in the $C^{\gamma}$-topology. We will come back to this discussion later in Section \ref{sec:holderRandomCocycles}.

\subsection{Continuous cocycles with positive Lyapunov exponent}
Before addressing the regularity of the Lyapunov exponent associated with more restrictive topologies, we highlight that a byproduct of Theorem \ref{thm:BoMa} is the fact that, in $C^0(\Sigma, \SL_2(\R))$, outside of uniformly hyperbolic cocycles, there is no open set formed only by positive Lyapunov exponent. However, we can still have plenty of points with that property. This is exactly the content of the next result.
\begin{theorem}[A. Avila, \cite{Av2011}]
\label{thm:Av2011}
    The set $\{A\in C^{\gamma}(\Sigma,\, \SL_2(\R))\colon\, L(A)>0\}$ is dense in $C^{\gamma}(\Sigma,\, \SL_2(\R))$ for any $\gamma\geq 0$.
\end{theorem}
The proof of this result is a consequence of the so called \emph{regularization expressions} for the Lyapunov exponent, the simplest of which is a (global) expression provided in \cite{AvBo2002},
\begin{align}
\label{eq:AvBoFormula}
    \int_{0}^{1} L(A\, R_{2\pi\theta})\, d\, \theta = \int_{\Sigma} \log\left(
        \frac{
            \norm{A} + \norm{A^{-1}}        
        }{2}
    \right)\, d\, \mu,
\end{align}
where $R_{2\pi\theta}$ is the rigid rotation of angle $2\pi\theta$.  In order to obtain Theorem \ref{thm:Av2011}, local regularizing expressions such as \eqref{eq:AvBoFormula} are used to guarantee the positivity of the Lyapunov exponent under small perturbations.


An adaption of the proof of Theorem \ref{thm:Av2011}  is provided in \cite{Av2011} to guarantee positivity of the Lyapunov exponent with respect of the energy $E$ for Schr\"odinger cocycles $A_{E,\varphi}$ (see Example \ref{ex:schrodingerCocycles}). More precisely, the result says that for a dense set of potentials $\varphi\in C^0(\Sigma)$, we have $L(A_{E,\varphi}) > 0$ for a a dense set of energies $E\in \R$.

\begin{remark}
\normalfont
    Up to now, all of the mentioned results do not use the hyperbolic structure of the base dynamics $(\Sigma, \sigma, \mu)$.
\end{remark}


\section{Locally constant cocycles}

To understand the properties resulting from the hyperbolic structure of the base dynamics, we focus our attention on a specific finite dimensional subspace of $C^0(\Sigma, \SL_2(\R))$ in which the base dynamics has a direct influence on the fiber action. These cocycles, known as \emph{locally constant cocycles}, are maps $A:\Sigma\to\SL_2(\R)$ which are constant in cylinders of the form $[0;\, i] := \{x\in \Sigma\colon x_0 = i\}$, $i=1,\ldots, \kappa$.

The theory of Lyapunov exponents of locally constant cocycles has been actively developed over the past 60 years.  Because many of the main ideas used in modern techniques to study the regularity and positivity of the Lyapunov exponents come from the locally constant setting, it is worth discussing it in a bit more detail.

\subsection{Random product of matrices}

In order to properly introduce this probabilistic setup, let $\{A_1,\ldots, A_{\kappa}\}$ be a finite set of matrices in $\SL_2(\R)$. Consider a sequence of independent random matrices $L_1,\ldots, L_n,\ldots$ on $\SL_2(\R)$ with a common Bernoulli distribution given by the probability measure 
$p_1\delta_{A_1}+\ldots + p_{\kappa}\delta_{A_{\kappa}}$, associated to the data $(p,A)$, where $p = (p_1,\ldots, p_{\kappa})$ and $A = (A_1,\ldots, A_{\kappa})\in \SL_2(\R)^{\kappa}$. Furstenberg and Kesten in \cite{FuKe1960}, obtained that the limit
\begin{align*}
    L(p, A) = \lim_{n\to\infty}\, \frac{1}{n}\log\norm{L_n\cdot\ldots\cdot L_1}.
\end{align*}
exists $p$-almost surely. The quantity $L(p,A)$ is called the Lyapunov exponent of the random sequence $L_1,\ldots, L_n,\ldots$. It is also called the \emph{Lyapunov exponent of the product of the random matrices} $(A_1,\ldots, A_{\kappa})$. This Lyapunov exponent only depends on the data $(p,A)$. So, the focus of the study is to analyze the properties of the real function which associate the data $(p,A)$ to the Lyapunov exponent $L(p, A)$.

Before proceeding with the comparison of this probabilistic model with the previously discussed linear cocycles, it is important to highlight that the dependence of the quantity $L(p,A)$ on the probability vector $p = (p_1,\ldots, p_{\kappa})$ is in general more regular than the dependence on the matrix vector $A = (A_1,\ldots, A_{\kappa})\in \SL_2(\R)^{\kappa}$. In fact, the next result due to Y. Peres shows that the dependence of $L(p,A)$ on the probability vector is highly regular.
\begin{theorem}[Y. Peres, \cite{Pe2006}]
\label{thm:Peres2006}
    If $p_i\neq 0$, for every $i$ and $L(p, A) >0$, then there exists a neighborhood of $p$, formed by probability vectors $q$, such that the function $q\mapsto L(q, A)$ is real analytic. 
\end{theorem}
The strategy to obtain this result is to study the so called \emph{Markov operator} associated to $(p, A)$. This is defined as the linear operator $Q_p = Q_{(p,A)}:C^0(\bP^1)\to C^0(\bP^1)$ with
\begin{align*}
    Q_p(\varphi)(\hat v)
    := \sum_{i=1}^{\kappa}\, p_i\, \varphi(A_i\, \hat v),
\end{align*}
(the notation only highlights the dependence on $p$ once $A$ is fixed). The idea of the proof of Theorem \ref{thm:Peres2006} is that, for every $\hat v\in \bP^1$, we can extend the maps $F_i:= p\mapsto Q^n_p(\log\norm{A_i\,(\,\cdot\,)})(\hat v)$ to complex variable functions $z\mapsto Q^n_z(\log\norm{A_i\, (\cdot)})(\hat v)$ which are homogeneous polynomials of degree $n$ in the variables of $z$. Using the fact that $L(p, A)>0$, its is possible to show that for points $z$ close to $p$, the sequence of functions
\begin{align*}
    z\mapsto Q^n_z(\log\norm{A_i\, (\,\cdot\,)})(\hat v),
\end{align*}
converges and hence the limit is a holomorphic function of the variable $z$. The main tool for the convergence of this sequence is a uniform contraction on average provided by the positivity of the Lyapunov exponent (and a generic assumption, see Step 1 in the proof of Theorem \ref{thm:DuKl2016} below) proved first in \cite{LePage1989}: for every sufficiently small $\theta >0$ there exists $\lambda>0$ such that for every $n$ large enough, we have
\begin{align}
\label{eq:220323.4}
    \sup_{\hat u, \hat v\in \bP^1}\int\left(
        \frac{
            d(A^n(x)\, \hat u,\, A^n(x)\, \hat u)
        }{d(\hat u, \hat v)}
    \right)^{\theta}\, d\mu(x)
    \leq e^{-\lambda\, n}.
\end{align}
To finish, it is enough to notice that, for each probability vector $q$ close to $p$ we have that the sum $\sum_i  Q^n_q(\log\norm{A_i\, (\,\cdot\,)})(\hat v)$ converges to the Lyapunov exponent $L(q, A)$.

From Theorem \ref{thm:Peres2006}, to fully understand the function $(p,A)\mapsto L(p, A)$, we may fix the probability vector $p$ and restrict our attention to the map $A= (A_1,\ldots, A_\kappa)\mapsto L(A)$. This is exactly the dynamical context that we had before, once we make the natural identification of  $A = (A_1,\ldots, A_{\kappa})\in \SL_2(\R)^{\kappa}$ with the locally constant cocycle cocycle $A:\Sigma\to \SL_2(\R)$ given by $A((x_n)_n) = A_{x_0}$. So, as described in Section \ref{sec:continuousRandomCocycles}, we are interested in the positivity and regularity of the function that associates each vector of matrices $A$ in the $3\kappa$-dimensional manifold, $\SL_2(\R)^{\kappa}$, to the real number $L(A)$.

\subsection{Stationary measures and criteria for positivity of the LE}

An essential tool to understand the Lyapunov exponent of locally constant cocycles are the so called \emph{stationary measures}. These projective probability measures carry all the asymptotic information of the fiber action $A$ and are defined as the fixed points of the dual of the Markov operator $Q_A$. More precisely, we say that a probability measure $\eta$ on $\bP^1$ is \emph{forward stationary} (or simply stationary) for the cocycle $A$ if $\eta = \sum_{i=1}^{\kappa} p_i(A_i)_*\eta$. We say that $\eta$ is \emph{backward stationary} if $\sum_{i=1}^{\kappa} p_i(A_i^{-1})_*\eta$

The next proposition provides a few properties of the stationary measures useful for later discussion.
\begin{proposition}
\label{prop:propertiesStationryMeasures}
    Let $A\in \SL_2(\R)^{\kappa}$ be a locally constant cocycle. Consider $\Sigma^+$ the set of positive sequences $(x_n)_{n\geq 0}$ and let $F^+_A:\Sigma^+\times \bP^1\to \Sigma^+\times \bP^1$ be the cocycle in $\Sigma^+\times \bP^1$ induced by $A$.
    \begin{enumerate}
        \item\label{prop:propertiesStationryMeasures-item1} The set of stationary measures $\text{Stat}(A)$ is non-empty compact and convex;
        
        \item\label{prop:propertiesStationryMeasures-item2} It holds that, $\eta \in \text{Stat}(A)$ if and only if $\mu\times\eta$ is $F_A^+$-invariant. Moreover, $\eta$ is extremal in $\text{Stat}(A)$ if and only the system $(F^+_A, \mu\times \eta)$ is ergodic;

        \item\label{prop:propertiesStationryMeasures-item3} For any extremal stationary measure $\eta$ we have that
        \begin{align*}
            \int_{\Sigma\times\bP^1} \log\norm{A(x)\, v}\, d(\mu\times\eta)(x,\hat v)\in \{L(A), - L(A)\};
        \end{align*}
        
        \item\label{prop:propertiesStationryMeasures-item4} (Furstenberg's formula): It holds that
        \begin{align}
            L(A) = \sup_{\eta\in \text{Stat}(A)}\, \int_{\Sigma\times\bP^1} \log\norm{A(x)\, v}\, d(\mu\times\eta)(x,\hat v).
        \end{align}
    \end{enumerate}
\end{proposition}
\begin{proof}[Outline of the proof of Proposition \ref{prop:propertiesStationryMeasures}]
         Item 1 is a classical argument similar to Bogoliouboff's Theorem. The first part of Item 2 is a direct computation and in the second part we use Item 1. 
         
         For item 3. Take $\eta\in \text{Stat}(A)$, extremal, and consider the function $\Phi_A:\Sigma\times\bP^1\to \R$ given by $\Phi_A(x,\hat v) = \log\norm{A(x)\, v}$. Then, by the ergodic theorem,
         \begin{align*}
             \lim_{n\to\infty}\frac{1}{n}\log\norm{A^n(x)\, v}
             &= \lim_{n\to\infty}\frac{1}{n}\Sum_{j=0}^{n-1} \Phi_A((F_A^+)^j(x,\hat v)) \\
             &=  \int_{\Sigma\times\bP^1} \log\norm{A(x)\, v}\, d(\mu\times\eta)(x,\hat v),
         \end{align*}
         for $(\mu\times\eta)$-a.e. $(x,\hat v)$. Using item \ref{prop:basicPropertiesContCocycles-Oseledets} of Proposition \ref{prop:basicProperties} and the fact that $\eta$ is extremal, we see that for $\mu$-a.e. $x\in \Sigma$, the left-hand-side above is either $L(A)$ or $-L(A)$. Item 4, is a consequence of item 3 and the ergodic decomposition for stationary measures. See \cite[Section~5]{Vi2014} for detailed proofs.
\end{proof}
\begin{remark}
\normalfont
    One may wonder why in the item 3 above we need to restrict our attention to the cocycle $F_A^+$ generated by the \emph{one-sided shift} $(\Sigma^+, \sigma)$ and not the usual cocycle $F_A$. The main reason is because the measure $\mu\times \eta$ is $F^+_A$-invariant but may not be $F_A$-invariant.
        
    To see this, take the set $[-1;\, j]\times \hat V$. Then, $F_A^{-1}([-1;\, j]\times \hat V) = [0;\, j]\times A_j^{-1}(\hat V)$ and thus, $\mu\times \eta(F_A^{-1}([-1;\, j]\times \hat V)) = p_j\cdot(A_j)_*\eta(V)$. So, $\mu\times \eta$ is  $F_A$-invariant if and only if $(A_j)_*\eta = \eta$ for every $j=1,\ldots, \kappa$. 
\end{remark}
Item \ref{prop:propertiesStationryMeasures-item3} and \ref{prop:propertiesStationryMeasures-item4} of Proposition \ref{prop:propertiesStationryMeasures} are very important tools relating the Lyapunov exponent and the stationary measures of the cocycle. This will be used many times in the rest of these notes, starting with the following example:
\begin{example}[Kifer's example, \cite{Ki1982}]
\label{ex:Kifer}
\normalfont
    Take $\alpha > 1$ and consider the matrices
    \begin{align*}
        A_1 = \begin{pmatrix}
            \alpha & 0 \\
            0 & \alpha^{-1}
        \end{pmatrix}
        \quad
        \text{and}
        \quad
        A_2 = R_{\pi/2}.
    \end{align*}
    Let $A$ be the locally constant cocycle generated by $A_1$ and $A_2$ with probabilities $p_1,p_2>0$. The measure $\eta = 1/2\delta_{\hat e_1} + 1/2\delta_{\hat e_2}$, where $e_1$ and $e_2$ denote respectively the unitary horizontal and vertical direction, is a stationary measure for $A$. Furthermore, a direct computation shows that
    \begin{align*}
        \int_{\Sigma\times \bP^1} \log\norm{A(x)\, v}\, d(\mu\times \eta)(x,\hat v) = 0.
    \end{align*} 
    Later (see Proposition \ref{prop:propertiesofTypicalCocycles}), we are going to see that $\eta$ is indeed the unique stationary measure for $A$ and so, using Furstenberg's formula (item \ref{prop:propertiesStationryMeasures-item4} of Proposition \ref{prop:propertiesStationryMeasures}), $L(A) = 0$.
\end{example}
The last example shows that sometimes we have an easy description of the stationary measures of our cocycle. In general that is not the case, even if we assume that our cocycle has very good properties (with respect to the regularity of the Lyapunov exponent) such as being uniformly hyperbolic:
\begin{example}[Bernoulli convolutions]
\label{ex:bernoulliConvolutions}
\normalfont
    In this example, we fix the probability vector of the base dynamics as $p = (1/2,1/2)$. Let $\lambda \in (0,1)$ and set
    \begin{align*}
        A_1 = \begin{pmatrix}
            \sqrt{\lambda} & \frac{1}{\sqrt{\lambda}} \\
            0 & \frac{1}{\sqrt{\lambda}}
        \end{pmatrix}
        \quad
        \text{and}
        \quad
        A_2 = \begin{pmatrix}
            \sqrt{\lambda} & -\frac{1}{\sqrt{\lambda}} \\
            0 & \frac{1}{\sqrt{\lambda}}
        \end{pmatrix}.
    \end{align*}
    Let $A_{\lambda}$ be the locally constant cocycle generated by $A_1$ and $A_2$. For the purpose of this example, we parameterize the projective space $\bP^1$ as $\R\cup \{\infty\}$. In these coordinates, we can write the action of a matrix $B = \begin{pmatrix} a & b \\ c & d \end{pmatrix}$ as $B(t) = (at + b)/(ct+d)$. Thus, in these coordinates,
    \begin{align*}
        A_1(t) = \lambda\,t + 1
        \quad
        \text{and}
        \quad
        A_2(t) = \lambda\, t - 1.
    \end{align*}
    In other words, $A_1, A_2$ are similarities (see \cite[Section~8.3]{Ma2015} for a precise definition and properties). See Picture \ref{pic:bernoulliConvolutions} for a graphical description of the orbit of $0$ by this iterated function system.
    \begin{figure}[ht]
        \centering
        \includegraphics[width=\textwidth]{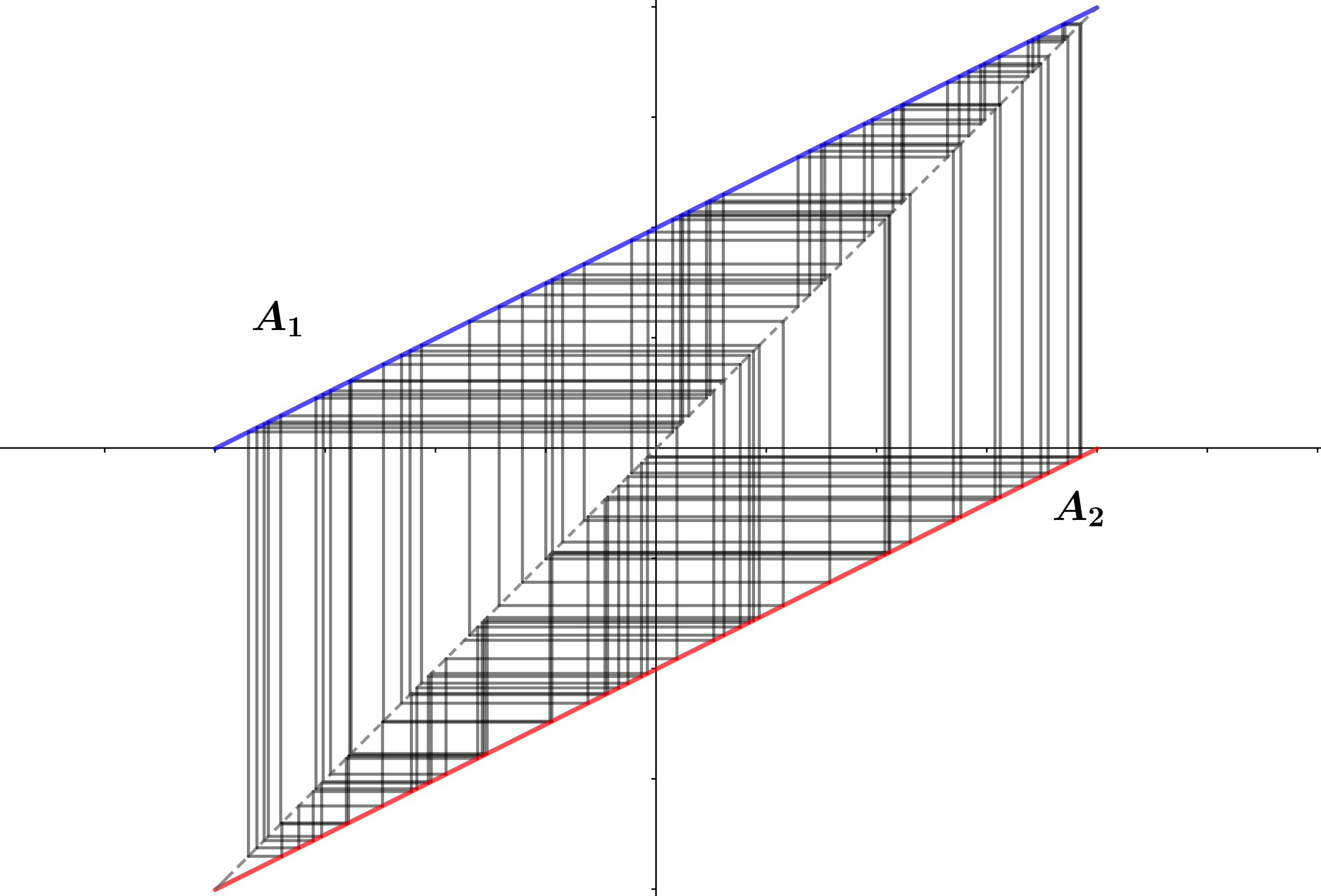}
        \caption{Random orbit of $0$ by $A_{\lambda}^j(x)(0)$ with $0\leq j \leq 100$ and $\lambda = 1/2$.}
        \label{pic:bernoulliConvolutions}
    \end{figure}
    
    Let $\eta_{\lambda}$ be the (unique) self-similar measure associated with the affine contractions $A_{\lambda} := (A_1,A_2)$. Observe that $\eta_{\lambda}$ is supported at $[-1/(1-\lambda), 1/(1-\lambda)]$. It is clear by definition of self similar measure that $\eta_{\lambda}$ is stationary for $A_{\lambda}$. Now, consider the functions $a:\Sigma\to\{+1,-1\}$ and $\tau:\Sigma\to\R$ given by
    \begin{align*}
        a(x) = \left\{
            \begin{array}{cc}
                +1, & x_0 = 1 \\
                -1, & x_0 = 2.
            \end{array}
        \right.
        \quad
        \text{and}
        \quad
        \tau(x) = \sum_{j=0}^{\infty}\, a(\sigma^j(x))\, \lambda^j.
    \end{align*}
    We claim that $\eta_{\lambda}$ coincides with the distribution of the random variable $\tau$, $\tau_*\mu$ (recall that $\mu$ is the fixed measure on $\Sigma$). Indeed, notice that, for any measurable set $B\subset \R$ we have
    \begin{align*}
        \mu(\tau^{-1}(B))
        &= \mu\left(
            [0;\, 1]\cap \tau^{-1}(B)
        \right) + \mu\left(
            [0;2]\cap \tau^{-1}(B)
        \right)\\
        &= \mu\left(
            [0;\, 1]\cap \sigma^{-1}\circ\tau^{-1}\circ (A_1)^{-1}(B)
        \right) + \mu\left(
            [0;\, 2]\cap \sigma^{-1}\circ\tau^{-1}\circ (A_2)^{-1}(B)
        \right)\\
        &= \frac{1}{2}(A_1)_*\tau_*\mu(B) + \frac{1}{2}(A_2)_*\tau_*\mu(B).
    \end{align*}
    Hence, $\tau_*\mu$ is a self similar measure supported in $[-1/(1-\lambda), 1/(1-\lambda)]$ and thus by uniqueness of the self similar measure for $(A_1,\, A_2)$,  $\tau_*\mu = \eta_{\lambda}$. In other words, the stationary measure $\eta_{\lambda}$ is the distribution measure of the Bernoulli convolutions
    \begin{align*}
        \sum_{j=0}^{\infty} \pm \lambda^j,
    \end{align*}
    where the signs $+$ and $-$ are chosen with probability $1/2$.
    
    A classical problem that goes back to Erd\"os, in \cite{Er1940}, is to determine the fractal properties of the measures $\eta_{\lambda}$. More specifically, the goal is to describe, for each value of $\lambda$, if the measure $\eta_{\lambda}$ is singular or absolutely continuous with respect to Lebesgue:
    \begin{problem}
        Describe precisely the set of $\lambda\in (0,1)$ such that $\eta_{\lambda}$ is absolutely continuous or singular with respect to Lebesgue.
    \end{problem}
    There has been a great progress regarding this problem. See \cite{Ma2015} for details about the statements and references therein.
    \begin{itemize}
        \item For $\lambda\in (0,1/2)$, $\eta_{\lambda}$ is a self similar measure supported on a Cantor set, thus singular with respect to Lebesgue;

        \item If $\lambda = 1/2$ we have that $\eta_{\lambda} = \text{Leb}|_{[-2,2]}$;

        \item For $\lambda \in (1/2, 1)$ the situation is much trickier. Only a countable set of such $\lambda$ are known where $\eta_{\lambda}$ is singular with respect to Lebesgue. Those are given by the inverse of the so called \emph{Pisot's numbers}: solutions of polynomial equations with integers coefficients such that all other complex solutions have modulus less than one.
        
        \item The breakthrough result was provided by B. Solomyak in \cite{So1995} proving that for almost every $\lambda\in (1/2, 1)$, $\eta_{\lambda}$ is absolutely continuous with respect to Lebesgue.
    \end{itemize}
    Notice that for every $\lambda$ the cocycle $A_{\lambda}$ is uniformly hyperbolic. Indeed, it is enough to observe that for every $x\in \Sigma$, $\norm{A^n(x)} \geq \lambda^{-n/2}$ and use the criteria provided in \ref{re:criteriaUH}.

    To finish the discussion, observe that Dirac measure on the fixed point $\infty$, $\delta_{\infty}$, is also a stationary measure for $A_{\lambda}$. Hence $\text{Stat}(A_{\lambda}) = [\eta_{\lambda},\, \delta_{\infty}]$. Moreover,
    \begin{align*}
        L(A_{\lambda}|_{\delta_{\infty}}) = - L(A) = \frac{1}{2}\log\lambda
        \quad
        \text{and}
        \quad
        L(A_{\lambda}|_{\eta_{\lambda}}) = L(A) = -\frac{1}{2}\log\lambda.
    \end{align*}
\end{example}

Example \ref{ex:bernoulliConvolutions} above shows that even among the uniformly hyperbolic cocycles (family in which Lyapunov exponent varies regularly) it is a hard task to understand the properties of the stationary measures. 

Nevertheless, the situation is not entirely hopeless. There is a large class of locally constant cocycles $A$ where soft analysis can be used to determine properties of the set $\text{Stat}(A)$. Below, we give a list of concepts which will be useful for our discussion. The concepts are sorted from the strongest to the weakest. We say that the vector (or cocycle) $A = (A_1,\ldots, A_{\kappa})\in \SL_2(\R)^{\kappa}$ is 
\begin{enumerate}
    \item \emph{Strongly-irreducible}, if there is no finite collection of projective directions $\hat V = \{\hat v_1, \ldots, \hat v_m\}\subset \bP^1$ such that $A_i(\hat V) = \hat V$ for every $i=1,\ldots, \kappa$;

    \item \emph{Irreducible}, if there is no projective direction $\hat v\in \bP^1$ with $A_i\, \hat v = \hat v$ for every $i=1,\ldots, \kappa$;
    \item \emph{Quasi-irreducible}, if the unique possible direction $\hat v\in \bP^1$ such that $A_i \hat v = \hat v$ for every $i=1,\ldots, \kappa$ must to satisfy $L(A|_{\hat v}) = L(A)$.    
\end{enumerate}

\begin{example}
\label{ex:irreducibiltiyConditions}
\normalfont
    Here, we mention a few examples satisfying the conditions listed above.
    \begin{enumerate}
        \item\label{ex:irreducibiltiyConditions-item1}  \emph{Strongly irreducible:} for any $\beta>1$ and $\theta\in \R\backslash\Q$,
        \begin{align*}
            A_1 = \begin{pmatrix}
                \beta & 0\\
                0 & \beta^{-1}
            \end{pmatrix}
            \quad
            \text{and}
            \quad
            A_2 = R_{2\pi\theta};
        \end{align*}
        Another interesting example is provided by the \emph{Anderson's model}. That is, the family of Schr\"odinger cocycles $A_E$ over the shift (see Example \ref{ex:schrodingerCocycles}) for any energy $E\in \R$, given by 
        \begin{align*}
            A_{1,E} = \begin{pmatrix}
                a_1 - E & -1\\
                1 & 0
            \end{pmatrix}
            \quad
            \text{and}
            \quad
            A_{2,E} = \begin{pmatrix}
                a_2 - E & -1\\
                1 & 0
            \end{pmatrix}.
        \end{align*}
        Assuming $a_1\neq a_2$, $A_E$ is strongly irreducible for every energy $E\in \R$. For an example with zero Lyapunov exponent is enough to consider the constant cocycle $A\equiv R_{2\pi\theta}$ with $\theta\in \R\backslash \Q$.
        
        \item \emph{Irreducible but not strongly irreducible:} That is provided by Kifer's example (see Example~\ref{ex:Kifer}).

        \item \emph{Quasi-irreducible but not irreducible:} consider the triangular cocycle $A \in \SL_2(\R)^{2}$ given by
        \begin{align*}
            A_1 = \begin{pmatrix}
                a_1 & b_1 \\
                0 & d_1
            \end{pmatrix}
            \quad
            \text{and}
            \quad
            A_2 = \begin{pmatrix}
                a_2 & b_2 \\
                0 & d_2
            \end{pmatrix},
        \end{align*}
        with $p_1\log|a_1| + p_2\log|a_2|>0$ and $b_j \neq 0$, for some $j=1,2$.
        The inverse $A_{\lambda}^{-1}$ of the cocycle introduced in Example \ref{ex:bernoulliConvolutions} provides a particular case of the example discussed above.

        \item \emph{Not quasi-irreducible:} Consider the diagonal cocycle $A\in \SL_2(\R)$ given by
        \begin{align*}
            A_1 = \begin{pmatrix}
                \beta & 0\\
                0 & \beta^{-1}
            \end{pmatrix}
            \quad
            \text{and}
            \quad
            A_2 = \begin{pmatrix}
                \beta^{-1} & 0 \\
                0 & \beta
            \end{pmatrix},
        \end{align*}
        with $\beta>1$. Notice that $L(A)>0$ if and only if $p_1\neq p_2$.
    \end{enumerate}
\end{example}

\begin{remark}
\normalfont
    We mentioned before the relation between the Lyapunov exponent of Schr\"odinger cocycles and the spectrum of the respective Schr\"odinger operator. In the Anderson model presented above, the spectrum of Schr\"odinger operator is completely determined and is given by the set
    \begin{align*}
       \text{Spec}(H_x) =  [a_1 - 2, a_1 + 2] \cup [a_2 - 2, a_2 + 2].
    \end{align*}
    See \cite{Da2017} for more details.
\end{remark}

Now we are ready to collect properties of the stationary measures for most $\SL_2(\R)$ cocycles which are described in some of the items of the next proposition. In what follows, for every matrix $B\in \SL_2(\R)$, we denote by $B^t$ its transpose matrix.
\begin{proposition}
\label{prop:propertiesofTypicalCocycles}
    Let $A = (A_1,\ldots, A_{\kappa})\in \SL_2(\R)^{\kappa}$ be a locally constant cocycle.
    \begin{enumerate}
        \item\label{prop:propertiesofTypicalCocycles-item1} The set of strongly irreducible cocycles is a countable intersection of open and dense subsets of $\SL_2(\R)^{\kappa}$;

        \item\label{prop:propertiesofTypicalCocycles-item2} The cocycle $A$ is strongly irreducible if and only if all the stationary measures are non atomic;

        \item\label{prop:propertiesofTypicalCocycles-item3} (Furstenberg's criterion) If the semigroup generated by $A_1,\ldots, A_{\kappa}$ is unbounded and $A$ is strongly irreducible, then $L(A) > 0$;

        \item\label{prop:propertiesofTypicalCocycles-item4} Assume that $A$ is irreducible and $L(A)>0$, then $A$ is strongly irreducible;
        
        \item\label{prop:propertiesofTypicalCocycles-item5} If $L(A)>0$ and $A$ is quasi-irreducible, then there exists a unique stationary measure.
        
        \item\label{prop:propertiesofTypicalCocycles-item6} If $L(A)>0$, then for $\mu$-a.e. $x\in \Sigma$, 
        \begin{align*}
            A^n(x)^t_*(\text{Leb})
            = \left(
                A(x)^t\cdots A(\sigma^{n-1}(x))^t
            \right)_*(\text{Leb})
            \overset{\ast}{\rightharpoonup} \delta_{e^-(x)^{\perp}}.
        \end{align*}
        Moreover, there exists at most one non-atomic stationary measure. If $\eta$ is a non-atomic stationary measure, then for $\mu$-a.e. $x\in \Sigma$,
        \begin{align*}
            A^n(x)^t_*\eta
            \overset{\ast}{\rightharpoonup} \delta_{e^-(x)^{\perp}}
            \quad
            \text{and}
            \quad
            \eta = \int_{\Sigma}\, \delta_{e^-(x)^{\perp}}\, d\mu(x).
        \end{align*}

        \item\label{prop:propertiesofTypicalCocycles-item7} If $L(A)>0$ and $A$ is not irreducible, then either $A$ (or $A^{-1}$) is quasi-irreducible or $A$ is conjugated to a diagonal cocycle.

        \item\label{prop:propertiesofTypicalCocycles-item8} If $L(A)>0$, then there exists at most two ergodic stationary measures.

        \item\label{prop:propertiesofTypicalCocycles-item9} Assume that $A$ is strongly irreducible and let $\eta$ be the stationary measure. Then $L(A)>0$ if and only if there exists $y\in \Sigma$ and $\hat v\in \bP^1$ such that
        \begin{align*}
            A^n(y)_*(\text{Leb})
            \overset{\ast}{\rightharpoonup} \delta_{\hat v}.
        \end{align*}
        
        In this case $\eta$ is unique and there exists a measurable map $\xi:\Sigma\to\bP^1$ such that for $\mu$-a.e. $x\in \Sigma$
        \begin{align*}
            A^n(x)_*\eta
            \overset{\ast}{\rightharpoonup} \delta_{\xi(x)}
            \quad
            \text{and}
            \quad
            \eta = \int_{\Sigma}\, \delta_{\xi}\, d\mu.
        \end{align*}
        Denoting by $\hat u_n(x)$ the singular direction of $A^n(x)$ associated to the largest singular value, we have that for $\mu$-a.e. $x\in \Sigma$, $\xi(x)$ is determined by
        \begin{align*}
             \xi(x) = \lim_{n\to\infty} A^n(x)\, \hat u_n(x).
        \end{align*}
        We also have that the set $\supp\eta$ is the unique minimal set for the action of the semi-group generated by $(A_1,\ldots,A_{\kappa})$.
    \end{enumerate}
\end{proposition}
\begin{proof}[Outline of the proof and references:] Item \ref{prop:propertiesofTypicalCocycles-item1} is a consequence of the fact that the set
\begin{align*}
    \mathcal{V}_n = \left\{
        A\in \SL_2(\R)^{\kappa}\colon\,
        \nexists\, \hat V\subset \bP^1,\,
        |\hat V| \leq  n\,
        \text{ with }
        A_i(\hat V) = \hat V,
        \text{ for all } 
        1\leq i \leq \kappa
    \right\},
\end{align*}
is open and dense in $\SL_2(\R)$. The union of $\mathcal{V}_n$ for all $n\geq 1$ is the set of strong irreducible cocycle. Notice in particular that the set $\mathcal{V}_1$, which is the set of irreducible cocycles, is open and dense in $\SL_2(\R)^{\kappa}$.

For item \ref{prop:propertiesofTypicalCocycles-item2}, notice that if $A$ preserves a finite set, say $\hat V$, with minimal cardinality, then we can build the atomic stationary measure $\eta := \sum_{\hat v\in \hat V}1/|\hat V|\cdot \delta_{\hat v}$. For the other implication, see \cite[Lemma~6.9]{Vi2014}. Item \ref{prop:propertiesofTypicalCocycles-item3} can be found in \cite[Theorem~6.11]{Vi2014}. 
For a proof of \ref{prop:propertiesofTypicalCocycles-item4}, see \cite[Theorem~6.1]{BoLa2012}.
Item \ref{prop:propertiesofTypicalCocycles-item5} can be found in \cite[Proposition~4.2]{DuKl2017}. For a proof of item \ref{prop:propertiesofTypicalCocycles-item6} consult \cite[Section~6.3.2]{Vi2014} and \cite[Section~1.8]{Fu2002}.

To see item \ref{prop:propertiesofTypicalCocycles-item7}, assume that $A$ is not irreducible and $L(A)>0$. Then $A$ preserves a direction $\hat v\in \bP^1$ and so $\delta_{\hat v}$ is an ergodic stationary measure for $A$. If $L(A|_{\hat v}) = L(A)$, then either $A$ is quasi-irreducible or admits another invariant direction and so in this case is conjugated to a diagonal cocycle. Therefore, we may assume that $L(A|_{\hat v}) = -L(A)$. Let $\eta$ be an ergodic stationary measure for $A$, with,
\begin{align*}
    L(A) = \int_{\Sigma\times\bP^1}\, \log\norm{A(x)\, v}\, d(\mu\times\eta)(x,\hat v),
\end{align*}
and $\eta \neq \delta_{\hat v}$. Consider the quantity $t = \max\{\eta(\hat v)\colon\,\hat v \in \bP^1\}$, and define $\hat V_t = \{\hat v\in \bP^1\colon\, \eta(\hat v) = t\}$. Using that $\eta$ is stationary is easy to see that $\hat V_t$ is preserved by $A_i$, for every $i=1,\ldots, \kappa$. Since $L(A)>0$, we see that there exists a hyperbolic matrix in the semi-group generated by $A_1,\ldots, A_{\kappa}$. This imposes a restriction on the number of elements of $\hat V_t$, i.e., $|\hat V_t|\leq 2$ (one of the invariant directions of the hyperbolic matrix is $\hat v$). Notice that $\hat v\notin \hat V_t$ since otherwise, by ergodicity of $\eta$, we would have $\eta = \delta_{\hat v}$. Then, $|\hat V_t| = 1$ and so $A$ is conjugated to a diagonal. The conclusion, therefore is that if $A$ is not conjugated to a diagonal, then $\eta$ is non-atomic. But, using item \ref{prop:propertiesofTypicalCocycles-item6} we see that $\eta$ is unique with such property. In particular, $\text{Stat}(A) = [\delta_{\hat v},\, \eta]$ and $A^{-1}$ is quasi-irreducible.

Item \ref{prop:propertiesofTypicalCocycles-item8} may be obtained from \ref{prop:propertiesofTypicalCocycles-item5}, the proof of item \ref{prop:propertiesofTypicalCocycles-item7} and the observation that for diagonal cocycles there are only two ergodic stationary measures. Item \ref{prop:propertiesofTypicalCocycles-item9} is due to Guivarc'h and Raugi \cite{GuRa1986} and can be recovered as combinations of a few results in \cite{Fu2002} (see Theorem 1.23, Lemma 1.30 and Theorem 1.34 in \cite{Fu2002} and references therein).

\end{proof}

\begin{example}
\normalfont
    It is easy to see that the set of strongly irreducible cocycles is not open. Indeed, let $\theta\in \R\backslash \Q$ and consider $p_k/q_k$ sequence of rationals converging to $\theta$. Then the sequence of cocycles $A_{1,k} = R_{2\pi p_k/q_k} $ are not strongly irreducible and it converges to the cocycle $A_1= R_{2\pi\theta}$ which is strongly irreducible.
\end{example}

Lets summarize the content of the above proposition. Assume that $L(A) > 0$. If $A$ is irreducible, then $A$ is strongly irreducible and thus there exists a unique stationary measure (see item \ref{ex:irreducibiltiyConditions-item1} of Example \ref{ex:irreducibiltiyConditions}).

In other direction, if $A$ is \emph{reducible}, i.e., there exists an invariant direction, then $A$ is conjugated to a triangular cocycle. So, we may assume that $A\in \SL_2(\R)^{\kappa}$ is given by matrices, $A_i = \begin{pmatrix} a_i & b_i \\ 0 & d_i\end{pmatrix}$, for every $i=1,\ldots, \kappa$. The measure $\delta_{\hat e_1}$ is clearly a stationary measure for $\hat A$. Assume first that $\delta_{\hat e_1}$ is the only stationary measure. So, by Furstenberg's formula, $L(A|_{\hat e_1}) = L(A)$ and in particular $A$ is quasi-irreducible. Now, if there are two ergodic stationary measures, then we have the following situations: either the other ergodic stationary measure is atomic and so $A$ is a diagonal cocycle, or there is a non-atomic ergodic stationary measure $\eta$ satisfying
\begin{align}
\label{eq:220323.1}
    \int_{\Sigma\times\bP^1}\, \log\norm{A(x)\, v}\, d(\mu\times\eta)(x,\hat v) = L(A),
\end{align}
and $L(A|_{\hat e_1}) = - L(A)$ and so $A^{-1}$ is quasi-irreducible (that is the case, for instance, of the Example \ref{ex:bernoulliConvolutions}).

\begin{example}
\normalfont
    All quasi-irreducible uniformly hyperbolic cocycles have a unique forward and a unique backward stationary measure. These measures have disjoint support on $\bP^1$ and are not necessarily singular with respect to Lebesgue (see Example \ref{ex:bernoulliConvolutions}). In \cite{AvBoYo2010} we have an explicit description of the \emph{multicones} containing the support of these measures.
\end{example}

The situation for zero Lyapunov exponent is trickier. For example, we could consider the cocycle given by the matrices $A_i = I$ for every $i=1,\ldots, \kappa$, where $I\in \SL_2(\R)$ is the identity matrix. Here, every measure on $\bP^1$ is stationary.

On the other hand, we have the situation in which the stationary measures are unique. That is the case for example when $A$ is strongly irreducible and the semi-group generated by $A_1,\ldots, A_{\kappa}$ is bounded. Then $L(A) = 0$ and there is only one stationary measure. This measure is indeed equivalent to the Lebesgue measure on $\bP^1$. 

Another interesting example of a cocycle with zero Lyapunov exponent is given by Kifer's example (Example \ref{ex:Kifer}). This cocycle is unbounded and irreducible, but it is not strongly irreducible (the set of directions $\{\hat e_1, \hat e_2\}$ is preserved by the action of $A$). The unique stationary measure is given by the measure $\eta = 1/2\cdot \delta_{\hat e_1} + 1/2\cdot \delta_{\hat e_2}$.

Cocycles with zero Lyapunov exponent present some type of rigidity. That is the content of the so called \emph{Invariance Principle:}
\begin{theorem}[F. Ledrappier, \cite{Le2006}]
\label{thm:LedrappierInvariancePrinciple}
    Let $A = (A_1,\ldots, A_{\kappa})$ be a locally constant cocycle. If $L(A) = 0$ and $\eta$ is a stationary measure for $A$, then $(A_i)_*\eta = \eta$, for every $i=1\ldots, \kappa$.
\end{theorem}
One approach to obtain such a result is through the notion of \emph{Furstenberg's entropy} of a given stationary measure $\eta$ which for a cocycle $A$ is defined by the quantity
\begin{align}
\label{eq:FurstenbergEntropy}
    h_A(\eta) := \sum_{i=1}^{k}p_i\int_{\bP^1}-\log\frac{d(A^{-1}_i)_* \eta}{d\eta}d\eta.
\end{align}
This quantity is a measurement of the the lack of invariance of the stationary measure $\eta$ by the $A_i$-action: $h_A(\eta) = 0$ if and only if $(A_i)_*\eta = \eta$ for every $i=1,\ldots, \kappa$. Theorem \ref{thm:LedrappierInvariancePrinciple} is therefore a consequence of the following inequality
\begin{align*}
    0\leq h_A(\eta) \leq 2L(A). 
\end{align*}
See \cite[Theorem~7.2]{Vi2014} for more details about this proof.

\begin{remark}
\normalfont
    The rigidity presented in Theorem \ref{thm:LedrappierInvariancePrinciple} is not a sufficient condition to guarantee zero Lyapunov exponent. Indeed, if $A=(A_1,\ldots, A_{\kappa})\in \SL_2(\R)^{\kappa}$ is a diagonal cocycle, meaning that $A_iA_j = A_jA_i$ for all $i,j=1\ldots, \kappa$, then all the stationary measures are preserved by all the matrices $A_i$, but we could have positive Lyapunov exponent (see for item 4 of Example \ref{ex:irreducibiltiyConditions}). 
\end{remark}

We can now say a bit more about the locally constant cocycles with zero Lyapunov exponent.
\begin{proposition}
\label{prop:propertiesofZeroLEofLCC}
    Let $A\in \SL_2(\R)^{\kappa}$ and assume that $L(A) = 0$.
    \begin{enumerate}
        \item\label{prop:propertiesofZeroLEofLCC-item1} If $A$ is strongly irreducible, then there exists a unique stationary measure. This measure is absolutely continuous with respect to Lebesgue;

        \item\label{prop:propertiesofZeroLEofLCC-item2} If $A$ is irreducible and the semi-group generated by $A$ is unbounded, then there is a unique stationary measure $\eta = 1/2\delta_{\hat u} + 1/2\delta_{\hat v}$ for some $\hat u, \hat v\in \bP^1$. Moreover, up to conjugation, $A_i$ is either diagonal or a rotation by $\pi/2$ and both cases must to occur;

        \item\label{prop:propertiesofZeroLEofLCC-item3} If $A$ is reducible, then, up to conjugation,
        $A_i = \begin{pmatrix} a_i & b_i \\ 0 & d_i\end{pmatrix}$ for every $i=1,\ldots,\kappa$ and
        \begin{align*}
            a_1^{p_1}\cdots a_{\kappa}^{p_k} = 1.
        \end{align*}
        If $b_1=\cdots = b_{\kappa} = 0$, i.e., $A_i$ is diagonal for every $i$ and not always equal to identity, then we have two ergodic stationary measures, namely $\eta_1 = \delta_{\hat e_1}$ and $\eta_2 = \delta_{\hat e_2}$ and so $\text{Stat}(A) = [\eta_1,\, \eta_2]$. If, otherwise, there exists $b_i\neq 0$, then $\eta = \delta_{\hat e_1}$ is the unique stationary measure.
    \end{enumerate}
\end{proposition}
\begin{proof}[Outline of the proof:]
    If $A$ is strongly irreducible, by Furstenberg's criteria we have that the group $G$ generated by $A_1,\ldots, A_{\kappa}$ is contained in a compact subgroup of $\SL_2(\R)$. This implies that, up to change of the norm in $\R^2$, we may assume that the matrices in $G$ are orthogonal. In particular, Lebesgue measure on $\bP^1$ is the unique stationary measure. This finishes item \ref{prop:propertiesofZeroLEofLCC-item1}. 

    For item \ref{prop:propertiesofZeroLEofLCC-item2}, observe that the semi-group being unbounded implies that there exists a sequence of elements $B_n$ in the semi-group such that $\norm{B_n}\to \infty$. By Theorem \ref{thm:LedrappierInvariancePrinciple}, $(B_n)_*\eta=\eta$. So, up to taking a subsequence, the projective action of $B_n$ converges to a quasi projective map $P$ (see \cite{BoV04}) that has a one dimensional kernel $\hat{v}$ and one dimensional image $\hat{u}$. Thus $\eta$ must have the form $a\,\delta_{\hat{v}}+b\,\delta_{\hat{u}}$. This implies that $\{\hat{v},\hat{u}\}$ is invariant. 
   
    As $A$ is irreducible, there exists some matrix $A_i$ that exchanges $\hat{u}$ and $\hat{v}$, so $a=b=1/2$. Up to changing the canonical basis to $u$ and $v$, $A_i$ are diagonal or exchange both (rotation of $\pi/2$).  Thus, one of the matrices $A_i$ is a rotation and since the cocycle $A$ is unbounded, another matrix $A_j$ must to be hyperbolic and diagonal.
    
    
    
    For item \ref{prop:propertiesofZeroLEofLCC-item3}, the first part is a direct consequence of the existence of an invariant direction and the fact that the exponent is equal to $\sum p_i \log |a_i|$. The second part is consequence of the invariance of $\eta$ for every element of the group and the fact that for a parabolic matrix there exists a unique invariant measure.   
\end{proof}

\begin{remark}
\normalfont
    In the case that $A$ is irreducible, $A$ can have infinitely many stationary measures. That is the case, for example, for the groups generated by a single matrix $A = R_{2\pi\theta}$, with $\theta\in \Q\backslash(\Z\cup \Z/2)$.
\end{remark}

\subsection{Regularity for locally constant}
We have seen a few results with criteria for the positivity of the Lyapunov exponent (Problem 1) in the context of locally constant cocycles. Now, we turn our attention to the problem of regularity of the function $L$. 

We start with the simplest question: is the function that associates each $A\in \SL_2(\R)^{\kappa}$ its Lyapunov exponent $L(A)$ continuous?

The approach to handle this question goes as follows: let $(A_n)_n\subset \SL_2(\R)^{\kappa}$ be a sequence of locally constant cocycles converging to $A\in \SL_2(\R)^{\kappa}$. Since the points with zero Lyapunov exponents are continuity points of $L$, we may assume that $L(A)>0$.

For each $k\geq 1$, consider an (ergodic) stationary measure $\eta_k$ for $A_k$ satisfying that
\begin{align}
\label{eq:220323.2}
    \int_{\Sigma\times\bP^1}\, \log\norm{A_k(x)\, v}\, d(\mu\times\eta_k)(x,\hat v) = L(A_k).
\end{align}
Up to taking a subsequence, the limit of $\eta_k$ in the weak$^*$ topology exists and is a stationary measure $\eta$ for the limit cocycle $A$ (not necessarily ergodic). The problem is then reduced to prove that
\begin{align}
\label{eq:220323.3}
    \int_{\Sigma\times\bP^1}\, \log\norm{A(x)\, v}\, d(\mu\times\eta)(x,\hat v) = L(A).
\end{align}
Using the above strategy and the uniqueness of the stationary measure when $A$ is quasi-irreducible we have the following result.
\begin{theorem}[Furstenberg, Kifer, \cite{FuKi1983}]
    If $A$ (or $A^{-1}$) is quasi-irreducible, then $A$ is a continuity point of $L$.
\end{theorem}
We proceed with the analysis of continuity of the Lyapunov exponent of cocycles $A$ such that neither $A$ or $A^{-1}$ are quasi-irreducible. Then, up to a change of coordinates, we may assume that $A$ is diagonal, i.e., $A_i = \begin{pmatrix} a_i & 0\\ 0 & d_i\end{pmatrix}$ for every $i=1,\ldots, \kappa$. Hence, in this case, $\text{Stat}(A) = [\delta_{\hat e_1},\, \delta_{\hat e_2}]$ with, say, $L(A|_{\hat e_1}) = L(A) = - L(A|_{\hat e_2})$. Write
\begin{align}
\label{eq:100423.1}
    \eta = q_1\delta_{\hat e_1} + q_2\delta_{\hat e_2}.
\end{align}
If $q_2 = 0$, then equality \eqref{eq:220323.3} is satisfied. So, we just need to deal with the case $q_2\neq 0$. The fact that $L(A|_{\hat e_2}) = -L(A) < 0$ implies that $\hat e_2$ is a $p$-\emph{expanding} fixed point for the projective action of $A$, i.e, it expands on average around $\hat e_2$. Furthermore, this behavior passes to nearby cocycles (although $\hat e_2$ will not be fixed by them) $A_k$, for $k$ large. This indicates that $\eta_k$ concentrates mass around $\hat e_2$, so we expect that $\eta_k$ has an atom close to $\hat e_2$. That is exactly the case, and the formalization is provided by the so called \emph{energy argument}: if $\eta_k$ is non-atomic for a subsequence of $k$'s  and $\hat e_2$ is $p$-expanding, then we should have $\eta(\{\hat e_2\}) = 0$ (see \cite[Section~10.4]{Vi2014}).

So, we may assume that $\eta_k$ has an atom for every $k$. In particular, $A_k$ is not strongly irreducible (see Item \ref{prop:propertiesofTypicalCocycles-item2} of Proposition \ref{prop:propertiesofTypicalCocycles}). Let $\hat V_k$ be a finite set of projective directions which is invariant by the coordinates of $A_k$. The fact that $L(A)>0$ implies that there exists $z\in \Sigma$ and $m\in \N$ such that $A^m(z)$ is a hyperbolic matrix. Then $A^m_k(z)$ for every $k$ (sufficiently large) is also hyperbolic. This guarantees that $|\hat V_k| \leq 2$. Assume that for every $k\geq 1$, $\hat V_k = \{\hat v_k\}$. Up to taking a subsequence, $\hat v_k$ converges to $\hat v\in \{\hat e_1,\, \hat e_2\}$. 

Using that $\delta_{\hat v_k}=\eta_k\to \eta = \delta_{\hat v}$ and equation \eqref{eq:100423.1}, jointly with the fact that $q_2\neq 0$, we see that we must to have $q_1 = 0$, i.e., $\hat v = \hat e_2$. Hence $L(A|_{\hat v}) = -L(A)$. But, since $\eta_k = \delta_{\hat v_k}$ for every $k$, by equation \eqref{eq:220323.2} we have that $L(A_k|_{\hat v_k}) \geq 0$. This contradiction implies that $q_2 = 0$. The case where there are infinitely many $k$ such that $|\hat V_k| = 2$ is handled similarly.

The next result due to Bocker and Viana summarizes the above discussion.
\begin{theorem}[Bocker, M. Viana, \cite{BoVi2017}]
\label{thm:BoVi2017}
    The function $L:\SL_2(\R)^{\kappa}\to \R$ is continuous.
\end{theorem}
Surprisingly, when compared with the general Problem 2 for continuous cocycles (see Theorem \ref{thm:BoMa}), in the world of locally constant we always have continuity of the Lyapunov exponent. That finishes the soft analysis of the function $L$.

Now we discuss the hard analysis version of the Problem 2 for locally constant cocycles. In other words, we study the modulus of continuity of the Lyapunov exponent. The first result to be mentioned in this direction is due E. Le Page.
\begin{theorem}[E. Le Page, \cite{LePage1989}]
\label{thm:LePage}
    Let $I\subset \R$ be a compact interval and let $\lambda\mapsto A_{\lambda}\in \SL_2(\R)^{\kappa}$ be a $\gamma$-H\"older continuous one-parameter family of locally constant cocycles, $\gamma>0$. Assume that for every $\lambda\in I$, $A_{\lambda}$ is strongly irreducible and generates an unbounded semi-group. Then, the map
    \begin{align*}
        I\ni \lambda\mapsto L(A_{\lambda})
    \end{align*}
    is locally H\"older continuous.
\end{theorem}
\begin{remark}
\normalfont
    The H\"older exponent of the conclusion may be smaller than the $\gamma$ in the assumption.
\end{remark}
One reason for the choice of a one parameter family in the previous result is the immediate application of the theorem to guarantee continuity of the Lyapunov exponent with respect to the energy in the Anderson model (see item \ref{ex:irreducibiltiyConditions-item1} in the  Example \ref{ex:irreducibiltiyConditions} above). This result was later generalized by Duarte and Klein, where a machinery to obtain a modulus of continuity of the Lyapunov exponent was developed.
\begin{theorem}[P. Duarte, S. Klein, \cite{DuKl2016}]
\label{thm:DuKl2016}
    Let $A\in \SL_2(\R)^{\kappa}$ be a quasi-irreducible cocycle with positive Lyapunov exponent. Then, there exist $\theta>0$ and a neighborhood $\mathcal{U}(A)$ of $A$ in $\SL_2(\R)^{\kappa}$ such that $L:\mathcal{U}(A)\to \R$ is $\theta$-H\"older continuous.
\end{theorem}
Now we discuss the strategy to obtain Theorem \ref{thm:DuKl2016}. The proof contains three main steps:
\paragraph{Step 1:} Analysis of spectral properties of the Markov operator $Q_A:C^0(\bP^1)\to C^0(\bP^1)$,
\begin{align*}
    Q_A(\varphi)(\hat v) = \sum_{i=1}^{\kappa}\, p_i\, \varphi(A_i\hat v).
\end{align*}
Let us elaborate a bit more on that. We consider for each $\theta\in (0,1)$ the space of $\theta$-H\"older continuous functions $C^{\theta}(\bP^1)$ with the norm
\begin{align*}
    \norm{\varphi}_{\theta} := \norm{\varphi}_{\infty} + [\varphi]_{\theta}
    \quad
    \text{where}
    \quad
    [\varphi]_{\theta}
    := \sup_{\hat v \neq \hat u}\, \frac{|\varphi(\hat v) - \varphi(\hat u)|}{d(\hat v, \hat u)^{\theta}}.
\end{align*}
The general idea is to prove that there exists $\theta\in (0,1)$ such that the operator $Q_A$ preserves $C^{\theta}(\bP^1)$ and when restricted to this space is \emph{quasi-compact}. In other words, denoting by $\eta$ the (unique) stationary measure for $A$ and
\begin{align*}
    F_{\eta} = \left\{
        \varphi\in C^{\theta}(\bP^1)\colon\,
        \int_{\bP^1}\, \varphi\, d\eta = 0
    \right\}.
\end{align*}
There exists a number $\rho\in (0,1)$ such that
\begin{align*}
    \text{Spec}(Q_A) = \{1\}\, \cup\, \text{Spec}(Q_A|_{F_\eta})
    \quad
    \text{and}
    \quad
    |\text{Spec}(Q_A|_{F_\eta})|<\rho < 1
\end{align*}
(the eigenvalue $1$ is associated with the constant functions). That is a consequence of the fact that for $n$ sufficiently large $Q_A^n$ contracts the $\theta$-H\"older semi-norm $[\cdot]_{\theta}$. To see that, we observe that for every $\varphi\in C^{\theta}(\bP^1)$,
\begin{align*}
    [Q_A^n(\varphi)]_{\theta}
    \leq  [\varphi]_{\theta}\, \sup_{\hat u\neq\hat v}\int_{\Sigma}\left(
        \frac{
            d(A^n(x)\, \hat u,\, A^n(x)\, \hat u)
        }{d(\hat u, \hat v)}
    \right)^{\theta}\, d\mu(x).
\end{align*}
So, the contraction property will be a consequence of the (exponential) decay of the following quantities:
\begin{align}
\label{eq:220323.7}
    \sup_{\hat u\neq\hat v}\int_{\Sigma}\left(
        \frac{
            d(A^n(x)\, \hat u,\, A^n(x)\, \hat u)
        }{d(\hat u, \hat v)}
    \right)^{\theta}\, d\mu(x)
    = \sup_{\hat v\in \bP^1}\, \int_{\Sigma}\, \frac{1}{\norm{A^n(x)\, v}^{2\theta}}\, d\mu(x).
\end{align}
It is at this point that the assumptions that $L(A)>0$ and $A$ is quasi-irreducible are used. These guarantee that the limit
\begin{align}
\label{eq:230323.1}
    \lim_{n\to\infty}\, \frac{1}{n}\int_{\Sigma}\log\norm{A^n(x)\, v}\, d\mu(x) = L(A),
\end{align}
is \emph{uniform} in the unitary vector $v\in \R^2$ (compare with Oseledet's Theorem). This interesting fact comes from a more general result due to Furstenberg and Kifer in \cite{FuKi1983} which is a non-random version of Oseledet's Theorem (see also \cite{Ki2012}). That is the main tool to ensure the exponential decay in $n$ of the quantity in \eqref{eq:220323.7}.
\begin{remark}
\normalfont
    This type of decay was used in Y. Perez result \cite{Pe2006}. See the discussion just after the statement of Theorem \ref{thm:Peres2006}.
\end{remark}

\paragraph{Step 2:} Establish uniform large deviation estimates: once we have the quasi-com\-pactness operator we may use standard techniques of additive random process to prove the following type of estimate: there exist constants $\delta,\, C,\, \kappa,\, \varepsilon_0>0$ such that for every cocycle $B\in \SL_2(\R)^{\kappa}$ with $\norm{A-B}< \delta$, for every $\varepsilon\in (0,\varepsilon_0)$ and for every $n \in \N$
\begin{align*}
    \mu\left(\left\{
        x\in \Sigma\colon\,
        \left|
            \frac{1}{n}\log\norm{B^n(x)} - L(B)
        \right|>\varepsilon
    \right\}\right)\leq C\, e^{-\kappa\varepsilon^2\, n}.
\end{align*}

\paragraph{Step 3:} Combine uniform large deviation estimates with accurate analysis of the geometry of the projective action of ``very hyperbolic'' matrices $A^n(x)$. The idea here is to use a process of exclusion of sequences $x, y\in \Sigma$ such that $B^n(x)$ and $B^n(y)$ are very hyperbolic (exponentially large norm) but the product $B^n(x)\, B^n(y)$ is small. Using a tool called \emph{avalanche principle} (see \cite{DuKl2017}) combined with the uniform large deviation estimate it is possible to show that this process of exclusion of sequences only eliminates a small probability subset of $\Sigma$. This is enough to have very good control of the finitary differences $|\log\norm{B^n(x)} - \log\norm{A^n(x)}|$ for suitable scales $n$ and most of the sequences $x\in \Sigma$. This provides the desired modulus of continuity for the Lyapunov exponent function.

The next example indicates that, even though, under the assumption of $L(A)>0$ and $A$ strongly irreducible, we have H\"older regularity in a neighborhood of $A$, the optimal H\"older exponent can get arbitrarily close to zero.
\begin{example}[Halperin, Simon and Taylor, \cite{SiTa1985}]
\normalfont
    Here, we fix $\kappa = 2$. Consider real numbers $0< a_0 < a_1$ and define
    \begin{align*}
        A_0 = \begin{pmatrix}
            a_0 & -1 \\
            1 & 0
        \end{pmatrix}
        \quad
        \text{and}
        \quad
        A_1 = \begin{pmatrix}
            a_1 & -1 \\
            1 & 0
        \end{pmatrix}.
    \end{align*}
    Let $A$ be the locally constant cocycle generated from $A_0$ and $A_1$. $A$ is unbounded and strongly irreducible. So, by Furstenberg's criteria $L(A)>0$. Thus, applying Theorem \ref{thm:DuKl2016}, we conclude that there exists $\theta\in (0,1)$ such that $L$ is $\theta$-H\"older continuous in a neighborhood of $A$. However, Halperin/Simon and Taylor showed that if
    \begin{align*}
        \theta_0 > \frac{2\log 2}{
            \cosh^{-1}\left(
                1 + \frac{a_1-a_0}{2}
            \right)
        },
    \end{align*}
    then $L$ is not $\theta_0$-H\"older continuous. In particular, making $|a_1 - a_2|\to\infty$ we can build examples of cocycles such that the H\"older exponent converges to $0$.
\end{example}
The case of zero Lyapunov exponent is different. It is possible that the Lyapunov exponent is not even H\"older continuous.
\begin{example}[P. Duarte, S. Klein, M. Santos \cite{DuKlSa2018}]
\normalfont
    Take $A$ given by Kifer's example (Example \ref{ex:Kifer}) with $p = (1/2,1/2)$ and $\alpha = 2$. This is an example with zero Lyapunov exponent. Using Halperin/Simon and Taylor's strategy Duarte, Klein and Santos showed that the Lyapunov exponent is not even $\beta$-H\"older continuous at $A$ for any $\beta>0$. Actually, they proved that the best regularity that we could expect is a weak version of H\"older regularity called $\log$-H\"older regularity (see conclusion of item 2 of Theorem \ref{thm:EHYTVi2020}).
\end{example}

An improvement in understanding the behavior of the Lyapunov exponent for locally constant cocycles was provided by E. Tall and M. Viana as described in the next result. Notice that there is no irreducibility assumption of any kind.
\begin{theorem}[EHY Tall, M. Viana \cite{EHYTVi2020}]
\label{thm:EHYTVi2020}
     It holds that,
    \begin{enumerate}
        \item (Pointwise H\"older) Assume that $L(A)>0$. Then, there exist a neighborhood $\mathcal{U}(A)\subset \SL_2(\R)^{\kappa}$ of $A$, $C>0$ and $\theta>0$ such that for every $B\in \mathcal{U}(A)$,
        \begin{align*}
            \left|
                L(A) - L(B)
            \right|
            \leq C\, \norm{A - B}^{\theta}.
        \end{align*}

        \item (Pointwise log-H\"older) For every $A\in \SL_2(\R)^{\kappa}$, there exist a neighborhood $\mathcal{U}(A)\subset \SL_2(\R)^{\kappa}$ of $A$, $C>0$ and $\theta>0$ such that for every $B\in \mathcal{U}(A)$,
        \begin{align*}
            \left|
                L(A) - L(B)
            \right|
            \leq C\, \log\left(
                \frac{1}{\norm{A - B}}
            \right)^{-\theta}.
        \end{align*}
    \end{enumerate}
\end{theorem}
The proof of this result is a consequence of a careful analysis of the phenomena presented in the discussion just before Theorem \ref{thm:BoVi2017}. The techniques are based on many classical probabilistic results such as the central limit theorem and the diffusion power law.

Due to the technical level of the proof of Theorem \ref{thm:EHYTVi2020}, we do not discuss its details here. Instead, we provide a result of similar flavour which is the pointwise Lipschitz continuity of the Lyapunov exponent at a strongly irreducible cocycle $A\in \SL_2(\R)^{\kappa}$ with $L(A)=0$ (the argument can be found in \cite{DuKlSa2018} or in \cite{EHYTVi2020}). Indeed, by Furstenberg's criterion, the closure $G$ of the group generated by the coordinates of $A$ is a compact subgroup of $\SL_2(\R)$. So, there exists a norm on $\R^2$ such that the matrices on $G$ are orthogonal. Let $\norm{\cdot}'$ be the associated operator norm. Notice that for any cocycle $B\in \SL_2(\R)^{\kappa}$,
\begin{align*}
    0\leq L(B)
    = \lim_{n\to\infty}\, \frac{1}{n}\int_{\Sigma}\log\norm{B^n}'\, d\mu
    \leq \int_{\Sigma}\, \log\norm{B}'\, d\mu
    \quad
    \text{and}
    \quad
    \int_{\Sigma}\, \log\norm{A}'\, d\mu = 0.
\end{align*}
Then,
\begin{align*}
    |L(B) - L(A)|
    &= L(B)
    \leq \int_{\Sigma}\, \log\norm{B}'\, d\mu
    \leq \int_{\Sigma}\, \log\left(
        1 + \norm{B - A}'
    \right)\, d\mu\\
    &\leq \int_{\Sigma}\, \norm{B - A}'\, d\mu
    \leq C\norm{A-B}.
\end{align*}

The non-pointwise continuity was treated by Duarte and Klein as described by the next result. 
\begin{theorem}[P. Duarte, S. Klein, \cite{DuKl2020}]
\label{thm:DuKl2020}
    Assume that $L(A)>0$. Then, there exists a neighborhood $\mathcal{U}(A)\subset \SL_2(\R)$ of $A$ such that $L:\mathcal{U}(A)\to\R$ is weak-H\"older continuous. More precisely, there exist constants $C,\, \alpha,\, \beta>0$ such that for every $B_1, B_2\in \mathcal{U}(A)$ we have
    \begin{align*}
        |L(B_1) - L(B_2)|
        \leq C\,  \exp\left(
            -\alpha\left(
                \log\frac{1}{
                    \norm{B_1 - B_2}
                }
            \right)^{\beta}
        \right).
    \end{align*}
\end{theorem}
The proof of this result follows the same general strategy as described in the proof of Theorem \ref{thm:DuKl2016}. The idea is to establish some version of uniform large deviation estimates in a neighborhood of $A$. The result only provides a weak-H\"older regularity. This is a consequence of the fact that large deviation estimates obtained in Theorem \ref{thm:DuKl2020} are no longer of exponential type but only sub-exponential. These worse estimates are due the lack of uniformity of the convergence in the limit \eqref{eq:230323.1} in the diagonal case (the only case that still needs analysis).

The idea to proof Theorem \ref{thm:DuKl2020} goes as follows. If $A$ is diagonal with $L(A)>0$, then we have classical large deviation estimates (LDE) to $A$ which is uniform in a neighborhood of $A$ among the diagonal cocycles. If $B\in \SL_2(\R)^{\kappa}$ is a quasi-irreducible cocycle near $A$, then until certain scale $n_1$ the finite scale Lyapunov exponent, $\frac{1}{n_1}\log\norm{B^{n_1}}$, absorb the property from the diagonal and satisfy LDE. By the quasi-irreducibility of $B$, this cocycle also satisfy a LDE, but that could only be seen after a different scale $n_2$ with possibly $n_2\gg n_1$. From here, the main idea is to observe that this scales are not that far from each other. Indeed, $n_2 - n_1 \leq O(\log\delta^{-1})$, where $\delta$ is a distance from $B$ to the diagonal cocycles. This is enough to save the LDE between scales $n_1$ and $n_2$ (bridge argument) although the exponential type could be lost in the process.

\subsection{Sharp modulus of continuity}
\label{subsection:sharpModulusContinuity}

As we could see earlier, with the exception of uniformly hyperbolic cocycles, the Lyapunov exponent can have a rather bad  modulus of continuity. For irreducible $A\in \SL_2(\R)^{\kappa}$ with positive Lyapunov exponent (open and dense set of cocycles) we have seen (Theorem \ref{thm:DuKl2016}) that $L$ is $\theta$-H\"older continuous in a neighborhood of $A$. In this subsection, we discuss geometric obstructions that provides bounds for how large the H\"older exponent $\theta$ can be.

Let $A=(A_1,\ldots, A_{\kappa})\in \SL_2(\R)^{\kappa}$ and denote by $\langle A \rangle^+$ the semi-group generated by the matrices $A_1,\ldots, A_{\kappa}$. We say that the cocycle $A$ admits a \emph{heteroclinic tangency} if there exist matrices $B,T,D\in \langle A\rangle^+$ such that $B, D$ are hyperbolic and $T\, \hat v_+(B) = \hat v_-(D)$, where $\hat v_+(B), \hat v_-(D)$ denotes the eigen-directions associated respectively to the largest eigenvalue of $B$ and the smallest eigenvalue of $D$. In this case, we say that the triple $(B,\, T,\, D)$ is a heteroclinic tangency for the cocycle $A$. 

\begin{example}[Heteroclinic tangecies]
\normalfont
    Consider the following examples.
    \begin{enumerate}
        \item Let $A$ be the Kifer's example (Example \ref{ex:Kifer}). Then, the triple, $(A_1, A_2, A_1)$ is a tangency for $A$;

        \item Let $A_{\alpha,\beta}$ be the cocycle in the Example \ref{ex:BockerVianaExample}. Then, the triple $(A_1, A_1, A_2)$ is a heteroclinic tangency for $A_{\alpha,\beta}$.

        \item Fix $\kappa = 3$ and consider the cocycle $A\in \SL_2(\R)^{3}$ given by the matrices
        \begin{align*}
            A_1 =\begin{pmatrix}
                \beta & 0 \\
                0 & \beta^{-1}
            \end{pmatrix},
            \quad
            A_2 = R_{2\pi\theta}\, A_1\, R_{-2\pi\theta}
            \quad
            \text{and}
            \quad
            A_3=R_{\pi(4\theta + 1)/2}.
        \end{align*}
        where $\theta\in (0,1/2)$ with $\theta\neq 1/4$. Then $A$ is strongly irreducible and the triple $(A_1, A_3, A_2)$ a tangency for $A$.
    \end{enumerate}   
\end{example}

\begin{proposition}
\label{prop:propertiesTangecies}
    It holds that,
    \begin{enumerate}
        \item\label{prop:propertiesTangecies-item1} If $A$ is uniformly hyperbolic, then $A$ does not admits a heteroclinic tangency;

        \item\label{prop:propertiesTangecies-item2} If $(B,\, T,\, D)$ is a tangency for $A$, then $(B^n,\, T,\, D^m)$ is also a tangency for $A$;

        \item\label{prop:propertiesTangecies-item3} If $A\in \SL_2(\R)^{\kappa}$ is the boundary of the uniformly hyperbolic cocycles and $\langle A\rangle^+$ does not contain a parabolic matrix, then $A$ admits a heteroclinic tangency;

        \item\label{prop:propertiesTangecies-item4} Cocycles $A\in \SL_2(\R)^{\kappa}$ admitting a heteroclinic tangency  are  dense outside the set of the uniformly hyperbolic cocycles.

        \item\label{prop:propertiesTangecies-item5} If $(B,\, T,\, B)$ is a tangency for $A$, then $\langle A\rangle^+$ contains an elliptic element.
    \end{enumerate}
\end{proposition}
\begin{proof}[Outline of the proof and references:]
    Item \ref{prop:propertiesTangecies-item1} is a consequence of the multicone characterization of the uniformly hyperbolic locally constant cocycles in \cite[Theorem~2.2]{AvBoYo2010}. Item \ref{prop:propertiesTangecies-item2} is a direct consequence of the definition. Item \ref{prop:propertiesTangecies-item3} can be found in \cite[Theorem~4.1]{AvBoYo2010}. For a proof of item \ref{prop:propertiesTangecies-item4} see \cite[Section~7]{BeDu2022}. For item \ref{prop:propertiesTangecies-item5} see \cite[Remark 4.2]{AvBoYo2010}.
\end{proof}

The existence of tangencies will be the main tool to design  perturbations that cause a drastic drop in the value of the Lyapunov exponent. These perturbations are carried out within carefully chosen one parameter families  of cocycles $\{A_t\}_{t\in I}\subset \SL_2(\R)^{\kappa}$ where the Lyapunov exponent variation can be measured in terms of the so called \emph{fibered rotation number}  of the family.

More precisely, we say that a family of locally constant cocycles $\{A_t\}_{t\in I}\subset \SL_2(\R)^{\kappa}$ is (strictly) \emph{positively winding} (or \emph{monotone}) if there exists $c_0>0$ and $n_0\in \N$ such that for every $n>n_0$, $x\in \Sigma$, $\hat v\in \bP^1$ and every $t\in I$,
\begin{align*}
    \frac{d}{dt}A^n_t(x)\, \hat v
    \geq c_0 > 0 > -c_0
    \geq \frac{d}{dt}A^{-n}_t(x)\, \hat v.
\end{align*}
For such family and for any subinterval $J\subset I$, the limit
\begin{align}
\label{eq:130423.1}
    \rho(J) = \lim_{n\to\infty}\frac{1}{n\pi}\ell_J(A^n_t(x)\, \hat v),
\end{align}
exists and is constant for $\mu$-a.e. $x\in \Sigma$ and $\hat v\in \bP^1$. Here, $\ell_J(A^n_t(x)\, \hat v)$ denotes the length of the projective curve $J\ni t \mapsto A^n_t(x)\, \hat v$. The expression \eqref{eq:130423.1} above defines a measure $\rho$ called \emph{fibered rotation measure} of the family $\{A_t\}_{t\in I}\subset \SL_2(\R)^{\kappa}$ (for more details see \cite{GoKl2021}, \cite{Go2020} or \cite{BeDuCaFrKl2022} and references therein).
\begin{example}
\label{ex:SchrodingerWinding}
\normalfont
    Consider the family of cocycles $\{A_E\}_{E\in \R}\subset\SL_2(\R)^{\kappa}$ provided by the Anderson model, i.e., for each $i=1,\ldots,\kappa$, $A_{i,E} = \begin{pmatrix} a_i - E & -1\\ 1 & 0 \end{pmatrix}$. By a direct computation, this family is positively winding (with $n_0=2$). In this case, up to a normalization constant, the fibered rotation number coincides with the integrated density of states which is the distribution measure of the spectrum of the Schrodinger operators $H_x$ (see Example \ref{ex:schrodingerCocycles}).
\end{example}

In what follows, $H(\mu) := \sum_{i=1}^{\kappa}\, -p_i\log p_i$ denotes the \emph{Shannon entropy} associated to the measure $\mu$.
\begin{theorem}[\cite{BeDu2022}, \cite{BeDuCaFrKl2022}]
\label{thm:upperBoundRegularityRotationNumber}
    Let $\{A_t\}_{t\in I}\subset \SL_2(\R)^{\kappa}$ be a positive winding family of locally constant cocycles. Assume that for some $t_0\in I$,
    \begin{enumerate}
        \item $A_{t_0}$ is strongly irreducible and $\langle A\rangle^+$ is unbounded;

        \item $A_{t_0}$ admits a heteroclinic tangency.
    \end{enumerate}
    Then, $\rho$ is not $\theta$-H\"older continuous for any $\theta > H(\mu)/L(A_{t_0})$.
\end{theorem}
We stress here that, by Theorem \ref{thm:DuKl2016}, the above assumptions guarantee that $L$ is $\theta$-H\"older continuous in a neighborhood of $A_{t_0}$ in $\SL_2(\R)^{\kappa}$ for some $\theta\in (0,1)$. The above result provides a limitation for such $\theta$ when $H(\mu)/L(A_{t_0}) < 1$.
\begin{remark}
\normalfont
    In \cite{BeDu2022}, Theorem \ref{thm:upperBoundRegularityRotationNumber} is proved, but in the context of families Schr\"odinger cocycles and the integrated density of states (see Example \ref{ex:SchrodingerWinding}). After careful adaptation of the techniques in \cite{BeDu2022} the more general version of the result, Theorem \ref{thm:upperBoundRegularityRotationNumber}, was obtained in \cite{BeDuCaFrKl2022}.
\end{remark}

Before we discuss the ideas behind the proof of Theorem \ref{thm:upperBoundRegularityRotationNumber}, we mention how to build families of cocycles satisfying the assumptions of this theorem and how to use these families to have a similar result for the Lyapunov exponent function. This type of relation was first established by Thouless, in \cite{Th1972}, for Schr\"odinger cocycles in the Anderson model (see Example \ref{ex:SchrodingerWinding}). For more general affine families locally constant cocycles, we have the following result.
\begin{theorem}[Dynamical Thouless formula, \cite{BeDuCaFrKl2022}, \cite{BerDuj2012}]
\label{thm:dynamicalThoulessFormula}
    Let $A\in \SL_2(\R)^{\kappa}$. Assume that there exists a vector of matrices $(B_1,\ldots, B_{\kappa})$ such that
    \begin{enumerate}
        \item\label{thm:dynamicalThoulessFormula-item1} $\text{rank}(B_j) = 1$ and $A_j(\text{Ker}(B_j)) = \text{Im}(B_j)$;

        \item\label{thm:dynamicalThoulessFormula-item2} For every $i,j=1,\ldots, \kappa$, $\text{Im}(B_i) \neq \text{Ker}(B_j)$.
    \end{enumerate}
    Then, the affine family $\{A_t := A + tB\}_{t\in \R}\subset \SL_2(\R)$ satisfies, for all $t\in \C$,
    \begin{align}
    \label{eq:ThoulessFormula}
        L(A_t) = L(B) + \int^{\infty}_{-\infty}\log|t - s|\, d\rho(s).
    \end{align}
\end{theorem}
\begin{remark}
\normalfont
    The assumption \ref{thm:dynamicalThoulessFormula-item1} guarantees that indeed, $A_t\in \SL_2(\R)^{\kappa}$, for every $t\in \R$. Assumption \ref{thm:dynamicalThoulessFormula-item2} ensures that the family $\{A_t\}_{t\in \R}$ is positively winding and that $L(B)>-\infty$. 

    The statement of Theorem \ref{thm:dynamicalThoulessFormula} presented here, is a particular case of the result in \cite{BeDuCaFrKl2022}. The result \cite{BerDuj2012} also provides a good description of the fibered rotation number $\rho$ as the  Laplacian of the Lyapunov exponent in the sense of distributions.
\end{remark}
The basic idea explored in the proof of Theorem \ref{thm:dynamicalThoulessFormula} is the following: it is possible to recover the Lyapunov exponent of the cocycle $A_t$ using the expression
\begin{align*}
    L(A_t) =  \max_{i,j=1,2}\{
        \lim_{n\to\infty}\,\frac{1}{n}\log|\langle A^n_t(x)\, u_i,\, u_j\rangle|
    \},
\end{align*}
where $\{u_1,\, u_2\}$ is any basis of $\R^2$. Notice that for each $n\in \N$ and each $x\in \Sigma$ the function $t\mapsto \langle A^n_t(x)\, u_i,\, u_j\rangle$ is a polynomial of degree $n$ (with real roots) and leading coefficient given by $\langle B^n(x)\, u_i,\, u_j\rangle$. So, factoring this polynomial through its roots we have
\begin{align*}
    \lim_{n\to\infty}\,\frac{1}{n}\log|\langle A^n_t(x)\, u_i,\, u_j\rangle|
    = \lim_{n\to\infty}\frac{1}{n}|\langle B^n(x)\, u_i,\, u_j\rangle| + \lim_{n\to\infty}\frac{1}{n}\sum_{t^*}\log| t - t^*|.
\end{align*}
Choosing the basis $\{u_1,\, u_2\}$ appropriately such that the first term on the right-hand side above converges a.s. to $L(B)$, gives us that for any $i,j=1,2$,
\begin{align*}
    \lim_{n\to\infty}\,\frac{1}{n}\log|\langle A^n_t(x)\, u_i,\, u_j\rangle| = L(B) + \lim_{n\to\infty}\frac{1}{n}\sum_{t^*}\log| t - t^*|.
\end{align*}
Now, using the winding property we see that the second term in the right-hand side above converges to the desired integral. 

Notice that the formula in \eqref{eq:ThoulessFormula} gives a relation between the regularity of the function $t\mapsto L(A_t)$ and the regularity of the fibered rotation number $\rho$. Indeed, the integral on the right-hand side of the formula can be rewritten, using integration by parts, as the Hilbert transform of $\rho$. By a result of Goldstein and Schlag \cite[Lemma~10.3]{GoSc2001} the Hilbert transform preserves the H\"older modulus of continuity. Therefore, the following result is a consequence of theorems \ref{thm:upperBoundRegularityRotationNumber}, \ref{thm:dynamicalThoulessFormula} and item \ref{prop:propertiesTangecies-item4} of Proposition \ref{prop:propertiesTangecies}.
\begin{theorem}[J. Bezerra, P. Duarte, A. Cai, C. Freijo, S. Klein, \cite{BeDuCaFrKl2022}]
\label{thm:BeDuCaFrKl2022}
    Assume that $A\in \SL_2(\R)^{\kappa}$ is not uniformly hyperbolic and $L(A)>0$. If $\theta > H(\mu)/L(A)$, then $L$ is not $\theta$-H\"older continuous in a neighborhood of $A$ is $\SL_2(\R)^{\kappa}$.
\end{theorem}

We now make a few comments about the proof of Theorem \ref{thm:upperBoundRegularityRotationNumber}. Naively speaking, the presence of a tangency allows us to produce sequences $x\in \Sigma$,  called \emph{matchings} of size $k$, for any $k$ sufficient large, such that moving the parameter $t$ inside a small interval $I_k$ of size $ e^{-k\,(L(A)-\varepsilon)}$  centered at $t_0$ the projective curve $I_k \ni t\mapsto A^k(x)\, \hat v$  goes one full circle around $\bP^1$. In particular, $\ell_{I_k}(A^k(x)\hat v)\geq \pi$. This behavior is additive in the sense that factoring $A^{km}(x)=\prod_j A^k(\sigma^{jk}(x))$, the numbers of full circles that the block curves $I_k\ni t\mapsto A^k(\sigma^{jk}(x)) \hat v$ go around $\bP^1$ essentially add up. Hence
\begin{align*}
    \frac{1}{km\pi}\ell_{I_k}(A^{km}(x))
    \geq \frac{1}{km\pi}\sum_{j=0}^{k-1} \ell_{I_k}(A^k(\sigma^{jk}(x))\, \hat v)
    \geq \frac{1}{km\pi}\, \sum_{j=0}^{k-1}\, \mathcal{X}_{\mathcal{M}(n,J)}(\sigma^{jk}(x)),
\end{align*}
where $\mathcal{M}(k,J)\subset \Sigma$ is the set of matchings of size $k$ in the interval $J$ (definition below). Therefore, we may use Birkhoff's ergodic theorem to guarantee that
\begin{align}
\label{eq:140423.1}
    \frac{\rho(I_k)}{|I_k|^{\theta}}
    \geq \frac{1}{k\pi} \frac{\mu(\mathcal{M}(k, J))}{|I_k|^{\theta}}.
\end{align}
In particular, if the right-hand side (RHS) of the above expression explodes, we obtain that $\rho$ cannot be $\theta$-H\"older continuous. So, a precise study of the set of matchings is required in order to give a lower bound for its probability.

Formally, we say that a sequence $x\in \Sigma$ is a $\gamma$-\emph{matching} of size $k$ at $t_0$ if there exist directions $\hat v, \hat w\in \bP^1$ and a natural number $1\leq m\leq k$ such that 
\begin{enumerate}
    \item $A^k_{t_0}(x)\, \hat v = \hat w$;

    \item $e^{\gamma} \leq \norm{A^m_{t_0}(x)\, \hat v} \leq \norm{A^m_{t_0}(x)} \leq 2e^{\gamma}$;

    \item $e^{\gamma}
    \leq \norm{A^{-(k-m)}_{t_0}(\sigma^k(x))\, \hat w}
    \leq \norm{A^{-(k-m)}_{t_0}(\sigma^k(x))} \leq 2e^{\gamma}$;
\end{enumerate}
In this case, we say that the pair of matrices $(B, D)$, where $B = A^m(x)$ and $D = A^{k-m}(\sigma^m(x))$,  connects $\hat v$ and $\hat w$ and produces the $\gamma$-matching (the value used for $\gamma$ is  $\gamma=k\, (L(A)-\varepsilon)$ for some conveniently small $\varepsilon$).
\begin{figure}[ht]
    \centering
    \includegraphics[width=\textwidth]{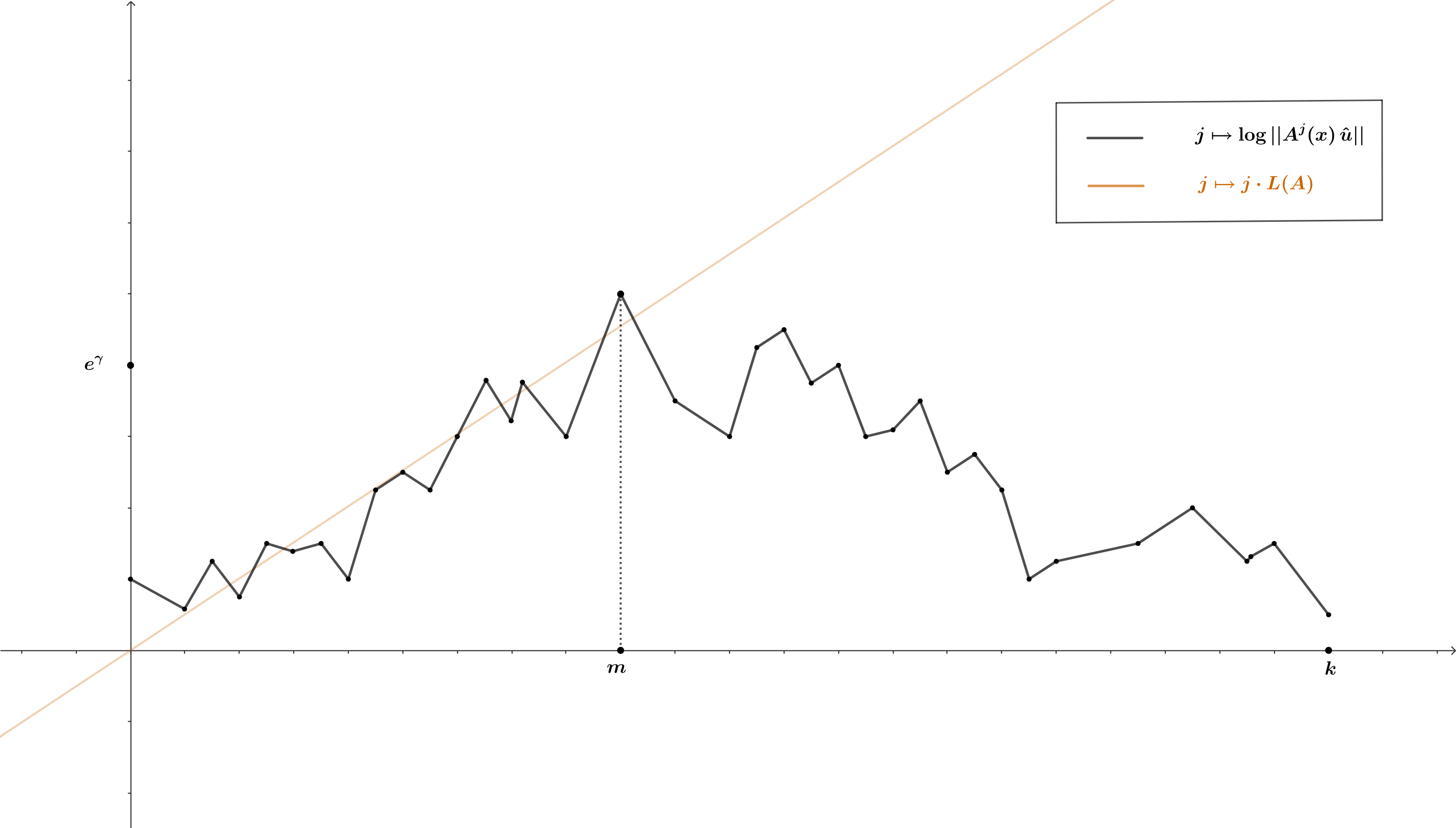}
    \caption{
        Growth of the map $j\mapsto \log\norm{A^j(x)\, \hat u}$ for a $\gamma$-matching.
    }
    \label{pic:matchingGrowth}
\end{figure}

If $(B, T, D)$ is a \emph{ balanced  tangency }  for $A_{t_0}$, i.e., both matrices $B$ and $D$ correspond to hyperbolic words  with eigenvalues $\sim e^{\pm c}$, then we can easily find many other matchings for $A_{t_0}$. Indeed, the pair $(T\,B^m,\, D^m)$ produces a $mc$-matching for every $m\geq 1$. However, this way of finding matchings only gives us a set of zero probability, which is not suitable for our purposes. 

To obtain matchings notice that for positively winding families, picking any  directions $\hat v, \hat w\in \bP^1$, the points $A_t^m(x)\, \hat v$ and $A^{-(k-m)}_t(\sigma^k(x))\, \hat w$ move (w.r.t. $t$) in \emph{opposite directions} at a speed that is  uniformly bounded away from zero. Assuming that the norms of $A_t^k(x)$ and $A^{-(k-m)}_t(\sigma^k(x))\, \hat w$ are large (which can be provided by the  uniform large deviation estimates) this process eventually gives us a matching at some moment $t^*\in \R$, i.e.,
\begin{align*}
    A_{t^*}^k(x)\, \hat v
    = A^{-(k-m)}_{t^*}(\sigma^k(x))\, \hat w.
\end{align*}
\begin{figure}[ht]
    \centering
    \includegraphics[width=\textwidth]{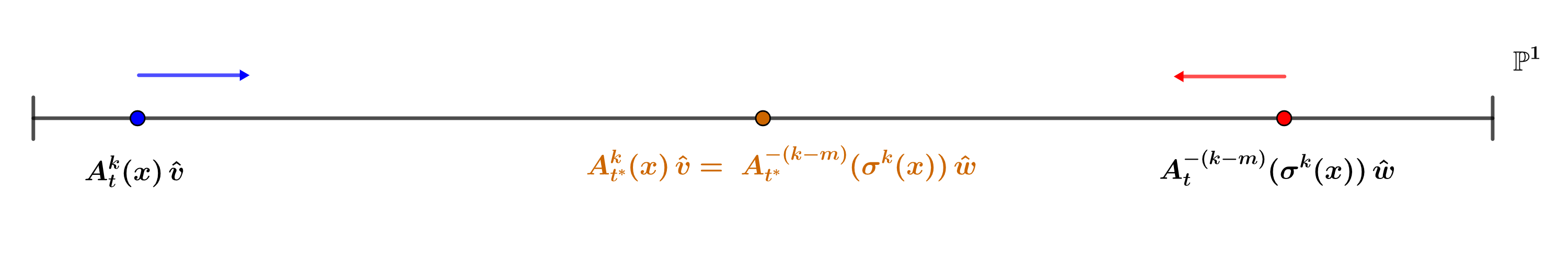}
    \caption{
        Moving the parameter $t$ to obtain matchings.
    }
    \label{pic:ideaForMathcings}
\end{figure}

To ensure that $t^\ast$ occurs in a  small neighborhood of $t_0$ the strategy is to use a tangency to force that at the initial time $t_0$ the points $A_{t_0}^k(x)\, \hat v$ and $A^{-k}_{t_0}(\sigma^k(x))\, \hat w$ are already close to each other.

Let $\tau$ be a finite word
associated to a \emph{balanced heteroclinic tangency} 
$(B,T,D)$ with $A_{t_0}^{|\tau|}(y) = D\, T\, B$
for some $y\in [0;\, \tau]$. More precisely let $|\tau|=m_B + m_D+m_T$ where $B=A_{t_0}^{m_B}(y)$,
$T=A_{t_0}^{m_T}(\sigma^{m_B} (y))$ and
$D= A_{t_0}^{m_D}(\sigma^{m_B+m_T} (y))$.
The size of the neighborhood $I$ where we will look for matchings is determined by
the \emph{Lyapunov exponent of the tangency} $\tau$  
    \begin{align*}
        L(\tau) := \frac{1}{|\tau|}\log\norm{A^{|\tau|}_{t_0}(y)}
        \approx \frac{1}{m_D}\log\norm{D}
        \approx \frac{1}{m_B}\log\norm{B}
    \end{align*}
in the sense that $|I| \sim e^{- |\tau| (  L(\tau)-\varepsilon))}$ for some conveniently and arbitrarily small $\varepsilon>0$. On the other hand, the measure of the matching occurrence event $\mathcal{M}(|\tau|,I)$  will be determined by the \emph{entropy of the tangency} $\tau$:
\begin{align*}
    H(\tau) := -\frac{1}{|\tau|}\log\mu([0;\, \tau]) ,
\end{align*}
in the sense that $\mu(\mathcal{M}(|\tau|,I))\geq e^{-|\tau|(H(\tau)-\varepsilon)}$ for some small $\varepsilon>0$.

The process to create multiple matchings goes as follows. Fix $\tau$ as before and take any points $\hat u, \hat w\in\bP^1$. Consider matrix products $B_n$ and $D_n$ associated with much longer words of size $n\gg |\tau|$. For most choices of these words we will have
\begin{enumerate}
	\item $\norm{B_n\, u} \gtrsim  e^{n\, (L(A)-\varepsilon)}$, $\norm{D_n^{-1}\, w} \gtrsim  e^{n\, (L(A)-\varepsilon)}$,
	\item $B_n\, \hat u$ is  at some distance  bounded away from the stable eigen-direction of $B$,
	\item $D_n^{-1}\, \hat w$ is   at a distance  bounded away from the unstable eigen-direction of $D$.
\end{enumerate}
    
Then the pairs $(T\, B\, B_n,\, D_n\, D)$ are almost matchings at $t_0$ because
    \begin{align*}
        d(T\, B\, B_n\, \hat u, D^{-1} \, D_n^{-1}\hat w)\lesssim e^{-|\tau|(L(\tau)-\varepsilon)}.
    \end{align*}
See picture \ref{pic:matchings} below for a graphical explanation of this process of creation of matchings.

True matchings  occur at nearby parameters $t^\ast_n$ with $|t_n^\ast-t_0|\lesssim e^{-|\tau|\,(L(\tau) -\varepsilon)}$, in the sense that the pair $(T\, B\, B_n,\, D_n\, D)$ 
will have a matching occurring at $t^\ast_n$.
\begin{figure}[ht]
    \centering
    \includegraphics[width=\textwidth]{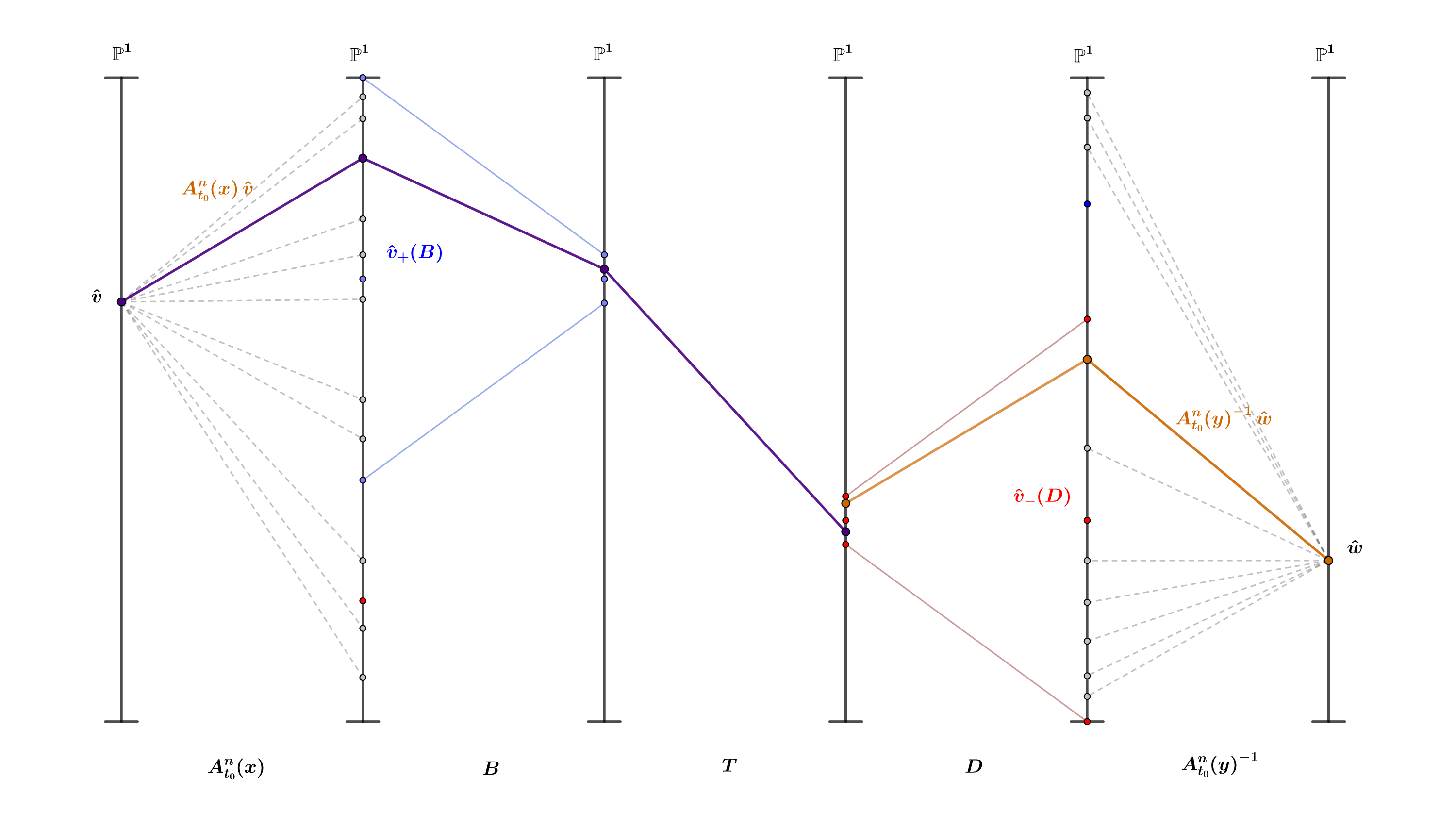}
    \caption{Construction of matchings. The vertical bars represent different copies of the projective space and between them the respective matrix action is depicted.}
    \label{pic:matchings}
\end{figure}
In particular, the right-hand side of \eqref{eq:140423.1} is bounded from below by  $e^{-|\tau|(H(\tau) -  \theta\, L(\tau)-2\varepsilon)}$.

Since we are able to produce tangencies $\tau$ such that $|\tau|\to\infty$, if we could prove that  $\theta > H(\tau)/L(\tau)$ for every such tangency, then  $\theta > (H(\tau)-\varepsilon)/(L(\tau)-\varepsilon)$  for $\varepsilon$ small enough, and we would conclude our result.

To overcome the possibility that the Lyapunov exponent and entropy of the tangency could be very different from $H(\mu)$ and  $L(A_{t_0})$ we can produce  many \emph{typical tangencies} of arbitrarily large size, i.e., tangencies $\tau$ where  $(H(\tau), L(\tau))$ is arbitrarily close to $(H(\mu),L(A_{t_0}))$,  by essentially the same process that we are using to produce matchings. This concludes the sketch of the proof of Theorem \ref{thm:upperBoundRegularityRotationNumber}.

\subsection{Regularity and dimension of the stationary measures}
The relation between the entropy and the Lyapunov exponent obtained in the bound for the regularity of $L$ in Theorem \ref{thm:BeDuCaFrKl2022} resembles Ledrappier-Young type of formulas (see \cite{LeYoI1985} and \cite{LeYoII1985}). In the case of linear cocycles this quocient is related with the dimension of the forward and backward stationary measures.

The definition of upper and lower local dimensions of a projective probability measure $\eta$ goes as follows:
\begin{align*}
    \overline{\dim}\, (\eta;\, \hat v)
    = \varlimsup_{r\to 0^+}\, \frac{\log\eta(B_r(\hat v))}{\log r}
    \quad
    \text{and}
    \quad
    \underline{\dim}\, (\eta;\,\hat v)
    = \varliminf_{r\to 0^+}\, \frac{\log\eta(B_r(\hat v))}{\log r}.
\end{align*}
We say that $\eta$ is \emph{exact dimensional} if there exists a number $\alpha \geq 0$ such that for $\eta$-a.e. every $\hat v\in \bP^1$ we have $\overline{\dim}\, (\eta;\, \hat v) = \underline{\dim}\, (\eta;\,\hat v) = \alpha$. In this case we write $\alpha=\dim \eta$.

\begin{theorem}[M. Hochman, Solomyak, \cite{HoSo2017}]
    Assume that $A$ is irreducible with positive Lyapunov exponent. Then the unique stationary measure is exact dimensional. Moreover,
    \begin{align}
    \label{eq:140423.2}
        \dim \eta = \frac{h_A(\eta)}{2\, L(A)},
    \end{align}
    where $h_A(\eta)$ is the Furstenberg entropy introduced in \eqref{eq:FurstenbergEntropy}.
\end{theorem}
Notice that the $h_A(\eta) \leq H(\mu)$. In \cite{HoSo2017} the authors proved that if the group generated by the matrices $A_1,\ldots, A_{\kappa}$ is a free group and the set of cocycle generators $\{A_1,\ldots, A_\kappa\}$ is \emph{Diophantine} (see \cite{HoSo2017}) then the formula \eqref{eq:140423.2} improves to
\begin{align*}
    \dim \eta^+ = \dim\eta^- = \frac{H(\mu)}{2\, L(A)},
\end{align*}
where $\eta^+$ and $\eta^-$ denotes respectively the forward and backward stationary measure of the cocycle $A$. So, the upper bound for the regularity of the Lyapunov exponent is in these cases is nothing more than $\dim \eta^+ + \dim \eta^-$.

\begin{problem}
    Is $\dim \eta^+ + \dim \eta^-$ always an upper bound for the regularity of the Lyapunov exponent?
\end{problem}

A related problem concerns the symmetry of the forward and backward systems.
\begin{problem}
    Is there an example of irreducible cocycle with $L(A)>0$ where $\dim \eta^+ \neq \dim \eta^-$?
\end{problem}

We could also wonder what is the lower bound for the regularity of typical cocycles. For a cocycle $A\in \SL_2(\R)^{\kappa}$ define the quantity
\begin{align*}
    \theta_A = \sup\{\theta\geq 0\colon
        A \text{ is } \theta-\text{H\"older around } A 
    \}.
\end{align*}
\begin{problem}
\label{problem:lowerBoundRegularityLCC}
    Assume $A$ irreducible with $L(A)>0$. What is the natural lower bound for $\theta_A$?
\end{problem}
As we have seen in the discussion of the proof of Theorem \ref{thm:DuKl2016}, the quantity $\theta_A$ is associated to the contraction properties of the Markov operator. This could be the path to obtain an interesting answer to problem \ref{problem:lowerBoundRegularityLCC}.

\begin{remark}
\normalfont
    The topics discussed above for locally constant cocycles, i.e., maps $A:\Sigma\to\SL_2(\R)$ that depend only on the zero-th coordinate, can be generalized to maps that depends on a finite (but fixed, say $m$) number of coordinates. Indeed such   cocycles can always be regarded  as  locally constant maps over a full  shift in $\kappa^m$ symbols.
\end{remark}

A full understanding of the Lyapunov exponent function among the locally constant is far from complete. One example of the lack of precise knowledge about this function is given by the next open problem
\begin{problem}
    Is there an example of locally constant cocycle which is a $C^r$-continuity point of the Lyapunov exponent for some $r\geq 1$?
\end{problem}
\begin{remark}
\normalfont
    We would expect that such a cocycle admits a very regular stationary measure. One example of a cocycle with $C^r$, for $r$ sufficiently large,  stationary measure was provided by J. Bourgain \cite{Bo2000}. See \cite{BeQu2018} for a better discussion on the regularity of stationary measures in the context of group actions.
\end{remark}


\section{Holder random cocycles}
\label{sec:holderRandomCocycles}
Now we go back to the initial discussion for continuous linear cocycles $A:\Sigma\to\SL_2(\R)$ possibly depending on infinitely many coordinates. To fully extract the hyperbolic properties of the basic dynamics it will be convenient to establish some notation.

For each $x = (x_n)_{n\in\Z}\in \Sigma$ we may write $x = (x^-,x^+)$, where $x^- = (x_{-n})_{n\leq0}$ and $x^+ = (x_n)_{n\geq 0}$. We define the local stable and local unstable set of $x$ respectively by
\begin{align*}
    W^s_{\text{loc}}(x) = \left\{
        y\in \Sigma\colon\,
        x^+ = y^+
    \right\}
    \quad
    \text{and}
    \quad
    W^u_{\text{loc}}(x) = \left\{
        y\in \Sigma\colon\,
        x^- = y^-
    \right\}.
\end{align*}
For each $x,y\in \Sigma$ with $x_0 = y_0$, $W^s_{\text{loc}}(x)$ and $W^u_{\text{loc}}(y)$ intersect each other at a single point, namely, $[x,y] := (y^-, x^+)$.

We consider the space $C^{\gamma}(\Sigma,\, \SL_2(\R))$ endowed with the $\gamma$-H\"older topology $C^{\gamma}$: $A\in C^{\gamma}(\Sigma,\, \SL_2(\R))$, $\norm{A}_{C^{\gamma}} = \norm{A}_{\infty} + [A]_{\gamma}$, where
\begin{align*}
    [A]_{\gamma} = \sup_{x,y\in \Sigma}\frac{\norm{A(x) - A(y)}}{d(x,y)^{\gamma}},
\end{align*}
and $d$ denotes  is  the usual distance on $\Sigma$,
\begin{equation}
	\label{distance def}
	d(x,y) := \zeta^{-N(x,y)}, \quad \text{ with }\;  N(x,y) = \inf\{\vert n\vert \colon\, x_n\neq y_n\},\quad   \zeta>1.
\end{equation}

Take $A\in C^{\gamma}(\Sigma,\, \SL_2(\R))$ and let $F_A:\Sigma\times\bP^1\to\Sigma\times\bP^1$ be the projective linear cocycle associated with $A$. Let $\mathcal{M}_\mu(F_A)$ be the space of $F_A$ invariant probability measures that project on $\mu$ in the first coordinate.

The analysis of $F_A$-invariant measures in the study of Lyapunov exponent for cocycles, now depending on infinitely many coordinates, plays a very similar role to the analysis of stationary measures measures for locally constant cocycles. So, we start listing a few properties of this class of measures.
\begin{proposition}
\label{prop:propertiesMeasuresProductSpace}
    It holds that,
    \begin{enumerate}
        \item\label{prop:propertiesMeasuresProductSpace-item1} $\mathcal{M}_\mu(F_A)$ is non-empty, compact and convex. The extremal points of $\mathcal{M}_{\mu}(F_A)$ consist of ergodic measures for $F_A$;

        \item\label{prop:propertiesMeasuresProductSpace-item2} For each $m\in \mathcal{M}_\mu(F_A)$, there exists an (essentially unique) measurable family of measures $\{m_x\}_{x\in \Sigma}$ on $\bP^1$, the disintegration of $m$, such that for every measurable product $X\times\hat V\subset \Sigma\times\bP^1$,
        \begin{align*}
            m(X\times\hat V) = \int_X\, m_x(\hat V)\, d\mu(x).
        \end{align*}
        Moreover, $A(x)_*\, m_x = m_{\sigma(x)}$, for $\mu$ almost every $x\in \Sigma$;

        \item\label{prop:propertiesMeasuresProductSpace-item3} Assume that for every $x\in \Sigma$, $A(x) = A(x^+)$, i.e., $A$ only depends on the positive coordinates. The projection of an (ergodic) measure $m\in \mathcal{M}_\mu(F_A)$ on $\Sigma^{+}\times\bP^1$ is an (ergodic) probability measure invariant by $F_A^{+}$ (the natural map). Conversely, if $m^{+}$ is an (ergodic) $F_A^{+}$-invariant probability measure on $\Sigma^{+}\times\bP^1$ projecting on $\mu$, then there exists a unique (ergodic) measure $m\in \mathcal{M}_\mu(F_A)$ such that the projection is $m^{+}$. The disintegration of $m$ and $m^+$ are related by
        $$
            m^+_{x^+}=\int_{W^s_{loc}(x^+)} m_x\, d\mu^s_{x^+}(x),
            \quad 
            m_x=\lim_{n\to \infty}A^n(\sigma^{-n}(x))\, m^+_{{\sigma^{-n}(x)}^+},
        $$
        where $\{\mu^s_{x^+}\}_{x^+\in\Sigma^+}$ is the disintegration of $\mu$ with respect to the partition on local stable sets. A similar property holds when $A$ only depends on the negative coordinates;

        \item\label{prop:propertiesMeasuresProductSpace-item4} If $L(A)>0$, denote by $\hat e^s, \hat e^u$ the Oseledets directions (Proposition \ref{prop:basicProperties}). Write 
        \begin{align*}
            m^u = \int\, \delta_{\hat e^u}\, d\mu
            \quad
            \text{and}
            \quad
            m^s = \int\, \delta_{\hat e^s}\, d\mu.
        \end{align*}
        Then, $[m^u,\, m^s] = \mathcal{M}_\mu(F_A)$, i.e., any $m\in \mathcal{M}_\mu(F_A)$ can be written as
        \begin{align*}
            m = q_1\, m^u + q_2\, m^s = \int\, q_1\delta_{\hat e^u} + q_2\, \delta_{\hat e^s}\, d\mu.
        \end{align*}
        with $q_1\geq 0, q_2\geq 0$, $q_1 + q_2 = 1$.
    \end{enumerate}
\end{proposition}
\begin{proof}[Outline of the proof and references:]
    Item \ref{prop:propertiesMeasuresProductSpace-item1} follows a similar strategy as in the proof of item \ref{prop:propertiesStationryMeasures-item1} in Proposition \ref{prop:propertiesStationryMeasures}. Item \ref{prop:propertiesMeasuresProductSpace-item2} is an application of Rokhlin's disintegration theorem with respect to the partition $\{\{x\}\times\bP^1\colon\, x\in \Sigma\}$ of $\Sigma\times\bP^1$. The invariance of the conditional measures $m_x$ comes from the invariance of $m$ by $F_A$. For a proof of item \ref{prop:propertiesMeasuresProductSpace-item3} see \cite[Proposition~3.4]{AvVi10}. Item \ref{prop:propertiesMeasuresProductSpace-item4} can be found in \cite[Lemma~5.25]{Vi2014}.   
\end{proof}

\subsection{Fiber bunched cocycles}
As mentioned before, the complete description of the regularity of the Lyapunov exponent among cocycles in $C^{\gamma}(\Sigma,\, \SL_2(\R))$ is still open, but there is a class of cocycles where significant progress has been made. This is the class of maps $A$ such that the growth rate of the quantities $\norm{A^n}$ occurs slower than the expansion provided by the base dynamics. More precisely, we say that a cocycle $A\in C^{\gamma}(\Sigma, \SL_2(\R))$ is $\gamma$-\emph{fiber bunched} if there exists $n\in \N$ such that
\begin{align*}
    \sup_{x\in \Sigma}\norm{A^n(x)}^2
    = \sup_{x\in \Sigma}\norm{A^n(x)}\,\norm{(A^n(x))^{-1}}^{-1}
    < \zeta^{n \gamma},
\end{align*}
where $\zeta>1$ appears in the  definition~\eqref{distance def}  of the distance on $\Sigma$. It is important to highlight that the class of $\gamma$-fiber bunched cocycles $FB^{\gamma}(\Sigma)$ is an open subset of $C^{\gamma}(\Sigma,\SL_2(\R))$.
\begin{remark}
\normalfont
    A direct computation shows that the derivative of the projective actions $A^n(x):\bP^1\to\bP^1$ satisfies that, for every $\hat w\in \bP^1$
    \begin{align}
    \label{eq:fiberBunchedInequality}
        \norm{D\, [A^n(x)](\hat w)}
        \leq \frac{1}{
            \norm{A^n(x)\, w}^2
        }\leq  \norm{A^n(x)}^2.
    \end{align}
    Thus, assuming $A$  $\gamma$-fiber bunched, we have
    \begin{align}
    \label{eq:110423.1}
        \sup_{x\in \Sigma}\,\norm{D[A^n(x)]}_{\infty} \leq \zeta^{n\gamma},
    \end{align}
    In particular, the system $F_A:\Sigma\times\bP^1\to\Sigma\times\bP^1$ can be seen as a basic model of a \emph{partial hyperbolic system} (with compact fibers). Indeed, if we interpret $\sigma:\Sigma\to\Sigma$ as the basic model for hyperbolic system, then equation \eqref{eq:110423.1} means that the fiber bunched condition guarantees that the growth rate of the fiber action $A^n(x)$ is controlled by the the growth rate of the hyperbolic basis $\sigma$.
\end{remark}
\begin{example}[Bocker-Viana's Example]
\label{ex:BockerVianaExample}
\normalfont
\label{ex:fiberBunchedFamily}
    Take $\Sigma = \{1,2\}^{\Z}$, i.e., $\kappa = 2$. For any $\beta \geq \alpha \geq 1$ consider the map $A_{\alpha, \beta}:\Sigma\to \SL_2(\R)$ that for each sequence $x = (x_n)_{n\in \Z}$ is given by 
    \begin{align*}
        A_{\alpha, \beta}(x) = \left\{
            \begin{array}{cc}
                A_1=\begin{pmatrix}
                    \beta & 0 \\
                    0 & \beta^{-1}
                \end{pmatrix}, & x_0 = 1 \\
                A_2=\begin{pmatrix}
                    \alpha^{-1} & 0 \\
                    0 & \alpha
                \end{pmatrix}, & x_0 = 2
            \end{array}
        \right.
    \end{align*}
    Thus, it is clear that $A_{\alpha,\beta}\in C^{\gamma}(\Sigma,\, \SL_2(\R))$ ($A_{\alpha,\beta}$ is constant into cylinders of the form $[0;\, i]$, $i=1,2$). Notice that $\sup_{x\in\Sigma}\norm{A^n_{\alpha,\beta}(x)}^2 = \beta^{2n}$ and so, $A_{\alpha,\beta}$ is $\gamma$-fiber bunched as long as $\beta < \zeta^{\gamma/2}$. Observe that this class of cocycles is not uniformly hyperbolic. Moreover, by Birkhoff's ergodic theorem,
    \begin{align*}
        L(A_{\alpha,\beta}) = |p_1\,\log\beta - p_2\, \log\alpha|.
    \end{align*}
    and so $L(A_{\alpha,\beta})>0$, except for the case where $\beta^{p_1}=\alpha^{p_2}$.
\end{example}

\subsection{Holonomies}

Before discussing the main results of this section, we elaborate on one of the most important features of fiber bunched cocycles, namely, the existence of linear holonomies. We say that a family of matrices $\{H^{s}_{x,y}\in \SL_2(\R)\colon\, x,y\in \Sigma,\, y\in W^s_{\text{loc}}(x)\}$ is a family of \emph{stable linear holonomies} if
\begin{enumerate}
    \item  $H^{s}_{\sigma^j(x),\sigma^j(y)}=A^j(y) \circ H^s_{x,y} \circ A^j(x)^{-1}$ for every $j\geq 1$;
    
    \item $H^{s}_{x,x}=\id$ and $H^s_{x,y}=H^s_{z,y}\circ H^s_{x,z}$ for any $z\in W^{s}_{loc}(x)$;
    
    \item  $(x,y) \mapsto H^s_{x,y}$ is uniformly continuous on
           $\lbrace (x,y): y \in W^{s}_{\text{loc}}(x)\rbrace \subset M\times M$.

\end{enumerate}
The family of \emph{unstable linear holonomies} $\{H^{u}_{x,y}\in \SL_2(\R)\colon\, y\in W^u_{\text{loc}}(x)\}$ is defined analogously. 
\begin{remark}
\label{remark:stable/unstableSetsPartHyp}
\normalfont
    Existence of families of stable and unstable holonomies for $A$ is also related with the existence of strong stable/unstable sets for the (partially hyperbolic in the fiber bunched case) skew-product map $F_A:\Sigma\times\bP^1\to \Sigma\times\bP^1$: for every $x\in \Sigma$, $\hat v\in \bP^1$, define
    \begin{align*}
        W^{ss}(x,\hat v) = \left\{
            (y,\, H^s_{x,y}\, \hat v)\colon\,
            y\in W^s_{\text{loc}}(x)
        \right\}
        \quad
        \text{and}
        \quad
        W^{uu}(x,\hat v) = \left\{
            (y,\, H^u_{x,y}\, \hat v)\colon\,
            y\in W^u_{\text{loc}}(x)
        \right\}.
    \end{align*}
    Similarly to what happens for smooth partially hyperbolic systems, the map $F_A$ restricted to these stable/unstable sets satisfy contraction/expansion asymptotic properties (see \cite[Lemma~2.10]{Vi2008} for a precise statement).
\end{remark}
\begin{example}
\label{ex:holonomiesDiagonal}
\normalfont
\item Assume that $A\in C^{\gamma}(\Sigma,\, \SL_2(\R))$ is a diagonal cocycle, i.e.,
\begin{align*}
    A(x) = \begin{pmatrix}
        a(x) & 0 \\
        0 & a(x)^{-1}
    \end{pmatrix}.
\end{align*}
Consider $x,y\in \Sigma$, $y\in W^s_{\text{loc}}(x)$. For every $n\geq 1$,
\begin{align*}
    A^n(y)^{-1}\, A^n(x) =
    \begin{pmatrix}
        \varphi_n(x,y) & 0\\
        0 &  \varphi_n(x,y)^{-1}
    \end{pmatrix}
    \quad
    \text{with}
    \quad
    \varphi_n(x,y) = \frac{a(x)}{a(y)}\cdots \frac{a(\sigma^{n-1}(x))}{a(\sigma^{n-1}(y))}.
\end{align*}
We claim that the limit $\varphi(x,y) := \lim_{n\to\infty}\, \varphi_n(x,y)$ exists and so using this claim,
\begin{align*}
H^s_{x,y} := \begin{pmatrix}
    \varphi(x,y) & 0 \\
    0 & \varphi(x,y)^{-1}
\end{pmatrix}.
\end{align*}
defines a family of stable holonomies for $A$ (the conditions in the definition can be easily verified). To prove the claim, first consider $C>0$ such that for every $z,w\in \Sigma$,
\begin{align*}
    |\log a(z) - \log a(w)| \leq C d(z,w)^{\gamma}.
\end{align*}
Then, since $y\in W^s_{\text{loc}}(x)$
\begin{align*}
    |\log\varphi_{n+1}(x,y) - \log\varphi_n(x,y)|
    &= |\log a(\sigma^n(x)) - \log a(\sigma^n(y))|\\
    &\leq C\, d(\sigma^n(x),\, \sigma^n(y))^{\gamma}
    \leq C\, \zeta^{-n\gamma}.
\end{align*}
So, $(\varphi_n(x,y))_n$, is a (uniform) Cauchy sequence. This proves the claim.
\end{example}

\begin{example}
\label{ex:honolomiesRotations}
\normalfont
    Let $\xi:\Sigma\to\R$ be a $\gamma$-H\"older function and define $A(x) = R_{2\pi\xi(x)}$. For $x,y\in \Sigma$, $y\in W^s_{\text{loc}}(x)$ and for every $n\geq 1$,
    \begin{align*}
        A^n(y)^{-1}\, A^n(x) = R_{S_n\xi(x) - S_n\xi(y)},
    \end{align*}
    where $S_n\xi$ denotes the Birkhoff's sum of the function $\xi$. We claim that $(S_n\xi(x) - S_n\xi(y))_n$ forms a Cauchy sequence. Indeed, to see that is enough to notice that
    \begin{align*}
        |S_{n+1}\xi(x) - S_{n+1}\xi(y) - (S_n\xi(x) - S_n\xi(y))|
        \leq |\xi(\sigma^n(x)) - \xi(\sigma^n(y))| \leq C\, \zeta^{-n\gamma}.
    \end{align*}
    Now, let $\xi_0(x,y) = \lim_{n\to\infty}\, (S_n\xi(x) - S_n\xi(y))$. Then, the $H^s_{x,y} := R_{\xi_0(x,y)}$ defines a family of $s$-holonomies for $A$ (the properties in the definition are easily verified).
\end{example}
\begin{remark}
\normalfont
    Notice that the cocycle $A$ in the Example \ref{ex:honolomiesRotations} is always $\gamma$-fiber bunched since it is bounded. On the other hand, Example \ref{ex:holonomiesDiagonal} may not be  $\gamma$-fiber bunched, but can always be seen as a ``direct product of fiber bunched'' $1$-dimensional cocycles. Indeed, in this case the existence of the limits $A^n(y)^{-1}\, A^n(x)$ is ensured by the conformality of the maps $A|_{\hat e_1}$ and $A|_{\hat e_2}$.
\end{remark}

\begin{proposition}
\label{prop:HolonomiesFB}
    If $A\in C^{\gamma}(\Sigma,\, \SL_2(\R)^{\kappa}$ is $\gamma$-fiber-bunched, then the families
    \begin{itemize}
        \item $\{\lim_{n\to\infty}\, A^n(y)^{-1}\, A^n(x)\colon\,
                y\in W^s_{\text{loc}}(x)\}$;

        \item $\{\lim_{n\to\infty}\, A^n(\sigma^{-n}(y))\, A^n(\sigma^{-n}(x))^{-1}\colon\,
                y\in W^u_{\text{loc}}(x)\}$
    \end{itemize}
    are respectively the unique family of stable and unstable linear holonomies such that for some $L>0$
    \begin{align}
    \label{eq:betaHolderHolonomies}
        \|H^s_{x,y}-\id\| \leq L \dist(x,y)^{\gamma}
        \quad
        \text{and}
        \quad
        \|H^u_{x,y}-\id\| \leq L \dist(x,y)^{\gamma}.
    \end{align}
\end{proposition}
\begin{proof}[Outline of the proof and references:]
    For a proof see \cite[Lemma~1.12]{BoGoVi2003}.
\end{proof}
\begin{remark}
\normalfont
    In general, the family of stable and unstable holonomies are not unique even for a constant fiber bunched cocycle (see the discussion after 4.9 in \cite{KaSa2014} and Theorem 5.5.5 in \cite{KaNi2011}). However, families of holonomies satisfying the conditions in \eqref{eq:betaHolderHolonomies} are unique even without fiber bunched assumptions (See Proposition 3.2 in \cite{KaSa2014}).
\end{remark}

\begin{remark}
\normalfont
    Any $C^{\gamma}(\Sigma,\, \SL_2(\R))$ cocycle , $C^0$-close to a constant one, has families of stable and unstable linear holonomies satisfying \eqref{eq:betaHolderHolonomies}, possibly with some $\gamma'<\gamma$ instead of $\gamma$ (see \cite[Proposition~3.6]{KaSa2014}).
\end{remark}
\begin{example}
\normalfont
    If $A\in \SL_2(\R)^{\kappa}$ is locally constant then the identity map forms a family of stable/unstable holonomies. Assume that there exists another family of linear stable holonomies, say $\{H^s_{x,y}\}$. Then, for every $x,y\in \Sigma$, $y\in W^s_{\text{loc}}(x)$ notice that $A^n(x) = A^n(y)$ and so we have
    \begin{align*}
        H^s_{\sigma^n(x),\sigma^n(y)} = A^n(x)\, H^s_{x,y}\, A^n(x)^{-1}.
    \end{align*}
    In particular, $H^s_{\sigma^n(x),\sigma^n(y)}$ and $H^s_{x,y}$ share the same spectrum for every $n$ ($A^n(x)$ is a conjugation). Since $d(\sigma^n(x),\sigma^n(y))$ converges to zero, we have that the spectrum of the matrix $H^s_{x,y}$ is a singleton with only $1$. So, either $H_{x,y}^s$ is the identity matrix or else is conjugated to a parabolic matrix of the form $\begin{pmatrix}1 & a
    
    \\ 0 & 1 \end{pmatrix}$.
\end{example}

One of the most important uses of the existence of holonomies is the possibility to study the problem from the point of view of one-sided shift. 
\begin{proposition}
\label{prop:reducitionToOneSided}
    If $A$ admits a family of linear stable holonomies, then $A$ is conjugated to a cocycle $A^+$ such that for every $x\in \Sigma$, $A^+(x) = A^+(x^+)$. Similar property holds for unstable linear holonomies.
\end{proposition}
\begin{proof}
    See \cite[Corollary~1.15]{BoGoVi2003}.
\end{proof}

\begin{remark}
\normalfont
    One could think that once the cocycle admits both families of holonomies, then it would be conjugated to a locally constant cocycle. Unfortunately, this is not the case in general. 
    Being conjugate to a locally constant cocycle is related to the problem of accessibility of partially hyperbolic systems.

    A map $F_A:\Sigma\times \mathbb{P}^1\to \Sigma\times \mathbb{P}^1$ is  said to be accessible if for every $(x,\,\hat{v}),(y,\,\hat{u})\in \Sigma\times \mathbb{P}^1$ there exists a path of strong stable and strong unstable sets, introduced in Remark \ref{remark:stable/unstableSetsPartHyp}, that goes from $(x,\, \hat v)$ to $(y,\, \hat u)$.
    
    Stable and unstable sets of a cocycle are mapped to the corresponding ones for the conjugated cocycle via the conjugation map. This implies that accessibility is preserved by conjugation. However, locally constant cocycles are never accessible. This shows that $\gamma$-fiber bunched accessible maps can not be conjugated to locally constant cocycles even though can be conjugated to cocycles depending only on the positive (or negative) coordinates.
\end{remark}

\subsection{su-States}

Holonomies allows us to transport components of the disintegration $\{m_x\}_{x\in \Sigma}$ of a given measure $m\in \mathcal{M}_{\mu}(F_A)$ throughout different points in the same stable set. Taking $x, y\in \Sigma$, $y\in W^s_{\text{loc}}(x)$, we are able to compare the structure of $m_y$ with the structure of $(H^s_{x,y})_*\, m_x$.

One especial case of this transport property stands out. We say that $m$ is an $s$-\emph{state} if for $\mu$ almost every $x,y\in \Sigma$, $y\in W^s_{\text{loc}}(x)$, $(H^s_{x,y})_*\, m_x = m_y$. Analogously, $m$ is an $u$-\emph{state} if for $\mu$ almost every $x,y\in \Sigma$, $y\in W^u_{\text{loc}}(x)$, $(H^u_{x,y})_*\, m_x = m_y$. $m$ is an $su$-\emph{state} when it is both an $s$-state and 
an $u$-state.
\begin{example}
\normalfont
    Assume that $A\in \SL_2(\R)^{\kappa}$ is locally constant. Then, it is easy to see that the identity matrix forms a family of stable linear holonomies for $A$. In this case, a measure $m\in \mathcal{M}_\mu(F_A)$ is an $s$-state if an only if it has a disintegration $\{m_x\}_x$ which is constant along stable sets. In particular, the projection of $m$ on $\Sigma^+\times \bP^1$ is the product measure $\mu\times \eta$, where $\eta$ is a forward stationary measure for the cocycle $A$. Reciprocally, if $\eta$ is a forward stationary measure, then the lift $\mu\times\eta$ (see Proposition \ref{prop:propertiesMeasuresProductSpace}) is constant along stable sets and so is an $s$-state. 
\end{example}
The previous example indicates that $s$ and $u$ states are the natural generalization of the concept of stationary measure when dealing with cocycles which (possibly) depend on infinitely many coordinates.
\begin{proposition}
\label{prop:properties-suStates}
    Let $A\in C^{\gamma}(\Sigma,\, \SL_2(\R))$ be a $\gamma$-fiber bunched cocycle.
    \begin{enumerate}
        \item\label{prop:properties-suStates-item1} The set of measures $m\in \mathcal{M}_{\mu}(F_A)$ which are $s$-states (respect. $u$-states) is non-empty, compact and convex;

        \item\label{prop:properties-suStates-item2} If $L(A)>0$, then the measure $m^s$ is an $s$-state and $m^u$ is an $u$-state;

        \item\label{prop:properties-suStates-item3} Assume that $L(A)>0$. Either $m^s$ (respect. $m^u$) is the unique $s$-state ($u$-state) or else all the measures in $\mathcal{M}_{\mu}(F_A)$ are $s$-states ($u$-states);

        \item\label{prop:properties-suStates-item4} Assume that $L(A)>0$, we have (the following consequence of Oseledets Theorem - Proposition \ref{prop:basicProperties})
        \begin{align*}
            L(A) = \int_{\Sigma\times\bP^1}\, \log \norm{A(x)\, v}\, d\,m^u(x,\hat v)
            = - \int_{\Sigma\times\bP^1}\, \log\norm{A(x)\, v}\, d\,m^s(x,\hat v);
        \end{align*}
    \end{enumerate}
\end{proposition}
\begin{proof}[Outline of the proof and references:]
    For a proof of item \ref{prop:properties-suStates-item1} see \cite[Proposition~4.4]{BoVi2004}. Item \ref{prop:properties-suStates-item2} and \ref{prop:properties-suStates-item4} are consequences of Oseledets theorem (see item \ref{prop:basicPropertiesContCocycles-Oseledets} of Proposition \ref{prop:basicProperties}). For a proof of item \ref{prop:properties-suStates-item3}, see \cite[Corollary~5.27]{Vi2014}.
\end{proof}

The next proposition shows the strength of the concept of $su$-states. It allows us to bypass the issues associated with the lack of regularity of the map $x\mapsto m_x$ which is only measurable in general. 
\begin{proposition}
\label{prop:continuity-SUStates}
    Let $A\in C^{\gamma}(\Sigma,\, \SL_2(\R))$ be a $\beta$ fiber bunched cocycle.
    \begin{enumerate}
        \item\label{prop:continuity-SUStates-item1} If $m\in \mathcal{M}_{\mu}(F_A)$ is an $u$-state and $A(x) = A(x^+)$, then the projection of $m$, $m^+$, on $\Sigma^+\times\bP^1$ admits a H\"older continuous disintegration $x^+\mapsto m^+_{x^+}$.

        \item\label{prop:continuity-SUStates-item2} If $m\in \mathcal{M}_{\mu}(F_A)$ is a $su$-state, then $m$ admits a continuous disintegration. 

        \item\label{prop:continuity-SUStates-item3} Assume that $L(A)>0$. If there exists more than one $s$-state in $\mathcal{M}_{\mu}(F_A)$, then $A$ is continuously conjugated to a triangular cocycle, i.e., there exists $C:\Sigma\to\SL_2(\R)$ continuous such that
        \begin{align*}
            C(\sigma^{-1}(x))\, A(x)\, C(x) = \begin{pmatrix}
                a(x) & b(x) \\
                0 & d(x)
            \end{pmatrix}.
        \end{align*}
        The same conclusion holds if there exists more than one $u$-state in $\mathcal{M}_{\mu}(F_A)$.

        \item\label{prop:continuity-SUStates-item4} Assume that $L(A)>0$. If there exists a non-ergodic measure $m\in \mathcal{M}_{\mu}(F_A)$ which is an $su$-state, then $A$ is continuously conjugated to a diagonal cocycle, i.e., there exists $C:\Sigma\to\SL_2(\R)$ continuous such that
        \begin{align*}
            C(\sigma^{-1}(x))\, A(x)\, C(x) = \begin{pmatrix}
                a(x) & 0 \\
                0 & d(x)
            \end{pmatrix}.
        \end{align*}
    \end{enumerate}
\end{proposition}
\begin{proof}[Outline of the proof and references:]
    For a proof of item \ref{prop:continuity-SUStates-item1} see \cite[Proposition~4.3]{BoVi2004}. Item \ref{prop:continuity-SUStates-item2} is proved in \cite[Proposition~4.8]{AvVi10}. Items \ref{prop:continuity-SUStates-item3} and \ref{prop:continuity-SUStates-item4} are consequences of the item \ref{prop:continuity-SUStates-item2}.
\end{proof}

The above proposition is an indication of the rigidity of the space of cocycles admitting $su$-states. These measures appear naturally when the Lyapunov exponent of the cocycle in consideration vanishes. This is the content of the next result.
\begin{theorem}[Invariance principle, \cite{AvVi10}]
\label{thm:invariancePrincipleFB}
    If $A\in C^{\gamma}(\Sigma,\, \SL_2(\R))$ is $\gamma$-fiber bunched and $L(A) = 0$, then any measure $m\in \mathcal{M}_{\mu}(F_A)$ is a $su$-state.
\end{theorem}
Recall that for locally constant cocycles $A\in \SL_2(\R)^{\kappa}$ the invariance principle states, under the assumption $L(A)=0$, that any stationary measure is fixed through the action of all the matrices $A_i$. Here, this invariance property is by the fact that both holonomies preserves the disintegrations of any measure $m\in \mathcal{M}_{\mu}(F_A)$. The proof of Theorem \ref{thm:invariancePrincipleFB} follows similar lines as the proof of Theorem \ref{thm:LedrappierInvariancePrinciple}.

\subsection{Typical fiber bunched cocycles}

The invariance principle provides a criterion to check positivity of the Lyapunov exponent: it is enough to find a single measure $m\in \mathcal{M}_{\mu}(F_A)$ which is not an $su$-state. Certainly, such a criterion is not very useful in practice. Our aim is to present a criterion for positivity of the Lyapunov exponent in the same lines as Furstenberg's criterion in the context of locally constant cocycles.

Let $A\in C^{\gamma}(\Sigma,\, \SL_2(\R))$ be a cocycle admitting families of linear stable and unstable holonomies. One necessary condition to guarantee positivity the Lyapunov exponent of the cocycle $A$ is the existence of some periodic point $q\in \Sigma$, of period $k$, which is a \emph{pinching} periodic point, i.e., the matrix $A^k(q)$ is hyperbolic. In this case, we have two eigen-directions denoted by $\hat{v_1}(q), \hat{v_2}(q)\in \bP^1$ for the matrix $A^k(q)$.

Assume for a moment that $A$ admits an $su$-state $m\in \mathcal{M}_{\mu}(F_A)$. In particular, by continuity of the disintegration it makes sense to consider the projective measure $m_q$. Observe that, this is an invariant measure for the matrix $A^k(q)$ and so, it must to have the following form,
\begin{align*}
    m_q = a\, \delta_{\hat v_1(q)} + b\, \delta_{\hat v_2(q)},
\end{align*}
with $a,b\geq 0$, $a+b=1$. Then, using the invariance of $m$ by the  holonomies and by the cocycle it is not hard to see that for any $z\in \Sigma$, there exist directions $\hat v_1(z), \hat v_2(z)\in \bP^1$ such that 
\begin{align*}
    m_z = a\, \delta_{\hat v_1(z)} + b\, \delta_{\hat v_2(z)}.
\end{align*}
Furthermore, the set of continuous sections $\{\hat v_1,\hat v_2\in C^0(\Sigma,\, \bP^1)\}$ is $F_A$-invariant in the sense that for every $z\in \Sigma$, $A(z)\, \hat v_i(z)\in \{\hat v_1(\sigma(z)),\, \hat v_2(\sigma(z))\}$ for $i=1,2$. So, if we want to guarantee that $su$-states are not allowed in our system we should ask that this type invariant set does not exists. 

Now we introduce a concept incompatible with the existence of $F_A$-invariant sets of continuous sections. Let $z\in W^u_{\text{loc}}(q)$ be a point homoclinic related to $q$ meaning that $z^-=q^-$ and for some $l\geq 1$, $\sigma^l(z)\in W^s_{\text{loc}}(q)$. Assume that the map $H^s_{\sigma^l(z),q}\, A^l(z)\, H^u_{q,z}$ \emph{twists} the invariant subspaces of $A^k(q)$, meaning that
\begin{align}
\label{eq:300323.1}
    H^s_{\sigma^l(z),q}\, A^l(z)\, H^u_{q,z}\left(\left\{
        \hat v_1(q),\, \hat v_2(q)
    \right\}\right)\cap \left\{
        \hat v_1(q),\, \hat v_2(q)
    \right\} = \varnothing.
\end{align}
In particular, equation \eqref{eq:300323.1} implies that invariant sets of continuous sections as above do not exist. 

The above discussion motivates the following definition: we say that the cocycle $A$ is \emph{pinching} if there exists a periodic point $q$ as above such that $A^k(q)$ is hyperbolic, and we say that $A$ is \emph{twisting} if there exists a homoclinic point $z$ related to $q$ such that the property in \eqref{eq:300323.1} is satisfied.

A weaker form of twisting is sufficient for most of the results discussed in these notes. We call  $\rho=\{x_0,\cdots, x_{n+1}\}$ an  $su$-\emph{loop} if 
$x_{i+1}\in W^*_{\text{loc}}(x_i)$, for $*\in\{s,u\}$ and $i=0,\dots, n$ with $x_0=x_{n+1}$. We say that a pinching cocycle is \emph{weak twisting} if there exists an $su$-loop with $x_0=q$ such that the map
$H_{\rho}=H^{*_{n}}_{x_n,x_{n+1}}\circ \cdots \circ H^{*_{0}}_{x_0,x_{1}}$, where $*_i=s$ or $u$ according to the definition of the loop, satisfies 
\begin{align}
\label{eq:weak.twits}
    H_{\rho}\left(\left\{
        \hat v_1(q),\, \hat v_2(q)
    \right\}\right)\cap \left\{
        \hat v_1(q),\, \hat v_2(q)
    \right\} = \varnothing.
\end{align}
Of course twisting implies weak twisting. Observe that the pinching and weak twisting condition also implies that there are no $su$-states.

\begin{figure}[ht]
    \centering
    \includegraphics[width=\textwidth]{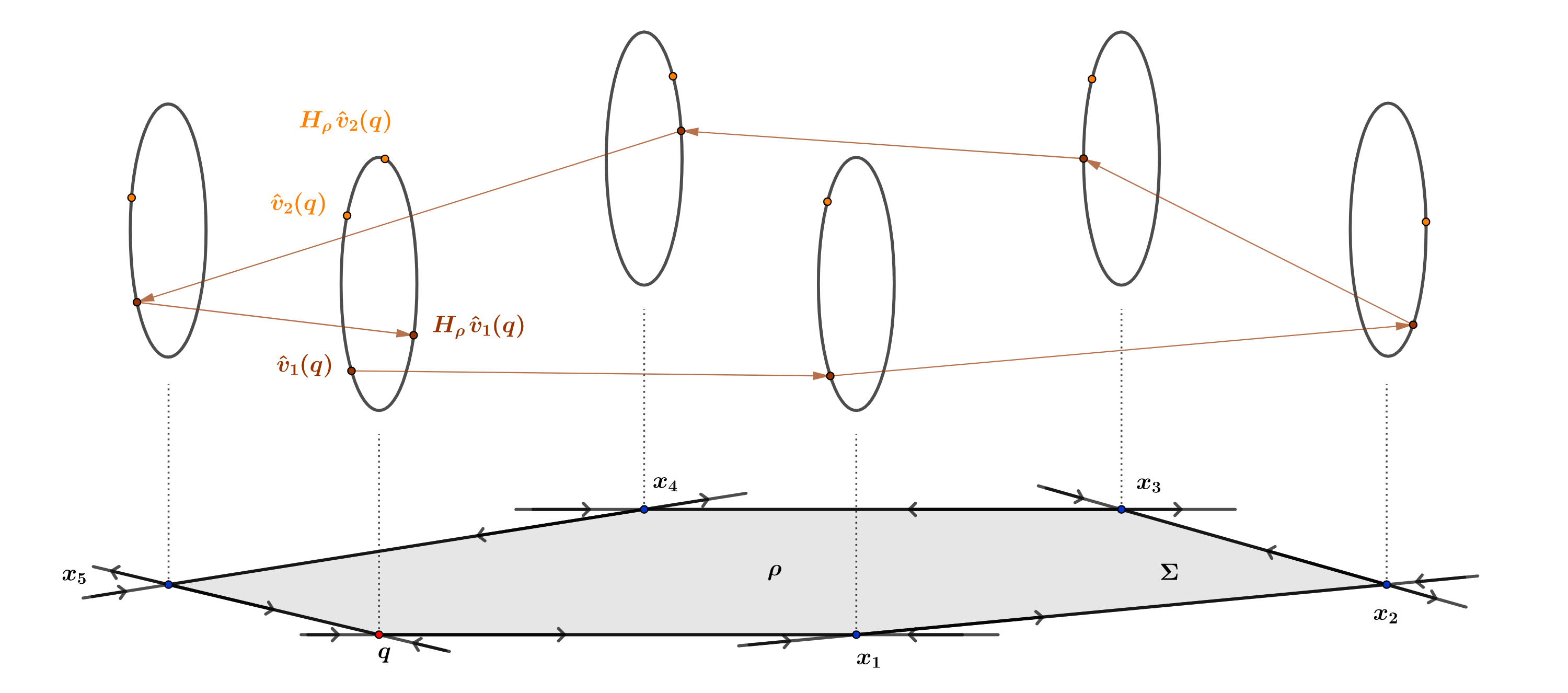}
    \caption{Weak Twisting.}
    \label{pic:Weak Twisting}
\end{figure}

\begin{problem}
    Does weak twisting imply twisting?
\end{problem}

\begin{remark}
\normalfont
    Note that when speaking about twisting property it is assumed that the cocycle in consideration admits families of stable and unstable holonomies.
\end{remark}
The next proposition summarizes the above the discussion:
\begin{proposition}
\label{prop:propertiesPinchingTwisting}
    Let $A\in C^{\gamma}(\Sigma,\, \SL_2(\R))$ be a $\gamma$-fiber bunched cocycle.
    \begin{enumerate}
        \item\label{prop:propertiesPinchingTwisting-item1} The set of pinching and twisting cocycles is an open and dense subset of  $FB^{\gamma}(\Sigma)$;
        
        \item\label{prop:propertiesPinchingTwisting-item2} If $A$ is pinching and weak twisting, then there are no $su$-states;

        \item\label{prop:propertiesPinchingTwisting-item3} If $L(A)>0$, then $A$ is pinching;

        \item\label{prop:propertiesPinchingTwisting-item4} Assume that $L(A)>0$. $A$ is weak twisting if and only if there are no $su$-states;
        
        \item\label{prop:propertiesPinchingTwisting-item5} If $A$ is pinching and weak twisting, then $L(A)>0$. Moreover, there is only one $u$-state and one $s$-state, namely $m^u$ and $m^s$ respectively;

        \item\label{prop:propertiesPinchingTwisting-item6} Let $A$ be a pinching and twisting cocycle depending only on the positive coordinates. If $m\in \mathcal{M}_{\mu}(F_A)$ is an $u$-state and $m^+$ is its projection on $\Sigma^+\times \bP^1$ with a continuous disintegration $\{m^+_{x^+}\}_{x^+}$, then for every $x^+\in \Sigma^+$, $m^+_{x^+}(\{v\})=0$ for every $v\in \mathbb{P}^1$. A similar property holds for $A$ depending only on the negative coordinates;
    \end{enumerate}
\end{proposition}
\begin{proof}[Outline of the proof and references:]
    For a proof of item \ref{prop:propertiesPinchingTwisting-item1} see \cite[Theorem~7]{BoGoVi2003}. Item \ref{prop:propertiesPinchingTwisting-item2} is a consequence of the discussion just before the statement. For \ref{prop:propertiesPinchingTwisting-item3} see \cite[Theorem~1.4]{Ka2011}.
    
    One side of \ref{prop:propertiesPinchingTwisting-item4} is a direct consequence of items \ref{prop:propertiesPinchingTwisting-item2} and \ref{prop:propertiesPinchingTwisting-item3}. Now, assume that $L(A)>0$ and that there are no $su$-states. Let $q\in \Sigma$ be a $\sigma$-periodic given by item \ref{prop:propertiesPinchingTwisting-item3}. For simplicity assume that $\sigma(q) = q$. Let $\hat V = \{\hat v_1,\, \hat v_2\}\subset \bP^1$ be the set of invariant directions of the matrix $A(q)$.

    If $A$ is not twisting, then we have that, 
    \begin{align}
    \label{eq:120423.1}
        H_{\rho}(\hat V) \cap \hat V \neq \varnothing
        \quad
        \text{ for every $su$-loop } \rho.
    \end{align}
    If either there exists $\hat v\in \hat V$ such that $H_{\rho}\, \hat v = \hat v$ for every $su$-loop $\rho$, or $H_{\rho}(\hat V) = \hat V$, for every $su$-loop $\rho$ we are able to build a $su$-state just considering the atomic measure with even weights in each case contradicting the assumption that there are no $su$-states. Therefore, we may assume that these cases do not happen. In particular, we can find a $su$-loop $\rho_1$ such that $H_{\rho_1}(\hat v_1)\notin \hat V$. By \eqref{eq:120423.1}, we have that $H_{\rho_1}(\hat v_2)\in \hat V$. Consider the following cases:
    
    \noindent
    \emph{Case 1:} Assume that $H_{\rho_1}(\hat v_2) = \hat v_1$. In this case, there exists $n\in \N$ such that the map $H_{\rho_1}\, A^n(q)\, H_{\rho_1}$ (also associated to an $su$-loop) satisfies $H_{\rho_1}\, A^n(q)\, H_{\rho_1}(\hat v_1)\notin \hat V$, for $i=1,2$.

    \noindent
    \emph{Case 2:} Assume that $H_{\rho_1}(\hat v_2) = \hat v_2$. Observe that in this case there exists an $su$-loop $\rho_2$ such that $H_{\rho_2}(\hat v_2)\notin \hat V$. Indeed, we already excluded the case where $H_{\rho}(\hat v_2) = \hat v_2$, for every $su$-loop $\rho$. If $\rho$ is an $su$-loop such that $H_{\rho}(\hat v_2) = \hat v_1$, then $H_{\rho_1}\, H_{\rho}(\hat v_2) = H_{\rho_1}(\hat v_1)\notin \hat V$ by the initial assumption. Let $\rho_2$ be an $su$-loop such that $H_{\rho_2}(\hat v_2)\notin \hat V$.

    \noindent
    \emph{Case 2.1:} Assume that $H_{\rho_1}(\hat v_2) = \hat v_2$ and $H_{\rho_2}(\hat v_1) = \hat v_1$. Here, there exists $n\in \N$ such that the map $H_{\rho_1}\, A^n(q)\, H_{\rho_2}$  satisfies $H_{\rho_1}\, A^n(q)\, H_{\rho_2}(\hat v_i)\notin \hat V$, for $i=1,2$.

    \noindent
    \emph{Case 2.2:} Assume that $H_{\rho_1}(\hat v_2) = \hat v_2$ and $H_{\rho_2}(\hat v_1) = \hat v_2$. Exactly the same as Case 1 with $\rho_2$ and $\hat v_2$ instead of $\rho_1$ and $\hat v_1$.

    In any of the above described cases, we obtain a contradiction with the assumption \eqref{eq:120423.1}. This concludes the proof of item \ref{prop:propertiesPinchingTwisting-item4}.
    
    Using the invariance principle, Theorem \ref{thm:invariancePrincipleFB}, and item \ref{prop:properties-suStates-item3} of Proposition \ref{prop:properties-suStates}, we have that item \ref{prop:propertiesPinchingTwisting-item5} holds. For a proof of item \ref{prop:propertiesPinchingTwisting-item6} see \cite[Proposition~5.1]{BoVi2004}.
\end{proof}

We say that a pinching $\gamma$-fiber bunched cocycle $A\in C^{\gamma}(\Sigma,\, \SL_2(\R))$ is \emph{quasi-twisting}, if there exists only one $u$-state in $\mathcal{M}_{\mu}(F_A)$.

\begin{proposition}
\label{prop:quasi-twistingDiagonal}
    Let $A\in C^{\gamma}(\Sigma,\, \SL_2(\R))$ be a $\gamma$-fiber bunched cocycle with $L(A)>0$. Then either $A$ (or $A^{-1}$) is quasi-twisting or $A$ is continuously conjugated to a diagonal cocycle.
\end{proposition}
\begin{proof}[Outline of the proof:]
    If $L(A)>0$ and $A$ is not weak twisting, there exists a measure $m\in \mathcal{M}_{\mu}(F_A)$ which is a $su$-state (see item \ref{prop:propertiesPinchingTwisting-item4} of Proposition \ref{prop:propertiesPinchingTwisting}). If $m$ is non-ergodic, then, by item \ref{prop:continuity-SUStates-item4} of Proposition \ref{prop:continuity-SUStates}, $A$ is conjugated to a diagonal.  Otherwise, $m\in \{m^s,\, m^u\}$ and so either $A$ or $A^{-1}$ is quasi-twisting. If $A$ is twisting, then in particular $A$ is quasi-twisting.
\end{proof}

\begin{example}[Pinching and twisting for locally constant cocycles]
\label{ex:pinchingTwistingLCC}
\normalfont
    Let $A = (A_1,\ldots, A_{\kappa})\in \SL_2(\R)^{\kappa}$ be a strongly irreducible, locally constant cocycle such that the semi-group generated by $A_1,\ldots, A_{\kappa}$ is unbounded. Then, by Furstenberg's criterion, $L(A)>0$. An argument similar to the proof of item \ref{prop:propertiesPinchingTwisting-item4} Proposition \ref{prop:propertiesPinchingTwisting} guarantees that $A$ is pinching and twisting. In this case  \eqref{eq:300323.1} gives only the cocycle matrix since the holonomies are identity matrices.
\end{example}

\subsection{Continuity of the Lyapunov exponent for fiber bunched cocycles}

It is very useful to think in pinching and twisting as the equivalent conditions to unboundedness and strong irreducibility in the Furstenberg's criterion. Using that interpretation it is natural to wonder if pinching and weak twisting $\gamma$-fiber bunched cocycles are continuity points of Lyapunov exponent in the $C^{\gamma}$ topology. That is indeed the case. Consider $A_k$ converging to $A$ and let $m^k\in\mathcal{M}_{\mu}(F_{A_k})$ ergodic $u$-states realizing the Lyapunov exponent, i.e., 
\begin{align}
\label{eq:040423.2}
    L(A_k) = \int_{\Sigma\times\bP^1}\, \log\norm{A_k(x)\, v}\, d\, m^k(x,\hat v).
\end{align}
This is possible since if $L(A_k)>0$ we may choose $m^k = m^u_{A_k}$ and in the case that $L(A_k)=0$, by the invariance principle Theorem \ref{thm:invariancePrincipleFB}, any ergodic measure $m^k$ works.

Up to a subsequence, we may assume that $m^k$ converges, and so it must converge to some measure $m\in \mathcal{M}_{\mu}(F_A)$. It is possible to show that limit is also an $u$-state for $A$. But, by the pinching and weak twisting condition on $A$, there is only one $u$-state, namely $m^u$. So,
\begin{align}
\label{eq:030423.1}
    L(A_k)
    = \int_{\Sigma\times\bP^1}\, \log\norm{A_k(x)\, v}\, d\, m^k(x,\hat v)
    \to \int_{\Sigma\times\bP^1}\, \log\norm{A(x)\, v}\, d\, m^u(x,\hat v)
    = L(A).
\end{align}
The above discussion is summarized in the next proposition.
\begin{proposition}
\label{prop:continuityPinchingTwisting}
    If $A\in C^{\gamma}(\Sigma,\, \SL_2(\R))$ is a pinching and weak twisting $\gamma$-fiber bunched cocycle, then $L$ is continuous at $A$ in the $C^{\gamma}$-topology.
\end{proposition}
\begin{example}
\normalfont
    If $A$ is an irreducible locally constant $\gamma$-fiber bunched cocycle, we claim that $A$ is a continuity point of the Lyapunov exponent in the H\"older topology, $\gamma>0$ (recall that the same is not true for $\gamma=0$). Indeed, if $L(A)=0$, there is nothing to prove. Assume that $L(A)>0$, then by item \ref{prop:propertiesofTypicalCocycles-item4} of Proposition \ref{prop:propertiesofTypicalCocycles} $A$ is strongly irreducible. By Example \ref{ex:pinchingTwistingLCC}, we see that $A$ is pinching and twisting. The claim follows by Proposition \ref{prop:continuityPinchingTwisting}.

    More generally, assume that $A\in C^{\gamma}(\Sigma,\, \SL_2(\R))$ is a $\gamma$-fiber bunched cocycle with $A(x) = A(x^+)$, for every $x\in \Sigma$. It was proved in \cite[Theorem A]{ViJa2019} that if additionally $A$ is pinching and weak twisting, then $A$ is a continuity point of the Lyapunov exponent in the $C^0$-topology. A consequence of this result is that for non-invertible base dynamics, Bochi-Mane Theorem \ref{thm:BoMa} is no longer true.
\end{example}

Now, assume that the fiber bunched cocycles $A$ is not weak twisting and we investigate the continuity of the Lyapunov exponent at $A$.

Again, it is enough to consider the case $L(A)>0$. In this case, using item \ref{prop:propertiesPinchingTwisting-item4} of Proposition \ref{prop:propertiesPinchingTwisting}, there exists a measure $\tilde{m}\in \mathcal{M}_{\mu}(F_A)$ which is a $su$-state. If $\tilde{m}\in \{m^s,\, m^u\}$ and this is the unique $su$-state, then either $A$ or $A^{-1}$ is quasi-twisting and a similar argument to prove continuity of $L$ used in the weak twisting case above works. For that reason, we may assume that, up to continuous conjugacy, $A$ is diagonal of the form (see Proposition \ref{prop:quasi-twistingDiagonal})
\begin{align*}
    A(x) = \begin{pmatrix}
        a(x) & 0 \\
        0 & a(x)^{-1}
    \end{pmatrix}.
\end{align*}
Consider $A_k\to A$ and $m^k\to m$ as above. In particular, by the fact that $A$ is diagonal, we have that $m$ must to have the following disintegration
\begin{align*}
    m_x = a\, \delta_{\hat e_1} + b\, \delta_{\hat e_2},
\end{align*}
where $a,b\geq 0$, $a+b=1$ and $\hat e_1,\hat e_2$ are the unitary vectors of the canonical basis of $\R^2$.

A technical issue that appears in the non locally constant case is that, even though the $u$-states for $A_k$, $m^k$, converges to an $u$-state, $m$, for $A$ we are not able to guarantee that the disintegrations $m^k_x$ converges to $m_x$ for $x\in \Sigma$.  That is mainly due to the lack of regularity of these disintegrations (which in general are only measurable). To overcome this issue we use the existence of stable holonomies (uniformly in a neighborhood of $A$) to change coordinates and restrict to the case where $A_k$ and $A$ only depends on the positive coordinates (see Proposition \ref{prop:reducitionToOneSided}). This is very useful by the following reason: now we are able to project the measures $m^k$ and $m$ in the unilateral system $\Sigma^+\times\bP^1$ obtaining measures $m^{k,+}$ (ergodic) and $m^+$, invariant under $F^+_{A_k}$ and $F^+_A$ respectively, such that 
\begin{enumerate}
    \item\label{item:proofBBB1} The disintegrations $\{m^{k,+}_{x^+}\}_{x^+\in \Sigma^+}$ and $\{m^+_{x^+}\}_{x^+\in \Sigma^+}$ are continuous disintegrations of $m^{k,+}$ and $m^+$;

    \item\label{item:proofBBB2} For every $x^+\in \Sigma^+$, $m^{k,+}_{x^+}$ converges to $m^+_{x^+}$ in the weak* topology.
\end{enumerate}
Now, we proceed similarly to the locally constant case.  Elaborating more on that, first we use a variation of the "energy argument" (using the fact that $A$ is diagonal and $L(A)>0$) to guarantee that the probabilities $m^{k,+}_{x^+}$ must to be atomic.

The proof follows by a classification of the atomic ergodic measures $m^{k,+}$ that can approach $m$. By item \ref{prop:propertiesPinchingTwisting-item3} of Proposition \ref{prop:propertiesPinchingTwisting}, there exists a periodic point $q\in \Sigma$, such that $A^l(q)$ is a hyperbolic matrix. So, $A_k^l(q)$ is also hyperbolic for every $k$ sufficiently large (we assume that for every $k$). This, jointly with the continuity of the disintegration of $m^{k,+}$ in $\Sigma^+$, produces constraints for the structure of the measures $m^{k,+}$. In either case, the expression \eqref{eq:030423.1} is verified. This is summarized in the following result due to Backes, Butler and Brown.
\begin{theorem}[L. Backes, C. Butler, B. Brown, \cite{BBB18}]
\label{thm:BBB2018}
    The map $FB^{\gamma}(\Sigma)\ni A\mapsto L(A)$ is continuous.  
\end{theorem}
This result shows that the problem of understanding the Lyapunov exponent function in the $C^{\gamma}$-topology, $\gamma>0$, is completely different from the problem in the $C^0$-topology. Recall that in the $C^{0}$-topology,  cocycles with non-zero Lyapunov exponent that are not uniformly hyperbolic (e.g., Example \ref{ex:fiberBunchedFamily}) are discontinuity points of $L$. Another important consequence of Theorem \ref{thm:BBB2018} is that now we have open sets of non-hyperbolic cocycles inside $C^{\gamma}(\Sigma,\SL_2(\R))$ in which the Lyapunov exponent does not vanish.

\begin{example}[Discontinuity example]
\normalfont
    In Bocker-Viana example (Example \ref{ex:BockerVianaExample}), it is possible to give a set of parameters $1\leq \alpha\leq \beta$ for which $A_{\alpha,\beta}$ is a discontinuity point of the Lyapunov exponent in the $C^{\gamma}$-topology. Indeed, a direct adaptation of the argument in \cite[Theorem~9.22]{Vi2014} shows that if $\alpha> \zeta^{2\gamma}$, then $A_{\alpha,\beta}$ is a discontinuity point of the Lyapunov exponent (recall that $A_{\alpha,\beta}$ is $\gamma$-fiber bunched for $\beta< \zeta^{\gamma/2}$).
\end{example}

\emph{Non-uniformly fiber bunched cocycles:}
As mentioned before, the existence of uniform holonomies in a neighborhood of the reference cocycle is the main ingredient in the proof of Theorem \ref{thm:BBB2018}. So, it is important to understand more general conditions that guarantee the existence of holonomies in a stable way. A common approach in non-uniformly hyperbolic dynamics is to try to flexibilize a uniform concept such as $\gamma$-fiber bunched considering its asymptotic version: a cocycle $A$ is called \emph{non-uniformly} $\gamma$ \emph{fiber bunched} if ''it has small Lyapunov exponent", more precisely, if $A$ is $\gamma$-H\"older continuous and
\begin{align*}
    L(A) < \frac{\gamma}{2}\log\zeta.
\end{align*}
Recall that $\zeta>1$ comes from the metric on $\Sigma$. The notion of non-uniformly $\gamma$-fiber bunched cocycles was first explored in \cite{Vi2008} in which many of its properties were established (the term dominated was used instead of non-uniformly fiber bunched). We list a few of these properties in the next proposition.
\begin{proposition}
    It holds that,
    \begin{enumerate}
        \item If $A$ is $\gamma$-fiber bunched, then $A$ is non-uniformly $\gamma$-fiber bunched;

        \item Non-uniformly $\gamma$-fiber bunched is an open condition in $C^{\gamma}(\Sigma,\, \SL_2(\R))$;

        \item If $A$ is non-uniformly $\gamma$-fiber bunched, then $A$ admits (non-uniformly) stable and unstable holonomies.
    \end{enumerate}
\end{proposition}

\begin{example}[Non-uniformly fiber bunched cocycles]
\label{ex:non-uniformlyFiberBunched}
\normalfont
    Let $A_{\alpha,\beta}$, $1\leq \alpha\leq \beta$ be the cocycle given in the  Bocker-Viana Example \ref{ex:BockerVianaExample}. Recall that,
    \begin{align*}
        L(A_{\alpha,\beta}) = |p_1\,\log\beta - p_2\, \log\alpha|.
    \end{align*}
    So, $A_{\alpha, \beta}$ is non-uniformly $\gamma$-fiber bunched as long as $|p_1\,\log\beta - p_2\, \log\beta|< \frac{\gamma}{2}\, \log\zeta$. We could consider, for instance, a probability vector $p = (p_1,p_2)$ such that $p_1/p_2$ is very close to $\log\alpha/\log\beta$. So, this family provides examples of cocycles that are not fiber bunched (we could consider $\beta > \zeta^{\gamma/2}$), but still non-uniformly $\gamma$-fiber bunched.
\end{example}

The main goal in \cite{Vi2008}, which initiated the study of non-uniformly fiber bunched cocycles, was to obtain the next result.
\begin{theorem}[M. Viana, \cite{Vi2008}]
\label{thm:Vi2008}
    The set of cocycles $A\in C^{\gamma}(\Sigma,\, \SL_2(\R))$ with positive Lyapunov exponent is open and dense in $C^{\gamma}(\Sigma,\, \SL_2(\R))$.
\end{theorem}

The idea for proving this result is to start with a cocycle with zero Lyapunov exponent. Then it is automatically non-uniformly fiber bunched. With this condition non-uniform holonomies can be constructed in large measure sets. The difficult part is to construct this large measure set having product structure and periodic points. Once this is done using a version of the invariance principle on this set and a pinching and twisting condition the zero exponent can be destroyed.

Even though Theorem \ref{thm:Vi2008} guarantees that the set of zero Lyapunov exponent is rare inside of $C^{\gamma}(\Sigma,\, \SL_2(\R))$, we still could have density of the set of cocycle with small Lyapunov exponent outside of the uniformly hyperbolic in $C^{\gamma}(\Sigma,\, \SL_2(\R))$. That would be a similar phenomena to what happens in the case of $C^0$-topology and Theorem \ref{thm:BoMa}.

\begin{problem}[Viana]
    Are non-uniformly $\gamma$-fiber bunched cocycles dense  in $C^{\gamma}(\Sigma,\SL_2(\R))$ outside of the uniformly hyperbolic cocycles?
\end{problem}

The concept of non-uniform fiber bunched was further explored to extend the result of continuity of the Lyapunov exponent.
\begin{theorem}[C. Freijo, K. Marin, \cite{FrMa2021}]
\label{thm:FreijoMarin}
    Assume that $A$ is non-uniformly $\gamma$-fiber bunched admitting a family of uniform stable holonomies. Then $A$ is a continuity point of the Lyapunov exponent in the $C^{\gamma}$-topology. 
\end{theorem}
In the same spirit as in the Pesin's theory, non-uniformly fiber-bunched cocycles admits very large sets of the base space where the holonomies are indeed uniform. In this large probability sets, we may work as in the Theorem \ref{thm:BBB2018} since we have uniformity of the holonomies. The main issue is to have a control of the lack of uniformity in the the small probability sets to obtain the properties \ref{item:proofBBB1} and \ref{item:proofBBB2} above. This is handled by a careful analysis of the convergence of the disintegrations $m^{k,+}$ and $m^{k,-}$ in these large probability sets.

A nice consequence of the techniques obtained in \cite{FrMa2021} is the following: if $A$ is an irreducible  non-uniformly $\gamma$-fiber bunched cocycle, then $A$ is a continuity point of the Lyapunov exponent in the $C^{\gamma}$-topology. Indeed, assuming that $L(A)>0$ and so $A$ is strongly irreducible, by Example \ref{ex:pinchingTwistingLCC}, we see that $A$ is pinching and twisting. In particular, there is only one $u$-state, namely $m^u$. However, the following fact may be extended from the context of fiber bunched cocycles to the non-uniformly fiber bunched cocycles: if $A_k$ is a sequence converging to $A$ in the $C^{\gamma}$-topology and $m^k$ is a sequence of $u$-states for $A_k$ realizing $L(A_k)$ (equation \eqref{eq:040423.2}), then, up to a subsequence, the limit of $m^k$ is still a $u$-state for $A$. This is enough to conclude the continuity.

For cocycles which are not non-uniformly fiber-bunched(large Lyapunov exponent) the problem of understanding the continuity properties of the Lyapunov exponent is still open. However, a few examples of discontinuity have been explored in the literature.
\begin{example}[C. Butler, \cite{Bu2018}]
\normalfont
In the Bocker-Viana's Example \ref{ex:BockerVianaExample}, assume that $\alpha = \beta > 1$ and write $A_{\beta} := A_{\alpha, \beta}$. For the probability vector $p = (p_1, p_2)$, assume that $p_1\in (1/2, 1)$. Then, in \cite{Bu2018}, Butler showed that $A_{\beta}$ is a discontinuity of the Lyapunov exponent if $\beta^{p_1 - p_2}\geq \zeta^{\gamma/2}$. Recall, from Example \ref{ex:non-uniformlyFiberBunched}, that $A_{\beta}$ is non-uniformly fiber bunched if $\beta^{p_1 - p_2} < \zeta^{\gamma/2}$. So, as long as $A_{\beta}$ is not non-uniformly $\gamma$-fiber bunched this cocycle is a discontinuity point of the Lyapunov exponent. 
\end{example}
\begin{remark}
\normalfont
    The arguments used \cite{Bu2018} may also be used to guarantee that if $\beta> \zeta^{\gamma}$, then $A_{\beta}$ is a discontinuity point of the Lyapunov exponent in the $C^{\gamma}$-topology. This is a slight improvement of the range presented in Example \ref{ex:BockerVianaExample}. Recall that in this case, the $\gamma$-fiber bunched condition is equivalent to $\beta < \zeta^{\gamma/2}$.
\end{remark}
\begin{remark}
\normalfont
    Observe that the Theorem \ref{thm:FreijoMarin} (or the comment that follows it) does not cover the case of $A_{\beta}$ for any $\beta>1$, even if we assume the non-uniform fiber bunched condition ($A_{\beta}$ is reducible).  
\end{remark}
\begin{problem}
    Is $A_{\beta}$, defined above, a continuity point of the Lyapunov exponent in the $C^{\gamma}$-topology, for $\beta\in (\zeta^{\gamma/2},\,\zeta^{\gamma})$ and $\beta^{p_1 - p_2} < \zeta^{\gamma/2}$?
\end{problem}

\subsection{Modulus of continuity of LE for fiber bunched cocycles}
Now we come back to the world of (uniform) fiber-bunched cocycles to discuss the modulus of continuity of the Lyapunov exponent. The machine developed in \cite{DuKl2016} to obtain H\"older regularity of the Lyapunov through large deviation estimates may also be applied to the context of cocycles $A:\Sigma\to\SL_2(\R)$ depending on infinitely many coordinates. That is the content of the next theorem.
\begin{theorem}[P. Duarte, S. Klein, M. Poletti, \cite{DuKlPo2022}]
\label{thm:DuKlPo2022}
    Assume that the $\gamma$-fiber bunched cocycle $A$ satisfies the pinching and twisting condition. Then, there exist $\theta>0$ and a neighborhood $\mathcal{U}\subset C^{\gamma}(\Sigma,\, \SL_2(\R))$ of $A$ such that, restricted to this neighborhood, the Lyapunov exponent function $L$ is $\theta$-H\"older continuous.
\end{theorem}
In order to obtain large deviation estimates in this context we follow the same lines of the proof of Theorem \ref{thm:DuKl2016}. Below, we mention some of the required adaptations.

To start, we use the families of unstable holonomies, which are uniform in a neighborhood of $A$ where the fiber bunched condition holds, to reduce the analysis to maps depending only on negative coordinates (see Proposition \ref{prop:reducitionToOneSided}) i.e., it is enough to assume that $A(x) = A(x^-)$, for every $x\in \Sigma$, and guarantee that there exist constants $\delta, C, \kappa, \varepsilon_0 > 0$ such that for every cocycle $B\in C^{\gamma}(\Sigma, \SL_2(\R))$, $\norm{A-B}_{C^{\gamma}}< \delta$, with $B(x) = B(x^-)$, for every $x\in \Sigma$, and for every $\varepsilon\in (0,\varepsilon_0)$ we have that 
\begin{align*}
    \mu\left(\left\{
        x\in \Sigma\colon\,
        \left|
            \frac{1}{n}\log\norm{B^n(x)} - L(B)
        \right|>\varepsilon
    \right\}\right)\leq C\, e^{-\kappa\varepsilon^2\, n}.
\end{align*}
After this reduction, we proceed with the analysis of the spectral properties of the Markov operator, but differently of the Step 1 in Theorem \ref{thm:DuKl2016}, the cocycle $A$ may depend on infinitely many coordinates, so there is a small adaptation on the definition of the Markov operator for this system. Indeed, set $Q:C^0(\Sigma^-\times \bP^1)\to C^0(\Sigma^-\times\bP^1)$ given by
\begin{align*}
    Q_A(\varphi)(x^-, \hat v) := \sum_{i=1}^{\kappa}\, p_i\,\varphi((x^-,i), A(x^-)\, \hat v).
\end{align*}
It is not hard to see that the measure $m^u$ on $\Sigma^-\times\bP^1$ is a fixed point of the adjoint action $Q^*$ (this is the definition of stationary measure in this context). 

The spectral analysis of $Q_A$ is concluded once we notice that:  
First, defining
\begin{align*}
	K_{n,\theta}(A)
	&:=
	\sup_{x\in \Sigma}\, \sup_{\hat u\neq \hat v}\int_{\Sigma}\,\left( 
	\frac{
		d(A^n(x)\, \hat u,\, A^n(x)\, \hat v)
	}{d(\hat u, \hat v)}
	\right)^{\theta}\, d\mu(x) ,
\end{align*}	 
for any $\theta\in (0,1)$ there exists $C>0$ such that for every $\varphi\in C^{\theta}(\Sigma^-\times\bP^1)$,
\begin{align*}
    [Q^n_A(\varphi)]_{\theta}
    &\leq  K_{n,\theta}(A)\, [\varphi]_{\theta}   + C\,\norm{\varphi}_{\infty}. 
\end{align*}
Here, we are using the slightly different $\theta$-H\"older norm,
\begin{align*}
    [\varphi]_{\theta}
    := \sup_{x^-\in \Sigma^-}\, [\varphi(x^-,\cdot)]_{\theta} + \sup_{\hat u\in \bP^1}\, [\varphi(\cdot, \hat u)]_{\theta}.
\end{align*}
Second, equivalently to the Step 1 of Theorem \ref{thm:DuKl2016} (see equation \eqref{eq:220323.7} and the  discussion that follows), the positivity of $L(A)$ and the pinching and twisting assumption is invoked to guarantee the exponential decay of the quantities $K_{n,\theta}(A)$ for some $\theta_0\in (0,1)$. So, using these  considerations there exist constants $\lambda, C>0$ such that for every $\varphi\in C^{\theta_0}(\Sigma^-\times\bP^1)$ and every $n$ sufficiently large,
\begin{align*}
    [Q^n(\varphi)]_{\theta_0} \leq e^{-\lambda n}[\varphi]_{\theta_0} + C\, \norm{\varphi}_{\infty}.
\end{align*}
These ingredients suffice to guarantee the (uniform in a neighborhood of $A$) spectral gap of the operator $Q_A$. To conclude the regularity of the Lyapunov exponent, we follow the rest of the steps mentioned in the proof of Theorem \ref{thm:DuKl2016}.

\begin{remark}
\normalfont
    Similar large deviation principle and limit theorems for this class of cocycles were obtained by Park and Piraino \cite{ParkPir22}.
\end{remark}

It is natural to wonder what happens if we no longer assume the pinching and twisting condition. 
\begin{problem}
    What is the modulus of continuity of the Lyapunov exponent near a fiber bunched diagonal cocycle cocycle (see Example \ref{ex:BockerVianaExample})?
\end{problem}
The solution of this problem would give an extension of the Theorem \ref{thm:DuKl2020} from the locally constant cocycle to a the fiber bunched context.

The precise modulus of continuity of the Lyapunov exponent is still far from being established. However, we could consider a generalization of Theorem \ref{thm:BeDuCaFrKl2022} to the context of $\gamma$-fiber bunched cocycles. To better pose the problem, we adapt the concept of heteroclinic tangency introduced in Section \ref{subsection:sharpModulusContinuity}.

We say that a cocycle $A\in C^{\gamma}(\Sigma,\, \SL_2(\R))$ admits a \emph{heteroclinic tangency} if there exist periodic points $q_1, q_2\in \Sigma$, of period $m_1,\, m_2$ respectively, and a heteroclinic point $z\in W^u_{\text{loc}}(q_1)\cap \sigma^{-k}(W^s_{\text{loc}}(q_2))$ such that
\begin{enumerate}
    \item[(i)] $A^{m_i}(q_i)$ is hyperbolic for $i=1,2$;

    \item[(ii)] For any $i=1,2$, let $\hat v_+(q_i), \hat v_-(q_i)\in \bP^1$ be the eigen-directions $A^{m_i}(q_i)$ associated respectively with the largest and smallest (in absolute value) eigenvalues. Then,
    \begin{align*}
        H^s_{z, q_2}\, A^k(z)\, H^u_{q_1,z}\, \hat v_+(q_1) = \hat v_-(q_2).
    \end{align*}
\end{enumerate}

\begin{problem}
    Let $A\in C^{\gamma}(\Sigma,\, \SL_2(\R))$ be a pinching and twisting $\gamma$-fiber bunched cocycle admitting a heteroclinic tangency. In that context, by Theorem \ref{thm:DuKlPo2022}, we know that $L$ is $\theta$-H\"older continuous in a neighborhood of $A$ and some $\theta>0$. Does that imply that $\theta < H(\mu)/L(A)$?
\end{problem}

\section{Conclusion and general models}
Positivity and continuity of the Lyapunov exponent in the classical setting of random products of matrices is well understood. The main elements in their study are the stationary measures and the structure of the group of matrices generated by the cocycle. Most of the techniques used in the locally constant case may be adapted, using holonomies, to deal with  H\"older cocycles.

The strategy to study these cocycles uses holonomies to reduce the analysis, up to conjugacy, to cocycles only depending on the positive coordinates.  These can be analyzed as cocycles over one-sided dynamics that, with proper adaptations,  share many properties with the locally constant cocycles. The similarity is expressed  for instance  in the properties of invariant measures by the action of the cocycle in the one-sided product space. Indeed,  $u$-states have in the general setting the central role
played by	the stationary measures in the locally constant case. These measures have a nice regularity which allowed us to recover many of the properties presented earlier for the stationary measures. However, for general cocycles, beyond the locally constant ones, the lack of independence in the matrix products along the orbit is an important issue  which requires suitable adaptations. Concepts such as strong irreducibility or  non compactness of the group generated by the  matrices of the cocycle are reformulated in terms of periodic points  and $su$-loops. These mentioned adaptations are only possible for cocycles admitting families of stable and unstable linear holonomies. Not much is known  for more general cocycles.

Most of the results mentioned in the previous section can be stated in some more general context. We could consider on the base measures that are not necessarily Bernoulli, but have some ``good" product structure which include Markov shifts. Nevertheless, 
because we did not want to focus on the structure of the measure but instead on the important properties of the fiber action, we preferred to state the results for Bernoulli measures which automatically satisfy all the hypothesis in  the previous theorems.

Now to finish we recall some recent results and open problems on more general base dynamics with some hyperbolicity.

\subsection{More general dynamics}

Recall that, by the classical construction of Markov partitions for Anosov maps, see for example \cite{Bow75}, Anosov diffeomorphisms can be semi-conjugated to Markov shifts. Via this conjugacy the previous results can be extended to linear cocycles over Anosov diffeomorphisms.

\begin{center}
    \emph{So, what happens beyond uniform hyperbolic base maps?}
\end{center}

Beyond uniform hyperbolicity, the two main classes of dynamics with some hyperbolicity that have been more studied are the \emph{partially hyperbolic} and the \emph{non-uniformly hyperbolic diffeomorphisms}.

There is some advance in the study of continuity and positivity of the exponents for cocycles over these maps.

For non-uniformly hyperbolic base dynamics with measures with product structure we have the following result: 
\begin{itemize}
    \item  There exists an open and dense set of H\"older cocycles with positive Lyapunov exponent, see \cite{Vi2008}.
\end{itemize}
\begin{problem}
    Is the Lyapunov exponent continuous with in the space of H\"older fiber bunched cocycles over non uniformly hyperbolic maps?
\end{problem}
In this case, fiber bunched means that the  inequality \eqref{eq:110423.1} is satisfied with $e^{\lambda}$ instead of $\zeta$, where $\lambda$ is  the minimum of the  Lyapunov exponents of the  base map's derivative.

By results in \cite{Sa13} and \cite{Ov19}, non-uniformly hyperbolic base dynamics may be codified using a symbolic model such as the one presented in these notes, see \cite{BPV19} where this idea was used. The drawback is that such symbolic dynamics now could have countably infinitely many symbols. In this case, the problem of continuity of the Lyapunov exponent for more general non-uniformly hyperbolic maps may be approached  from the symbolic point of view but now over non-compact spaces, see  \cite{ViSan}.

Below, we mention a few results in the context where the base dynamics is a partially hyperbolic map.
\begin{itemize}
    \item Existence of an open and dense set of H\"older cocycles with positive Lyapunov exponent. For volume preserving accessible maps, see \cite{ASV13}. For skew products, see \cite{Pol18};
    
    \item Criterion for continuity of the Lyapunov spectrum for skew products, see \cite{PoV18};
    
    \item Openness of the set of fiber bunched H\"older cocycles with positive Lyapunov exponent. For volume preserving accessible maps, see \cite{BPS18}.
\end{itemize}
Fiber bunched for partially hyperbolic maps means that the inequality \eqref{eq:fiberBunchedInequality} is satisfied with $\zeta$ equal to the minimum of the expansion of the diffeomorphism along the unstable direction and its inverse along the stable.

For partially hyperbolic maps, there are examples of H\"older cocycles with positive exponent accumulated by cocycles with zero exponent, even in the fiber bunched case, see~\cite{BPS18}. These examples are constructed over a direct product of an Anosov with a rotation, in particular they are not accessible. For volume preserving  accessible diffeomorphisms, this type of discontinuity can not happen, as stated in \cite{BPS18}. 
\begin{problem}
    Assume that the base dynamics is a partially hyperbolic volume preserving accessible diffeomorphism. Does the Lyapunov exponent varies continuously in the space of H\"older fiber bunched cocycles.
\end{problem}

Let us introduce now a class
of partially hyperbolic systems that we call \emph{mixed models}. These are maps of the form
\begin{align*}
	f:\Sigma \times S^1\to \Sigma\times S^1,
	\quad
	f(x,t)=(\sigma(x),t+\theta_x), 
\end{align*} 
where $S^1$ is the unit circle and
$\theta:\Sigma\to S^1$ is a continuous function. These maps preserve the measure  $\mu\times dt$, where $dt$ is the Lebesgue measure on $S^1$. In the same way that the full Bernoulli shift  $(\Sigma, \sigma, \mu)$ provides a basic model for hyperbolic systems, mixed models  are a basic class of partially hyperbolic systems with compact fibers where the action is elliptic.

\begin{remark}
\normalfont
    Linear cocycles over torus rotations, known as quasi periodic cocycles, have been extensively studied.    
    The lack of hyperbolicity in the base dynamics requires completely different techniques from the ones mentioned in these notes. For results regarding positivity and continuity of the Lyapunov exponent in this context, see \cite{DuKl2016}, \cite{DuKl2017} (for the Schr\"odinger context \cite{Da2017}) and references therein.
\end{remark}

An interesting class of linear cocycles over mixed models is the so called  \emph{mixed random quasi-periodic cocycles}, which corresponds to random products of quasi periodic cocycles. Consider a finite measure $\sum_{i=1}^{\kappa}\, p_i\delta_{(\theta_i,\, A_i)}$, where $(\theta_i,\, A_i)$ defines a quasi-periodic cocycle. Analogously to the definition of random cocycles, we consider a   ``locally constant" mixed cocycle, defined  by $A(x,t) := A_{x_0}(t)$, over the mixed  base dynamics, $f(x,t) = (\sigma(x),\, t+\theta_{x_0})$. 

Recently, this type of model has been attracted a lot of attention. We list a few results in this regard.
\begin{itemize}
    \item Existence of open and dense set of quasi-periodic cocycles $(A_1,\ldots, A_{\kappa})$ in which the Lyapunov exponent function is continuous and positive; see \cite{BePo21}.
    \item A criterion for positivity of the exponents, see \cite{CaiDK}.
    \item Uniform upper semi-continuity of the exponents \cite{CaiDK2022}.

    \item Analiticity of the Lyapunov exponent with respect to the probability vector \cite{BeSaTa}.
\end{itemize}

In the space of cocycles $A:\Sigma\times S^1\to \SL_2(\R)$ depending only on $x_0$ and smoothly on $t\in S^1$, there are examples of discontinuity of the Lyapunov exponent function, see \cite[section 5.2]{BPS18}. Here, the discontinuity comes from the discontinuity of smooth quasi periodic cocycles, see \cite{WaY13}.

It is known that for quasi periodic analytic cocycles, the exponent varies continuously, see \cite{AVJitSad2014}. So, we have the following question:
\begin{problem}
    Is the Lyapunov exponent a continuous function in the space of analytic cocycles (with respect to the circle coordinate) among the mixed random quasi periodic cocycles?
\end{problem}
As mentioned in \cite{BePo21}, there are open and dense conditions to guarantee continuity of the exponents (a generalization of the pinching and twisting condition). So, it is natural to ask the following:
\begin{problem}
    In the set of pinching and twisting cocycles (see the definition \cite{BePo21}) of the mixed random quasi periodic cocycles, does the exponent vary H\"older continuously?
\end{problem}

As we can see, there are many open questions to be explored about the behavior of the Lyapunov exponent function, especially beyond some restricted classes such as the random locally constant cocycles, the fiber bunched cocycles over hyperbolic maps and the quasi-periodic cocycles.

\subsection*{Acknowledges:} The authors would like to express their gratitude to P. Duarte for the invitation to write this survey and for his valuable suggestions that helped to improve the quality of text. The authors would like to thank the organizers of the workshop ``New trends in Lyapunov exponents", where discussions on the elaboration of this material began. The authors would also like to thank Graccyela Salcedo by the careful reading and suggestions on the final version of the text.  The gratitude of the authors also extends to Łukasz Rzepnicki for providing Picture \ref{pic:bernoulliConvolutions}.

Both authors were supported by FCT-Funda\c{c}\~{a}o para a Ci\^{e}ncia e a Tecnologia through the project  PTDC/MAT-PUR/29126/2017. J. B. was supported by the Center of Excellence ``Dynamics, Mathematical Analysis and Artificial Intelligence" at Nicolaus Copernicus University in Toruń. M.P. was  supported by CNPQ, Instituto Serrapilheira, grant ``Jangada Din\^amica: Impulsionando Sistemas Din\^amicos na Regi\~ao Nordeste" and the Coordena\c{c}\~ao de Aperfei\c{c}oamento de Pessoal de N\'ivel Superior - Brasil (CAPES) - Finance Code 001.

\bibliographystyle{alpha}
\bibliography{bib.bib}

\information

\end{document}